\documentclass[twoside]{amsart}
\usepackage{a4}

\usepackage{amsmath}

\usepackage{amscd}

\usepackage{amsfonts}

\usepackage{amssymb}  

\usepackage{amsthm}

\usepackage{array}

\usepackage[english]{babel}
\usepackage{csquotes}

\usepackage{bbm}

\usepackage{blkarray}

\usepackage{bmpsize}

\usepackage{booktabs}

\usepackage[font=footnotesize]{caption}

\usepackage{color}

\usepackage{enumerate}

\usepackage{extarrows}

\usepackage{enumitem}		

\usepackage{graphicx}

\usepackage[colorlinks=true]{hyperref}	
\usepackage[capitalise]{cleveref} 

\usepackage{indentfirst}	

\usepackage[modulo,right,pagewise]{lineno} 

\usepackage{mathrsfs}

\usepackage{mathdots}

\usepackage{mathtools}		
\usepackage{mathabx}		

\usepackage{multirow}

\usepackage{siunitx}

\usepackage{tabularx}

\usepackage{thmtools}

\usepackage{tikz}
\usetikzlibrary{fadings}
\usetikzlibrary{patterns}
\usetikzlibrary{shadows.blur}
\usetikzlibrary{shapes}

\usepackage{tikz-cd}

\usepackage{verbatim}

\usepackage[backend = biber,sortcites = true,sorting=nty,giveninits=true,
minbibnames=5,maxbibnames=5,doi = false,isbn = false]
{biblatex}
\renewbibmacro{in:}{%
  \ifentrytype{article}{}{\printtext{\bibstring{in}\intitlepunct}}}

\DeclareFieldFormat[article]{pages}{#1}
\DeclareFieldFormat[article]{volume}{\textbf{#1}}
\DeclareFieldFormat[article]{volume}{\textbf{#1}}
\DeclareFieldFormat[article]{number}{(#1)}
\DeclareFieldFormat[article]{year}{#1}
\DeclareFieldFormat[article]{url}{}
\DeclareFieldFormat[book]{url}{}

\renewbibmacro{volume+number+eid}{%
    \printfield{volume}%
    \printfield{number}}

\newtheorem{theorem}{Theorem}[section]
\newtheorem{proposition}[theorem]{Proposition}
\newtheorem{corollary}[theorem]{Corollary}
\newtheorem{lemma}[theorem]{Lemma}

\theoremstyle{definition}

\newtheorem{conjecture}[theorem]{Conjecture}

\newtheorem{definition}[theorem]{Definition}

\newtheorem{example}[theorem]{Example}

\newtheorem{question}[theorem]{Question}
\newtheorem{remark}[theorem]{Remark}

\numberwithin{equation}{section}
\numberwithin{theorem}{section}

\newcommand{\wtilde}[1]{\widetilde{#1}}

\newcommand{\msf}[1]{\mathsf{#1}}

\newcommand{\C}{\mathbb{C}}

\newcommand{\bbH}{\mathbb{H}}
\newcommand{\Hy}{\mathbb{H}}

\newcommand{\bbP}{\mathbb{P}}

\newcommand{\R}{\mathbb{R}}
\newcommand{\bbS}{\mathbb{S}}

\newcommand{\bbW}{\mathbb{W}}

\newcommand{\Z}{\mathbb{Z}}

\newcommand{\calA}{\mathcal{A}}

\newcommand{\calB}{\mathcal{B}}

\newcommand{\calC}{\mathcal{C}}

\newcommand{\calD}{\mathcal{D}}

\newcommand{\calF}{\mathcal{F}}

\newcommand{\calG}{\mathcal{G}}

\newcommand{\calI}{\mathcal{I}}

\newcommand{\calJ}{\mathcal{J}}

\newcommand{\calL}{\mathcal{L}}

\newcommand{\calP}{\mathcal{P}}

\newcommand{\calS}{\mathcal{S}}
\newcommand{\calT}{\mathcal{T}}

\newcommand{\calU}{\mathcal{U}}

\newcommand{\calW}{\mathcal{W}}
\newcommand{\calX}{\mathcal{X}}

\newcommand{\calY}{\mathcal{Y}}

\newcommand{\calQC}{\mathcal{QC}}

\newcommand{\scrC}{\mathscr{C}}
\newcommand{\scrD}{\mathscr{D}}

\newcommand{\scrF}{\mathscr{F}}

\newcommand{\scrP}{\mathscr{P}}
\newcommand{\scrQ}{\mathscr{Q}}

\newcommand{\scrT}{\mathscr{T}}

\newcommand{\scrV}{\mathscr{V}}

\newcommand{\frakg}{\mathfrak{g}}

\newcommand{\frakh}{\mathfrak{h}}

\newcommand{\frakn}{\mathfrak{n}}

\newcommand{\al}{\alpha}

\newcommand{\G}{\Gamma}
\newcommand{\Gam}{\Gamma}
\newcommand{\gam}{\gamma}

\newcommand{\Del}{\Delta}

\newcommand{\del}{\delta}
\newcommand{\ep}{\epsilon}
\renewcommand{\epsilon}{\varepsilon}

\newcommand{\thet}{\theta}

\newcommand{\Lam}{\Lambda}

\newcommand{\lam}{\lambda}

\newcommand{\Sig}{\Sigma}

\newcommand{\sig}{\sigma}

\newcommand{\Om}{\Omega}

\newcommand{\om}{\omega}

\DeclareMathOperator{\PSL}{PSL}

\DeclareMathOperator{\SO}{SO}

\DeclareMathOperator{\PSp}{PSp}

\DeclareMathOperator{\Id}{Id}

\DeclareMathOperator{\Isom}{Isom}

\DeclareMathOperator{\Out}{Out}

\DeclareMathOperator{\Dil}{Dil}

\DeclareMathOperator{\supp}{supp}

\DeclareMathOperator{\CAT}{CAT}

\DeclareMathOperator{\Curr}{\calC urr}

\DeclareMathOperator{\Int}{Int}
\DeclareMathOperator{\Hull}{Hull}
\DeclareMathOperator{\proj}{proj}
\DeclareMathOperator{\Gr}{Gr}

\DeclareMathOperator{\Area}{Area}
\DeclareMathOperator{\area}{area}
\DeclareMathOperator{\Jac}{Jac}
\DeclareMathOperator{\diam}{diam}
\DeclareMathOperator{\Comm}{Comm}

\newcommand{\genby}[1]{\langle #1\rangle}

\def\({\left(}
\def\){\right)}

\newcommand{\bs}{\backslash}

\newcommand{\ra}{\rightarrow}
\newcommand{\loopra}{\looparrowright}

\newcommand{\ov}[1]{\overline{#1}}
\newcommand{\inn}[2]{\langle #1,#2\rangle}

\def\1{{\bf 1}}

\title{Approximating hyperbolic lattices by cubulations}

\author{\small{Nic Brody and Eduardo Reyes}}

\begin{document}
\maketitle

\begin{abstract}
We show that an isometric action of a torsion-free uniform lattice $\G$ on hyperbolic space $\Hy^n$ can be metrically approximated by geometric actions of $\G$ on $\CAT(0)$ cube complexes, provided that either $n$ is at most three, or the lattice is arithmetic of simplest type.
This solves a conjecture of Futer and Wise.

Our main tool is the study of a space of co-geodesic currents, consisting of invariant Radon measures supported on codimension-1 hyperspheres in the Gromov boundary of $\Hy^n$. By pairing co-geodesic currents and geodesic currents via an intersection number, we show that asymptotic convergence of geometric actions can be deduced from the convergence of their dual co-geodesic currents. 

For surface groups, our methods also imply approximation by cubulations for actions induced by non-positively curved Riemannian surfaces with singularities, Hitchin and maximal representations, and quasiFuchsian representations. 

\end{abstract}

\section{Introduction}
Geometric group theory studies groups through their isometric actions on metric spaces. These actions may originate from geometric topology, combinatorial group theory, dynamics, or beyond, and different actions might reveal different properties of the group of interest.

For example, if $\G$ is the fundamental group of a Riemannian manifold $(M,\frakg)$, then the universal cover $\wtilde{M}$  admits an isometric action of $\G$ by deck transformations, when endowed with the pullback metric. Curvature bounds on $\frakg$ have strong consequences on $\G$. For instance, if $(M,\frakg)$ is closed and negatively curved, then  $\G$ is word-hyperbolic. Of special interest is the case in which $(M,\frakg)$ has constant negative curvature, since then $\G$ acts as a (torsion-free) uniform lattice on the real hyperbolic space $\Hy^n$.
 

There is another interesting class of isometric actions, which is that of cubical actions on $\CAT(0)$ cube complexes, when considering the combinatorial metric on their 1-skeletons. By the work of Sageev in the 1990s \cite{sageev,sageev.cod}, we know that these actions are related to the existence and behavior of codimension-1 subgroups. Nowadays, it is known that many groups admit geometric actions on such complexes (and that many groups do not), see, for instance, \cite{wise.qch,bergeron-wise,martin-steenbock,DKM,futer-wise,niblo-roller,delzant-py}.

One would like to understand how various actions of a given group $\G$ might be related quantitatively, at least from a large-scale perspective. By the Milnor-Schwarz lemma, any two geometric actions of $\G$ on geodesic metric spaces are equivariantly quasi-isometric. In this work, we will consider $\lam$-quasi-isometries, nomenclature which we adopt from \cite{futer-wise}.  Given two geometric actions of a group $\G$ on (roughly) geodesic metric spaces $X$ and $Y$, a \emph{$\lam$-quasi-isometry} $(\lam\geq 1)$ is a $\G$-equivariant map $f:X \ra Y$ satisfying 
\begin{equation}\label{eq.lam-qi}
    \lam_1^{-1}d_{X}(x,y)-A\leq d_{Y}(fx,fy)\leq \lam_2d_{X}(x,y)+A
\end{equation}
for all $x,y\in X$, where $\lam_1,\lam_2,A\geq 0$ satisfy $\lam_1\lam_2=\lam$ (note that $f$ is coarsely surjective by cocompactness of the actions).

The quantity $\lam$ reflects how far the actions on $X$ and $Y$ are from being homothetic up to a uniformly bounded additive perturbation. If $\G$ is a hyperbolic group, then \eqref{eq.lam-qi} holds for some $A$ if and only if for all $g\in \G$ we have
\begin{equation}\label{eq.lam-qi.mls}
    \lam_1^{-1}\ell_{X}[g]\leq \ell_{Y}[g]\leq \lam_2\ell_{X}[g],
\end{equation}
where
\[\ell_{X}[g]=\lim_{m \to \infty}{\frac{d_{X}(g^m x,x)}{m}} \ \ \text{ and } \ \ \ell_{Y}[g]=\lim_{m \to \infty}{\frac{d_{Y}(g^m y,y)}{m}}\]
are the \emph{stable translation length functions} associated to these actions, which are independent of the base points $x\in X,y\in Y$, see e.g. \cite{krat,cantrell-tanaka}. 

Suppose $M$ is a closed $n$-dimensional hyperbolic manifold with fundamental group $\G$. The goal of this paper is to compare the action of $\G$ on $\wtilde{M}=\Hy^n$ with geometric actions of $\G$ on $\CAT(0)$ cube complexes (with the combinatorial metric), and is motivated by the following conjecture of Futer and Wise \cite[Conjecture~7.8]{futer-wise}.

\begin{conjecture}[Futer-Wise]\label{conj:futer.wise}
Let $M=\Gam \bs \Hy^n$ be a closed hyperbolic $n$-manifold with fundamental group $\G$, so that either $n\leq 3$ or $M$ is arithmetic of simplest type. Then for every $\lam>1$, $M$ is homotopy equivalent to a compact non-positively curved cube complex $\calX$ such that there is a $\G$-equivariant $\lam$-quasi-isometry from $\wtilde{\calX}$ to $\wtilde{M}$.
\end{conjecture}

This conjecture cannot be strengthened to ask for the existence of 1-quasi-isometries. Indeed, every automorphism $g$ of the $\CAT(0)$ cube complex $\wtilde{\calX}$ either fixes a point or has an invariant (combinatorial) geodesic on its first cubical barycentric subdivision \cite{haglund}, and hence $2\ell_{\wtilde{\calX}}[g]\in\Z$ for any $g\in \G$ and any $\G$ acting geometrically on $\wtilde{\calX}$. On the other hand, it is known that the marked length spectrum $\{\ell_{\Hy^n}[g]: g\in \G\}$ of any uniform lattice action of $\G$ on $\Hy^n$ is not contained in a discrete subgroup of $\R$ (for instance, this follows from \cite[Prop.~3.2]{GMM}). Therefore  \eqref{eq.lam-qi.mls} cannot hold with $\lam_1\lam_2=1$.

\cref{conj:futer.wise} is meaningful since any $\G$ as in its statement admits many geometric actions on $\CAT(0)$ cube complexes; see \cite{bergeron-wise,BHW}. The conjecture predicts that there are ``enough'' cubulations to approximate geometric actions on $\Hy^n$. Now we explain the main results of this work, which imply \cref{conj:futer.wise}.




\subsection{Approximating negatively curved metrics on surfaces}

Suppose $M$ is a closed orientable surface of negative Euler characteristic and fundamental group $\G$, and equip $M$ with a negatively curved Riemannian metric $\frakg_0$. Bonahon \cite{bonahon} introduced the space $\Curr(\G)$ of \emph{geodesic currents} on $\G$, which can be interpreted as $\G$-invariant locally invariant measures on the space of geodesics in the universal cover of $(M,\frakg_0)$ with the pullback metric. Looking at the endpoints at infinity of these geodesics, we see that the space $\Curr(\G)$ is independent of the metric $\frakg_0$. Bonahon also introduced the \emph{intersection number}, which is a continuous, $\R_+$-bilinear function 
\[i:\Curr(\G) \times \Curr(\G) \ra \R\]
that extends the geometric intersection number of closed geodesics in $M$. Any negatively curved Riemannian metric $\frakg$ on $M$ has associated a \emph{Liouville current} $\al_{\frakg}$, characterized by the following property: if $g\in \G$ is non-torsion and $\eta_{[g]}$ is the corresponding rational current, then 
$i(\al_{\frakg},\eta_{[g]})=\ell_{\frakg}[g]$, 
where $\ell_{\frakg}[g]$ is the length of the unique geodesic in $(M,\frakg)$ in the free homotopy class of $[g]$. See \cref{sec.preliminaries} for the definitions involved.

Bonahon proved that any current is a limit of real multiples of rational currents in the weak-$\ast$ topology, and in particular there exists a sequence $(t_m\eta_{[g_m]})_m$ of real multiples of rational currents converging to $\al_{\frakg}$. Then $\eta_{[g_m]}$ is filling for $m$ large enough, and by Sageev's construction there exists a geometric action of $\G$ on a $\CAT(0)$ cube complex $\wtilde{\calX}_m$ with hyperplanes stabilized by the conjugates of subgroups commensurable to $\langle g_m \rangle$. Continuity of the intersection number then allows us to promote the convergence $t_m\eta_{[g_m]} \xrightharpoonup{\ast} \al_{\frakg}$ to the existence of $\G$-equivariant $\lam_m$-quasi-isometries from $\wtilde{\calX}_m$ to $\wtilde{M}$ with $\lam_m$ converging to 1, see \cref{prop.contcogeodesic}. 

In \cref{sec.2-dim}, we will formalize this argument and prove the following, which confirms \cref{conj:futer.wise} when $n=2$.

\begin{theorem}\label{thm.2-dimensions}
Let $M$ be a closed, smooth Riemannian surface of negative curvature and fundamental group $\G$. Then for any $\lam>1$, there exists a geometric action of $\G$ on a $\CAT(0)$ cube complex $\wtilde{\calX}$ and a $\G$-equivariant $\lam$-quasi-isometry from $\wtilde{\calX}$ to $\wtilde{M}$.
\end{theorem}

By the same argument, many other geometric structures on surface groups can be approximated by cubulations, including non-positively curved Riemannian metrics on surfaces with finitely many singularities, and positively ratioed Ansosov representations (which include Hitchin representations and maximal representations). See \cref{sec.2-dim} for the precise results.


\subsection{Approximating 3-dimensional and arithmetic hyperbolic lattices}
The strategy sketched above for the proof of the 2-dimensional case of \cref{conj:futer.wise} is the template we follow to solve the higher dimensional cases. If $\G$ is our hyperbolic group of interest, then the goal is to construct a reasonable topological space $\calS\Curr(\G)$ of \emph{co-geodesic currents} and a bilinear \emph{intersection number} $i_\calS:\calS\Curr(\G) \times \Curr(\G)\ra \R$. For this space we require the following.

\begin{itemize}
      \item[i)] If we want to approximate the geometric action of $\G$ on the metric space $X$, we need a co-geodesic current $\al_X\in \calS\Curr(\G)$ dual to $X$, so that it verifies
   \begin{equation*}
        \ell_{X}[g]=i_\calS(\al_{X},\eta_{[g]})
    \end{equation*}
    for all $g\in \G$.
    \item[ii)] If $\al\in \calS\Curr(\G)$ is a ``discrete'' co-geodesic current, we want a cocompact action of $\G$ on a $\CAT(0)$ cube complex $\wtilde\calX$ that is ``dual'' to $\al$ in the sense that
    \begin{equation*}
        \ell_{\wtilde{\calX}}[g]=i_\calS(\al,\eta_{[g]})
    \end{equation*}
for any $g\in \G$. 

   \item[iii)] We want a sequence $(t_m\al_m)_m\subset \calS\Curr(\G)$ of (real multiples of) discrete co-geodesic currents dual to geometric actions on $\CAT(0)$ cube complexes $\wtilde\calX_m$ and converging to $\al_X$ in $\calS\Curr(\G)$. 
   \item[iv)] We want the intersection number $i_\calS$ to be continuous at $(\al_X,\eta)$ for any $\eta\in \Curr(\G)$.
\end{itemize}
If the pair $(\calS\Curr(\G),i_\calS)$ satisfies these conditions, then \cref{prop.contcogeodesic} allows us to promote the convergence $\al_m \to \al_X$ in $\calS\Curr(\G)$ to a sequence of $\G$-equivariant $\lam_m$-quasi-isometries $\wtilde{\calX}_m \to X$ with $\lam_m \to 1$.

For $\G$ a uniform lattice on the real hyperbolic space $\bbH^n$ ($n\geq 3)$, the space $\calS\Curr(\G)$ will consist of locally finite $\G$-invariant measures on the space of quasiconformal codimension-1 hyperspheres in the Gromov boundary $ \partial \bbH^n =\bbS^{n-1}$. When $n=3$, in  \cite{labourie} these measures are referred to as conformal currents. The intersection number is defined in the expected way, using the fact that complements of quasiconformal codimension-1 hyperspheres have exactly two connected components. The action on $\bbH^n$ is represented by the homogeneous (Haar) measure supported on round codimension-1 hyperspheres.

In the case of an arithmetic manifold of simplest type, $M=\G \bs \bbH^n$ contains infinitely many immersed totally geodesic codimension-1 hypersurfaces. The lifts of these hypersurfaces in $\Hy^n$ were used in \cite{BHW} to produce geometric actions of $\G$ on $\CAT(0)$ cube complexes. For each such immersed hypersurface in $M$, the limit sets in $\bbS^{n-1}$ of all its lifts in $\Hy^n$ determine a co-geodesic current dual to a (non-necessarily proper) action on a $\CAT(0)$ cube complex, as in \cite{BHW}. By considering an appropriate sequence of immersed hypersurfaces, Ratner's measure classification theorem \cite{ratner} and the ergodicity theorem of Mozes-Shah \cite{mozes-shah} imply that (after normalizing), the sequence of corresponding currents converge to the Haar co-geodesic current. With this data as our input, we deduce the arithmetic case of \cref{conj:futer.wise}.

\begin{theorem}\label{thm.main.arithmetic}
    Let $M=\Gam \bs \Hy^n$ be a closed hyperbolic $n$-manifold that is arithmetic of simplest type. Then for any $\lam>1$ there exists a geometric action of $\G$ on a $\CAT(0)$ cube complex $\wtilde{\calX}$ and a $\G$-equivariant $\lam$-quasi-isometry from $\wtilde{\calX}$ to $\wtilde{M}=\Hy^n$.
\end{theorem}

Incidentally, our approach using \cref{prop.criterioncontinuous} gives an alternative way of certifying properness of cubulations. Thus, we provide another proof of cubulability of uniform arithmetic lattices of simplest type.

The 3-dimensional case is, in some sense, the most interesting one, since in general we do not expect the manifold $M=\G \bs \Hy^3$ to contain infinitely many totally geodesic immersed surfaces \cite[Sec.~5.3]{maclachlan-reid}. Indeed, if this is the case then $M$ is arithmetic \cite{BFMS}. However, closed hyperbolic 3-manifolds have plenty of \emph{almost} totally geodesic immersed surfaces, as established by Kahn and Markovic in their proof of the surface subgroup conjecture \cite{kahn-markovic}. 

At the level of measures on the 2-plane Grassmannian bundle $\Gr_2(M)$, Labourie \cite[Thm.~5.7]{labourie} proved that the Haar measure is the limit of a sequence of normalized area measures induced by (possibly disconnected) asymptotically Fuchsian minimal surfaces. Lowe and Neves \cite[Prop.~6.1]{lowe-neves} proved that we can choose a sequence of connected minimal surfaces. A recent result of Al Assal \cite[Thm.~1.1]{alassal} describes all the possible limits of sequences of such measures. We upgrade these convergences to convergence of the associated rational co-geodesic currents. Applying this to the Haar co-geodesic current, we settle the remaining case of \cref{conj:futer.wise}.

\begin{theorem}\label{thm.main.3manifold}
    Let $M=\Gam \bs \Hy^3$ be a closed hyperbolic 3-manifold. Then for any $\lam>1$ there exists a geometric action of $\G$ on a $\CAT(0)$ cube complex $\wtilde{\calX}$ and a $\G$-equivariant $\lam$-quasi-isometry from $\wtilde{\calX}$ to $\wtilde{M}=\Hy^3$.
\end{theorem}

\begin{remark}\label{rmk.oneorbithyperplane}
    In all three theorems that imply \cref{conj:futer.wise}, we can choose the approximating cubulation $\wtilde{\calX}$ to have a single orbit of hyperplanes. See \cref{prop.generalsurfaces} and Remarks \ref{rmk.singleorbit3} and \ref{rmk.singleorbitarith}. This complements a result of Fioravanti and Hagen \cite{fioravanti-hagen}, who proved that any cubulable hyperbolic group admits such a cubulation.
\end{remark}

We deduce an interesting consequence of \cref{thm.main.3manifold}, concerning convex-cocompact representations. A remarkable result of Brooks \cite{brooks} asserts that if $\pi:\G \ra \PSL(2,\C)$ is a \emph{convex-cocompact} representation (i.e. a representation inducing a quasi-isometric embedding of $\G$ into $\Hy^3$) then there exist arbitrarily small perturbations $\pi_\ep$ of $\pi$, such that $\pi_\ep$ can be extended to a proper and cocompact representation $\hat\pi_\ep:\hat \G \ra \PSL(2,\C)$ so that $\G<\hat\G$ is a quasiconvex subgroup. 

If we apply \cref{thm.main.3manifold} to the action of $\hat \G$ on $\Hy^3$, then $\G$ is convex-cocompact for any geometric action of $\hat\G$ on a $\CAT(0)$ cube complex \cite[Prop.~13.7]{haglund-wise.special}. Since the restriction of a $\lam$-quasi-isometry is also a $\lam$-quasi-isometry (onto its image), we deduce the following approximation result for the representation $\pi$. 

\begin{proposition}\label{prop.main.quasifuchsian}
    Let $\pi:\G \ra \PSL(2,\C)$ be a convex-cocompact representation of the torsion-free group $\G$, and let $\Hull(\Lam_{\pi(\G)})\subset \Hy^3$ be the convex hull of the limit set of $\pi(\G)$. Then for any $\lam>1$, there exists a geometric action of $\G$ on a $\CAT(0)$ cube complex $\wtilde{\calX}$ and a $\G$-equivariant $\lam$-quasi-isometry from $\wtilde{\calX}$ to $\Hull(\Lam_{\pi(\G)})$.
\end{proposition}


\subsection{Interpretation in terms of metric structures}

The results already mentioned can be described in terms of metric structures, introduced by Furman in \cite{furman} and studied by the second author in \cite{oregon-reyes.ms}. If $\G$ is a non-elementary hyperbolic group, let $\calD_\G$ denote the space of all left-invariant and Gromov hyperbolic pseudo metrics on $\G$ that are quasi-isometric to a word metric for a finite generating set. The \emph{space of metric structures} on $\G$, denoted by $\scrD_\G$, is the space of equivalence classes of $\calD_\G$ under \emph{rough similarity}. That is, two pseudo metrics $d_1,d_2\in \calD_\G$ represent the same metric structure if there exist $\lam,A>0$ such that 
\begin{equation}\label{eq.roughsimilarity}
    |d_1(g,h)-\lam d_2(g,h)|\leq A 
\end{equation} for all $g,h\in \G$. We let $[d]$ denote the rough-similarity class containing $d\in \calD_\G$.

A geometric action of $\G$ on a Gromov hyperbolic geodesic metric space $X$ induces a metric structure $\rho_X\in \scrD_\G$ as follows. For any base point $x\in X$ the pseudo metric $d_X^x(g,h):=d_X(gx,hx)$ belongs to $\calD_\G$, and the metric structure $\rho_X=[d_X^x]$ is independent of the point $x$. This is the case of geometric actions on negatively curved Riemannian manifolds or $\CAT(0)$ cube complexes. 

The space $\scrD_\G$ can be endowed with the metric $\Del$ defined by 
\[\Del([d_1],[d_2]):=\log{\inf \{\lam_1\lam_2 \colon \exists A>0 \text{ s.t. } \lam_1^{-1}d_1-A \leq d_2\leq \lam_2d_1+A\}}.\]
If $\rho_X,\rho_Y\in \scrD_\G$ are induced by the geometric actions of $\G$ on $X$ and $Y$, then $\Del(\rho_X,\rho_Y)$ is the infimum $\log {\lam}$ such that there exists a $\G$-equivariant $\lam$-quasi-isometry from $X$ to $Y$.

In this setting, \cref{conj:futer.wise} can be described as follows. If $\G$ is a torsion-free and acts geometrically on $\Hy^n$ where $n\leq 3$ or $\G$ is arithmetic of simplest type, then the induced metric structure $\rho_{\Hy^n}$ belongs to the closure in $(\scrD_\G,\Del)$ of the subspace
\[\scrD^{cub}_\G:=\{\rho_{\wtilde{\calX}}:\G \text{ acts geometrically on the } \CAT(0) \text{ cube complex }\wtilde{\calX}\},\]
when these actions are considered with the combinatorial metric. We say a metric structure $\rho\in\scrD_\G$ is \emph{approximable by cubulations} if it lies in the closure of $\scrD_\G^{cub}$ in $(\scrD_\G,\Del)$.


In the 2-dimensional case we have a clearer picture, since for a surface group $\G$, the space $\bbP\Curr_f(\G)$ of (projective) filling geodesic currents naturally embeds into $\scrD_\G$, see \cite[Cor.~6.10]{cantrell-reyes.manhattan} and \cite{derosa-martinezgranado}. In this case, \cref{thm.2-dimensions} generalizes to the following characterization of $\bbP\Curr_f(\G)$ in terms of cubulations.

\begin{corollary}\label{coro.fillingcurrents=cubulations}
    The space $\bbP\Curr_f(\G)\subset \scrD_\G$ equals the closure in $\scrD_\G$ of the space of metric structures represented by geometric actions of $\G$ on $\CAT(0)$ cube complexes with cyclic hyperplane stabilizers.
\end{corollary}

Similarly, if $\G$ is a surface group, then there is a natural injection of the quasiFuchsian space $\scrQ\scrF_\G$ into $\scrD_\G$ \cite{burger}, and \cref{prop.main.quasifuchsian} asserts that the image of this embedding is contained in the closure of $\scrD_\G^{cub}$. The next theorem gives a precise description of the intersection of the subspaces $\scrQ\scrF_\G$ and $\bbP\Curr_\G$ inside $\scrD_\G$.

\begin{proposition}\label{prop.QFintCurr=Teich}
    If $\G$ is a surface group, then in $\scrD_\G$ we have
    \[\scrQ\scrF_\G \cap \bbP\Curr_{f}(\G)=\scrT_\G,\]
   where $\scrT_\G$ is the Teichm\"uller space of $\G$. Moreover, we have:
    \begin{enumerate}
        \item for any $\rho\in \scrQ\scrF_\G \bs \scrT_\G$ there exists $\lam_\rho>0$ such that
    \[\Del(\rho,\rho_{[\eta]})\geq \lam_{\rho}\]
    for every $[\eta]\in \bbP\Curr_{f}(\G)$, where $\rho_{[\eta]}\in \scrD_\G$ represents $[\eta]$; and,
        \item for any $[\eta]\in \bbP\Curr_{f} (\G) \bs \scrT_\G$ there exists $\lam_{[\eta]}>0$ such that
    \[\Del(\rho,\rho_{[\eta]})\geq \lam_{[\eta]}\]
    for every $\rho\in \scrQ\scrF_\G$.
    \end{enumerate}
\end{proposition}

When $[\eta]$ represents a negatively curved Riemannian metric on a surface with fundamental group $\G$, Item (2) answers in the affirmative a question of Fricker and Furman \cite[Question~1.4]{fricker-furman}.

Combined with \cref{coro.fillingcurrents=cubulations}, this result says that translation lengths of points in $\scrQ\scrF_\G \bs \scrT_\G$ \emph{cannot} be represented by geodesic currents in $\Curr(\G)$ via the intersection number. In a forthcoming project \cite{martinezgranado-reyes}, the second-named author and Didac Martinez-Granado plan to construct a general setting in which quasiFuchsian representations are dual to appropriate co-geodesic currents.


\subsection{Applications}
The original motivation for \cref{conj:futer.wise} in \cite{futer-wise} was to obtain effective bounds for the density on which random quotients of cubulable hyperbolic groups are again hyperbolic and cubulable. In the case of quotients of hyperbolic groups, the random model also depends on a geometric action of $\G$ on a geodesic metric space $X$. More precisely, given such an action and a subset $E\subset \G$, the \emph{exponential growth rate} of a subset $E\subset \G$ is the quantity
\begin{equation}\label{eq.EGR}
    v_X(E):=\limsup_{k\to \infty}{\frac{\log \# \{g\in \G \colon d_X(go,o)\leq k\}}{k}},
\end{equation}
which is always finite and independent of the point $o\in X$. The quantity $v_X=v_X(\G)$ is called the \emph{exponential growth rate} of the action. This is equivalently defined as the abscissa of convergence of the Poincar\'e series determined by the action, and may also be called the \emph{critical exponent}.

In this setting, if the geometric action on $X$ is $\G$-equivariantly $\lam$-quasi-isometric to a geometric action on the $\CAT(0)$ cube complex $\wtilde\calX$, then \cite[Thm.~1.3]{futer-wise} gives bounds for the density depending on $\lam$, $v_X$ and the maximal exponential growth rate $v_X(H)$ where $H<\G$ is the stabilizer of an essential hyperplane of $\wtilde\calX$. Since our results allow us to take $\lam$ as close to 1 as we wish, by applying \cite[Thm.~1.3]{futer-wise} we deduce the following.

\begin{corollary}\label{coro.genericquotient}
    Let $\G$ act freely and geometrically on the Gromov hyperbolic space $X$ and let $k\leq e^{c\ell}$, where either:
   \begin{itemize}
        \item $X$ is a negatively curved Riemannian surface and $c<\frac{v_X}{41}$; or,
        \item $X=\Hy^3$ and $c<\frac{2}{41}$; or, 
        \item $X=\Hy^n$ with $n\geq 4$, $\G$ is arithmetic of simplest type and $c<\frac{1}{20}$. 
    \end{itemize}
Then with overwhelming probability as $\ell \to \infty$, for any set of conjugacy classes $[g_1],\dots,[g_k]$ in $\G$ with stable translation lengths $\ell_X[g_i]\leq \ell$, the quotient $\ov{\G} = \G / \genby{\genby{g_1,\dots,g_k}}$ is hyperbolic and acts geometrically on a $\CAT(0)$ cube complex.
\end{corollary}

Our work also has consequences for the growth of quasiconvex subgroups of lattices. Following \cite[Def.~1.1]{li-wise} if $\calP$ is a property, we say that $\G$ has a $\calP$ \emph{growth-gap} relative to the action on $X$ if 
\begin{equation*}
    \sup\{v_X(H) \colon H<\G, |\G:H|=\infty, H \text{ satisfies }\calP\}<v_X.
\end{equation*}
In \cite[Main~ Theorem~1.5]{li-wise} Li and Wise proved that if $\calX$ is a compact special cube complex, then $\G=\pi_1(\calX)$ does not have a quasiconvex growth-gap for its action on $\wtilde{\calX}^{(1)}$. Combining this with our results and Agol's theorem \cite{agol.haken}, we obtain the following result, which partially answers \cite[Problem~9.4~(1)]{li-wise}.
\begin{corollary}\label{cor.nogrowthgap}
    Let $\G$ act geometrically and freely on the Gromov hyperbolic space $X$, so that either: 
\begin{itemize}
\item $\G$ is a surface group and $\rho_X \in \bbP\Curr_f(\G)\cup \scrQ\scrF_\G$; or,
\item $X=\Hy^3$; or,
\item $X=\Hy^n$ with $n\geq 3$ and $\G$ is of simplest type.
\end{itemize}
Then this action has no growth-gap: there exists a sequence $(H_m)_m$ of infinite index quasiconvex subgroups of $\G$ such that
    \begin{equation*}
        \lim_{m \to \infty}{v_{X}(H_m)}=v_X.
    \end{equation*}
\end{corollary}

In the 3-dimensional case, \cref{cor.nogrowthgap} also follows from the work of Rao \cite[Thm.~1.2]{rao}. In the arithmetic case, \cref{cor.nogrowthgap} recovers a result of McMullen \cite[Sec.~6]{magee}.

This corollary contrasts with the case of non-elementary hyperbolic groups with property (T), which have growth-gap with respect to any geometric action on a $\CAT(-1)$ space \cite{CDS}.


\subsection*{Organization}
In \cref{sec.preliminaries} we discuss the main facts about hyperbolic groups, geodesic currents, $\CAT(0)$ cube complexes, and metric structures that we will use throughout our work. In \cref{sec.co-geodesic} we study co-geodesic currents for hyperbolic groups with sphere boundary, as well as the associated intersection number. The main results are \cref{prop.contcogeodesic}, which (under the appropriate assumptions) promotes weak-$\ast$ convergence of co-geodesic currents to convergence of metric structures, and \cref{prop.dualcubulation} which shows that discrete co-geodesic currents are dual to cubulations. These results are applied to surface groups in \cref{sec.2-dim}, from which we deduce \cref{thm.2-dimensions}. There, we show that singular non-positively curved Riemannian metrics on surfaces and positively ratioed representations can be approximated by cubulations. \cref{sec.currentsQC} discusses co-geodesic currents supported on quasiconformal codimension-1 hyperspheres for uniform lattices in $\Hy^n$. We prove \cref{prop.criterioncontinuous}, a criterion for continuity of the intersection number, and describe the Haar current, which is dual to the lattice action. \cref{sec.arithmetic} studies the arithmetic case of \cref{conj:futer.wise}, in which we prove \cref{thm.main.arithmetic} based on classical results from homogeneous dynamics.
In \cref{sec.3-dim}, we specialize to 3-dimensions and prove \cref{thm.main.3manifold}. This is a consequence of \cref{thm.3dimcurrtocurr}, which upgrades convergence of measures on the 2-plane Grassmannian bundle of a closed hyperbolic 3-manifold to convergence of co-geodesic currents. The key step to prove this theorem is \cref{prop.convergenceintegral}, which relies on an analysis of the geometry of minimal disks in $\Hy^3$ with quasicircles as limit sets. The approximation of quasiFuchsian representations by cubulations is discussed in \cref{sec.quasifuchsian}, in which we prove \cref{prop.main.quasifuchsian}. For that, we use \cref{thm.main.3manifold}, a theorem of Brooks \cite{brooks}, and \cref{prop.homeoccpt}, which relates different topologies on the space of convex-cocompact representations. In this section, we also prove \cref{prop.QFintCurr=Teich}. Corollaries \ref{coro.genericquotient} and \ref{cor.nogrowthgap} are proven in \cref{sec.applications}, and we conclude the paper with some further questions in \cref{sec.questions}.


\subsection*{Acknowledgements}
Both authors are grateful to Sami Douba, Sam Fisher, Ursula Hamenst\"adt, Didac Martinez-Granado, and Andrew Yarmola for interesting conversations. We also thank David Futer and Franco Vargas-Pallete for their helpful comments on an early draft. We are indebted to Fernando Al Assal for sharing his thesis before the preprint \cite{alassal} was publicly available. The second author would like to thank the Max Planck Institut f\"ur Mathematik for its hospitality and financial support.


\section{Preliminaries}\label{sec.preliminaries}

\subsection{Hyperbolic Groups}

For references on hyperbolic groups, see \cite{ghys-delaharpe,bridson-haefliger,CDP,drutu-kapovich}. A pseudo metric space $(X,d)$ is $\del$-\emph{hyperbolic} $(\del\geq 0)$ if every four-tuple $x,y,z,w\in X$ satisfies the inequality
\[(x|z)_{w,d}\geq \min\{(x|y)_{w,d},(y|z)_{w,d}\}-\del,\]
where $2(x|y)_{w,d}:=d(x,w)+d(w,y)-d(x,y)$ denotes the \emph{Gromov product}. The space $(X,d)$ is \emph{(Gromov) hyperbolic} if it is $\del$-hyperbolic for some $\del$.

A finitely generated group $\G$ is \emph{hyperbolic} if it admits a proper and cocompact action by isometries (i.e. a \emph{geometric} action) on a Gromov hyperbolic metric space. Examples of hyperbolic groups include finitely generated free groups and fundamental groups of closed negatively curved manifolds. A hyperbolic group is \emph{elementary} if it is virtually cyclic. In the sequel, unless otherwise stated, $\G$ always denotes a non-elementary hyperbolic group with identity element $o$. In this case, $\G$ has a well-defined \emph{Gromov boundary}, denoted by $\partial \G$. This is a compact, uncountable metrizable space. Furthermore, there is a natural topology on $\G \cup \partial \G$ that induces the topology on $\partial \G$, which allows us to discuss convergence in $\G\cup\partial \G$. The left action of $\G$ on itself induces a natural topological action on $\partial \G$.

The \emph{limit set} of a subset $A\subset \G$ is the set $\Lam(A)$ of all points in $\partial \G$ which are limits of sequences in $A$. The set $\Lam(A)$ is always closed, and it is non-empty so long as $A$ is infinite. If $g\in \G$ is a non-torsion element, then the
limit set of $\langle g\rangle$ consists of two points $g^\infty,g^{-\infty}$, characterized by the identities $g^{\pm \infty}=\lim_{m\to \pm \infty}{g^m}$. 


\subsection{Geodesic currents}

Let $\calG\subset  \partial \G \times \partial \G$ denote the space of ordered pairs of distinct points in $\partial \G$, equipped with the subspace topology. Note that $\G$ acts on $\calG$ via $g\cdot (x,y)=(gx,gy)$, and that this action cocompact. Let $\tau: \calG \to \calG$ denote the involution $\tau(x,y)=(y,x)$, which is fixed-point free since we have removed the diagonal. 

\begin{definition}
    A \emph{geodesic current} on $\G$ is a $\G$-invariant, locally finite Borel measure on $\calG$ that is also $\tau$-invariant. Let $\Curr(\G)$ be the space of all the geodesic currents on $\G$, equipped with the weak-$\ast$ topology.
\end{definition}

\begin{remark}\label{rmk.choiceordered}
    Equivalently, we could have defined a geodesic current as a $\G$-invariant locally finite measure on the set $\partial^2\G$ of \emph{unordered} pairs of distinct points in $\partial \G$.
    However, it will be more convenient to work with ordered pairs (particularly in \cref{sec.currentsQC}), so we have decided to work with measures on $\calG$ instead. 
\end{remark}

For any $g\in \G$ we associate a geodesic current $\eta_{[g]}\in \Curr(\G)$ as follows. Suppose first that $g$ is \emph{primitive}, in the sense that if $g=h^m$ for some $h\in \G$ and $m\in \Z$ then $m=\pm 1$. In particular $g$ is non-torsion, so we let $\calS_{[g]}$ be the set of all the pairs in $\calG$ of the form $(hg^{\pm \infty},hg^{\mp \infty})$ for some $h\in \G$. Then $\calS_{[g]}$ is a discrete $\G$-invariant subset of $\calG$ and the expression
    \[\eta_{[g]}:=\sum_{\gam\in \calS_{[g]}}{\del_{\gam}}\]
   defines a locally finite measure, where $\del_x$ denotes the Dirac measure at the point $x$. If $g$ is non-torsion, then we have $g=h^m$ for some primitive element $h$ and $m\in \Z$, and we define $\eta_{[g]}=|m|\eta_{[h]}$. For completeness, we define $\eta_{[g]}$ as the zero current when $g$ is torsion. Any such current $\eta_{[g]}$ will be called a \emph{rational} current.


   Bonahon \cite{bonahon} proved  that the set $\{\lam \eta_{[g]} : \lam>0, g\in \G\}$ is dense in $\Curr(\G)$. He also proved that the set of projective geodesic currents $\bbP\Curr(\G):=(\Curr(\G)-\{0\})/\R_+$ is compact and metrizable for the quotient topology.

\begin{remark}\label{remark.primitive}
    If $\G$ is torsion-free, a non-trivial element $g\in \G$ is primitive if and only if the subgroup generated by $g$ equals the set-wise stabilizer of $\{g^\infty,g^{-\infty}\}$ in $\G$. Indeed, if the latter subgroup is $A$, then $A$ is virtually cyclic, and hence cyclic and generated by an element $a\in \G$. Then $g=a^k$ for some $k$, and if $g$ is primitive then $A=\genby{g}$. Conversely, if $A=\genby{g}$ and $g=b^r$ for some $b\in \G$, then $b$ is non-torsion and fixes $\{g^{\infty},g^{-\infty}\}$, so that $b\in A$. Then $b=g^s$ for some $s$ and $r= \pm 1$. 
\end{remark}


\subsection{CAT(0) cube complexes and Sageev's construction}\label{subsec.CAT(0)cc}
In this subsection we describe Sageev's construction to produce cubical actions on $\CAT(0)$ cube complexes,  following \cite{bridson-haefliger,hruska-wise}. 

A \emph{cube complex} is a metric polyhedral complex in which all polyhedra are unit-length
Euclidean cubes. Such a complex is \emph{non-positively curved (NPC)} if its universal cover
is a $\CAT(0)$ metric space when endowed with the induced length distance.

By Gromov's criterion \cite[Thm.~II.5.2]{bridson-haefliger}, a cube complex $\calX$ is NPC if and only if the link of each vertex is a \emph{flag complex}.
Recall that a flag simplicial complex is a complex determined by its 1–skeleton: for every complete subgraph of the 1–skeleton, there is a simplex with 1–skeleton equal to that subgraph. A \emph{$\CAT(0)$ cube complex} is a simply connected NPC cube complex.

The crucial aspect of the geometry of a $\CAT(0)$ cube complex is the notion of a hyperplane. Each $n$-cube $C=[0,1]^n$ has $n$ \emph{midcubes}, obtained by setting one coordinate to $1/2$. The face of a midcube of $C$ is the midcube of a face of $C$, and so the set of midcubes of a cube complex $\calX$ has a cube complex structure, called the \emph{hyperplane complex}. A \emph{hyperplane} $\frakh$ of $\calX$ is a connected component of the hyperplane complex. If $\wtilde\calX$ is a $\CAT(0)$ cube complex, any hyperplane can be considered as a convex, 2-sided subspace of $\wtilde\calX$, which is a subcomplex of the first cubical barycentric subdivision of $\wtilde\calX$. A hyperplane $\frakh$ \emph{separates} two vertices of $\wtilde\calX$ of they lie on different components of $\wtilde\calX \bs \frakh$. An edge $e$ of $\wtilde\calX$ is \emph{dual} to  $\frakh$ if the vertices determined by $e$ are separated by $\frakh$. 
The \emph{carrier} $N(\frakh)$ is the set of all vertices in an edge dual to $\frakh$. The convex subcomplex of $\wtilde\calX$ spanned by a carrier is convex.

The \emph{combinatorial distance} on the $\CAT(0)$ cube complex $\wtilde\calX$ is the graph metric $d_{\wtilde\calX}$ on the 1-skeleton $\wtilde\calX^{(1)}$ so that each edge has length one. Often we will also consider the restriction of this metric to the 0-skeleton $\wtilde\calX^{(0)}$. In this case the distance $d_{\wtilde\calX}(v,w)$ of two vertices $v,w$ equals the number of hyperplanes separating them. If $\G$ is a group acting by cubical automorphisms on $\wtilde\calX$, then its action on $(\wtilde\calX^{(0)},d_{\wtilde\calX})$ is by isometries.

To obtain actions on $\CAT(0)$ cube complexes, we recall the notion of wallspace structure. A \emph{wall} of the non-empty set $X$ is an unordered pair $\calW=\{U,U^\ast\}$ of non-empty subsets of $X$ such that $U \cup U^\ast=X$ and $U,U^\ast \neq X$. The \emph{closed halfspaces} determined by $\calW$ are $U$ and $U^\ast$, and the \emph{open halfspaces} are $U\bs U^\ast$ and $U^\ast \bs U$. This wall \emph{separates} two points $x,y\in X$ if they lie in distinct open halfspaces determined by $\calW$.   

A \emph{wallspace structure} on $X$ is a set $\bbW$ of walls of $X$ satisfying:
\begin{itemize}
    \item for any $x\in X$ there exist only finitely many $\calW=\{U,U^\ast\}\in \bbW$ such that $x\in U\cap U^\ast$; and,
    \item for any $x,y\in X$ the number of walls in $\bbW$ separating them is finite.
\end{itemize}

\begin{remark}
    The definition of a wallspace structure presented here differs from that of \cite[Sec.~2.1]{hruska-wise}. For us, a wallspace is a set of walls rather than a collection, meaning we do not allow for duplicates of walls. In addition, closed halfspaces of walls are by definition non-empty. However, we do allow walls $\calW=\{U,U^\ast\}$ such that $U\cap U^\ast$ is non-empty. 
\end{remark}

If a group $\G$ acts on $X$, then the wallspace structure $\bbW$ on $X$ is $\G$-\emph{invariant} if for any $\calW=\{U,U^\ast\}\in \bbW$ and $g\in \G$, the wall $g\calW:=\{gU,gU^\ast\}$ also belongs to $\bbW$. 

As in \cite[Sec.~3]{hruska-wise}, a wallspace structure $\bbW$ on the space $X$ determines a $\CAT(0)$ cube complex $\wtilde\calX$ as follows.
An \emph{orientation} $v$ of $\bbW$ is a choice $v(\calW)\in \calW$ of a closed halfspace for each $\calW\in \bbW$. This definition is meaningful since $\# \calW=2$ for any $\calW$.  

Each vertex of the $\CAT(0)$ cube complex $\wtilde\calX$ will be an orientation of $\bbW$ satisfying some additional conditions. Two vertices $v,w$ of $\wtilde\calX$ are adjacent if and only if there exists a unique wall $\calW$ in $\bbW$ such that $\calW=\{v(\calW),w(\calW)\}$. In this case, we say that the hyperplane dual to the edge determined by $v,w$ \emph{corresponds} to $\calW$. This correspondence gives us a bijection between $\bbW$ and the set of hyperplanes of $\wtilde{\calX}$. 

Each point $x$ of $X$ determines a \emph{canonical cube} \cite[Def.~3.6]{hruska-wise}, which is the set of vertices $v\in \wtilde{\calX}$ such that $x\in v(\calW)$ for all $\calW\in \bbW$. The dimension of this cube equals the number of walls $\calW=\{U,U^\ast\}$ such that $x\in U\cap U^\ast$, which is finite. 

If the wallspace structure $\bbW$ is invariant under a group $\G$ acting on $X$, then $\G$ acts naturally on $\wtilde\calX$ by cubical automorphisms. In this case, the group $\G$ also acts on the set of hyperplanes of $\wtilde\calX$. The \emph{stabilizer} of the hyperplane $\frakh$ is then the set $\G_{\frakh}$ of all $g\in \G$ such that $g\frakh=\frakh$. It is evident that if $\frakh$ corresponds to the wall $\calW$, then $\G_\frakh$ is precisely the set of group elements fixing $\calW$ for the action of $\G$ on $\bbW$.

\subsection{Metric structures and the Lipschitz distance}

Let $\G$ be a hyperbolic group. Given a left-invariant pseudo metric $d$ on $\G$, we let $\ell_d:\G \ra \R$ be the translation length function of the left action of $\G$ on $(\G,d)$. Recall that $\calD_\G$ is the space of left-invariant pseudo metrics on $\G$ that are Gromov hyperbolic and quasi-isometric to a word metric for a finite symmetric generating subset of $\G$. This space is contained in the space $\ov\calD_\G$ of left-invariant pseudo metrics $d$ on $\G$ satisfying:
\begin{itemize}
    \item $\ell_d$ is not identically zero; and,
    \item for some (hence any) $d'\in \calD_\G$, there exists $\lam>0$ such that
    \[(g|h)_{o,d}\leq \lam (g|h)_{o,d'}+\lam\]
    for all $g,h\in \G$.
\end{itemize}
As with the case of $\calD_\G$, we let $\ov\scrD_\G$ be the quotient of $\ov\calD_\G$ by rough similarity (recall \eqref{eq.roughsimilarity}).
Points in $\ov{\scrD}_\G$ are called \emph{metric structures}, and $[d]$ denotes the metric structure induced by $d\in \ov\calD_\G$.  

The extension from $\scrD_\G$ to $\overline{\scrD}_\G$ allows us to encode some cocompact isometric actions of $\G$ that are not necessarily proper. For example, if $\G$ acts cubically and cocompactly on the $\CAT(0)$ cube complex $\wtilde\calX$ and the hyperplane stabilizers are quasiconvex of infinite index, then the isometric action of $\G$ on $(\wtilde\calX^{(1)},d_{\wtilde\calX})$ determines a metric structure $\rho_{\wtilde\calX}\in \ov\scrD_\G$ \cite[Prop.~6.14]{cantrell-reyes.manhattan}. Moreover, this structure belongs to $\scrD_\G$ if and only if the action of $\G$ on $\wtilde\calX$ is proper. For more details about the space $\ov\scrD_\G$, see \cite{cantrell-reyes.manhattan}.

Suppose $\G$ acts isometrically on a metric space $X$. If there is some $x\in X$ so that the function $d_X^x(g,h):=d_X(gx,hx)$ determines a pseudo metric belonging to $\ov\calD_\G$, we will write $\rho_X\in\ov\scrD_\G$ to denote this metric structure.


\section{Co-geodesic structures}\label{sec.co-geodesic}

Throughout this section we assume $\G$ is a torsion-free hyperbolic group with boundary homeomorphic to the $(n-1)$-sphere $S^{n-1}$ ($n\geq 2$). We let $\calC(\partial \G)$ be the space of non-empty compact subsets of $\partial \G$ equipped with the Hausdorff topology, and $\calC_{\geq 2}(\partial \G)\subset \calC(\partial \G)$ be the set of compact subsets of cardinality at least 2, equipped with the subspace topology. 
The natural action of $\G$ on $\partial \G$ induces topological actions on $\calC(\partial \G)$ and $\calC_{\geq 2}(\partial \G)$. All subsets of $\partial \G$ are considered with the subspace topology.

\subsection{Systems of spheres and $\calS$-currents}

The space of geodesic currents on a surface group is defined in terms of limit sets of geodesics on the surface. We begin this section by introducing an analogous space for more general hyperbolic groups with a sphere boundary.

\begin{definition}\label{def.systemspheres}
    A \emph{system of (hyper)spheres at infinity} for $\G$ is a non-empty set $\calS \subset \calC(\partial \G)$ of compact subsets of $\partial \G$ satisfying the following.
    \begin{itemize}
        \item[i)] Each $S\in \calS$ is homeomorphic to the $(n-2)$-dimensional sphere $S^{n-2}$.
        \item[ii)] $\calS$ is $\G$-invariant under the action of $\G$ on $\calC(\partial \G)$.
        \item[iii)] $\calS$ is a closed subset of $\calC_{\geq 2}(\partial \G)$. 
    \end{itemize} 
    Given such a system of spheres $\calS$, an \emph{$\calS$-current} is a $\G$-invariant Radon measure on $\calS$. We let $\calS\Curr(\G)$ denote the space of $\calS$-currents, endowed with the weak-$\ast$ topology. 
\end{definition}

We now make some observations about the above definition. First, we note that a system $\calS$ of spheres at infinity is locally compact, separable, and metrizable. This follows since $\calC_{\geq 2}(\partial \G)$ is locally compact, separable, and metrizable \cite[Sec.~2.1]{sasaki}. In addition, the action of $\G$ on $\calC_{\geq 2}(\partial \G)$ is cocompact \cite[Lem.~2.5]{sasaki}, and hence the action of $\G$ on $\calS$ is also cocompact. Therefore, by \cite[Thm.~2.23]{sasaki} we deduce that the space $\calS\Curr(\G)$ is locally compact, separable, and completely metrizable. We summarize these properties in the following lemma.

\begin{lemma}\label{lem.Scurrentscompact}
    Let $\calS$ be a system of spheres at infinity for $\G$. Then the following properties hold.
    \begin{itemize}
        \item $\calS$ is locally compact, separable and metrizable. 
        \item The action of $\G$ on $\calS$ is cocompact.
        \item The space of $\calS$-currents $\calS\Curr(\G)$ is locally compact and metrizable. 
        \item The space of projective $\calS$-currents $\bbP\calS\Curr(\G):=(\calS\Curr(\G)-\{0\})/\R_+$ is compact metrizable when endowed with the quotient topology.
    \end{itemize}
\end{lemma}

\begin{definition}\label{def.cogeodesic}    
   The system of spheres $\calS$ is \emph{co-geodesic} if in addition it satisfies the following.
   \begin{itemize}
       \item[iv)] The diagonal action of $\G$ on 
       \[\calI_\calS:=\{(S,(p,q))\in \calS\times \calG \colon p,q \text{ belong to different components of }\partial \G \bs S\}\]
       is free and properly discontinuous (we say that $S\in \calS$ \emph{separates} $p,q\in \partial \G$ if $(S,(p,q)) \in \calI_\calS$).
       \item[v)] For some (any) Borel fundamental domain $\calB$ for the action of $\G$ on $\calI_\calS$, the quantity $(\al \times \eta)(\calB)$ is finite for all $\al\in \calS\Curr(\G)$ and $\eta\in \Curr(\G)$. 
   \end{itemize}
   In that case, we define the \emph{intersection number} $$i_\calS: \calS\Curr(\G) \times \Curr(\G) \ra \R_{\geq 0}$$ according to 
   \begin{equation}\label{eq.defi_S}
       i_\calS(\al,\eta):=\frac{1}{2}(\al \times \eta)(\calB),
   \end{equation}
   and we say that the pair $(\calS\Curr(\G),i_\calS)$ is the \emph{co-geodesic structure induced} by $\calS$. A \emph{co-geodesic current} is an $\calS$-current for some co-geodesic system of spheres at infinity. 
\end{definition}

\begin{remark}
    The extra factor of $1/2$ in \eqref{eq.defi_S} comes from our choice of working with ordered pairs of points in $\partial \G$, see \cref{rmk.choiceordered}. It is also justified by \eqref{eq.dualityformula} in \cref{prop.dualcubulation} below.
\end{remark}

It follows from the definition that the function $i_\calS$ is $\R_+$-bilinear. Also, if $\calS$ and $\calS'$ are co-geodesic systems of spheres are infinity with $\calS\subset \calS'$, then there is a natural inclusion of $\calS\Curr(\G)$ into $\calS'\Curr(\G)$ as the closed subset of $\calS'$-currents with support in $\calS$, and also $i_{\calS'}(\al,\eta)=i_\calS(\al,\eta)$ for all $\eta\in \Curr(\G)$ and $\al\in \calS\Curr(\G)\subset \calS'\Curr(\G)$.

\begin{remark}
    It may happen that every system of spheres at infinity is co-geodesic. That is, that Items iv) and v) are consequences of Items i)-iii) of  \cref{def.systemspheres}. However, we do not intend to approach this problem in this work.
\end{remark}

\begin{example}[Surface groups]\label{ex.surfacegroups}
If $\G$ is the fundamental group of a closed orientable hyperbolic surface, then $n=2$ and $(n-2)$-spheres at infinity are just unordered pairs of distinct points in $\partial \G$. In this case, systems of spheres at infinity are precisely the closed and $\G$-invariant subsets of $\partial^2 \G$, and hence co-geodesic currents coincide with geodesic currents (see \cref{rmk.choiceordered}). In particular, \eqref{eq.defi_S} reduces to the usual definition of intersection number for geodesic currents on surfaces groups given by Bonahon \cite{bonahon}.
\end{example}

As in the case of surface groups, we want to relate co-geodesic currents with isometric actions of $\G$.

\begin{definition}\label{def.filling-dual}
Let $\calS$ be a co-geodesic system of spheres at infinity with induced co-geodesic structure $(\calS\Curr(\G),i_\calS)$.
\begin{enumerate}
    \item  An $\calS$-current $\al\in \calS\Curr(\G)$ is \emph{filling} (resp. \emph{weakly filling}) if $i_\calS(\al,\eta)>0$ for all non-zero currents $\eta \in \Curr(\G)$ (resp. $i_\calS(\al,\eta_{[g]})>0$ for any non-trivial element $g\in \G$). Similarly, a projective $\calS$-current $[\al]\in \bbP\calS\Curr(\G)$ is \emph{filling} (resp. \emph{weakly filling}) if $\al$ is filling (resp. weakly filling) for some (and hence, any) representative $\al$ of $[\al]$.
    \item A pseudo metric $d\in \ov{\calD}_\G$ is \emph{dual} to the $\calS$-current $\al\in \calS\Curr(\G)$ if for any $g\in \G$ we have  \[\ell_d[g]=i_\calS(\al,\eta_{[g]}).\] 
    Similarly, a metric structure $[d]\in \ov{\scrD}_\G$ is \emph{dual} to the projective $\calS$-current $[\al]\in \bbP\Curr(\G)$ if $d$ is dual to $t\al$ for some real number $t$.
\end{enumerate}
\end{definition}

\begin{remark}\label{rmk.filling=proper}Since $\G$ is torsion-free, for any $d\in \calD_\G$ the function $\ell_d$ uniquely extends to a continuous function $\ell_d:\Curr(\G) \ra \R$ such that $\ell_d(\eta_{[g]})=\ell_d[g]$ for all $g\in \G$ (see e.g.~ \cite[Thm.~1.7]{cantrell-reyes.manhattan}). Indeed, this function is $\R_+$-\emph{linear}, meaning that $$\ell_d(\eta_1+t\eta_2)=\ell_d(\eta_1)+t\ell_d(\eta_2)$$ for all $\eta_1,\eta_2\in \Curr(\G)$ and $t\in \R_+$. This is implicit in the proofs of \cite[Cor.~5.2]{oregon-reyes.ms} and \cite[Thm.~1.5]{EPS}, on which \cite[Thm.~1.7]{cantrell-reyes.manhattan} relies on.  From this, it follows that if a pseudo metric $d\in \ov{\calD}_\G$ is dual to $\al\in \calS\Curr(\G)$ and $\eta \mapsto i_\calS(\al,\eta)$ is continuous on $\Curr(\G)$, then
$$\ell_d(\eta)=i_\calS(\al,\eta)$$
for all $\eta\in \Curr(\G)$. In that case, $\al$ is filling if and only if $d\in \calD_\G$.
\end{remark}

\subsection{Duality principle}

The general strategy to prove our main results is to deduce convergence of a sequence of metric structures in $\scrD_\G$ from the convergence of a sequence of corresponding dual $\calS$-currents for a convenient co-geodesic system $\calS$ of spheres at infinity. For that purpose, we prove the following proposition.

\begin{proposition}\label{prop.contcogeodesic}
    Let $\calS$ be a co-geodesic system of spheres at infinity for $\G$, with induced co-geodesic structure $(\calS\Curr(\G),i_\calS)$. Let $(d_m)_m\subset \ov{\calD}_\G$ be a sequence of pseudo metrics with each $d_m$ being dual to $\al_m\in \calS\Curr(\G)$, such that the projective classes $[\al_m]$ converge to $[\al]$ in $ \bbP\calS\Curr(\G)$. If $[\al]$ is dual to the metric structure $[d]\in \scrD_\G$ and $i_\calS$ is continuous at $(\al,\eta)$ for all $\eta\in \Curr(\G)$, then:
    \begin{enumerate}
        \item $[\al_m]$ is filling for all $m$ large enough; and,
        \item $[d_m]$ converges to $[d]$ in $(\scrD_\G,\Del)$.
    \end{enumerate}
    
\end{proposition}

\begin{proof}[Proof]
    Let $(t_m)_m\subset \R_+$ be so that $t_m\al_m$ converges to $\al$ in $\calS\Curr(\G)$. From \cref{rmk.filling=proper} and the continuity of $i_\calS$ at $(\al,\eta)$ for any $\eta\in \Curr(\G)$, we see that $\al$ is filling.

    To prove (1), suppose for the sake of contradiction that, after extracting a subsequence and reindexing, there exist non-zero currents $\eta_m\in \Curr(\G)$ such that $i_\calS(\al_m,\eta_m)=0$ for all $m$. After extracting a further subsequence and rescaling, we can assume that $\eta_m$ converges to the non-zero current $\eta$. But then the continuity of $i_\calS$ at $(\al,\eta)$ implies  $i_\calS(\al,\eta)=\lim_{m}{i_\calS(t_m\al_m,\eta_m)}=0$, contradicting that $\al$ is filling. 
    
    To prove (2), without loss of generality, we can assume that $d_m\in \calD_\G$ for all $m$ (see \cref{rmk.filling=proper}), for which we need to prove that the sequence $$m \mapsto \exp{\Del([d_m],[d])}=\left( \sup_{g \neq o}{\frac{\ell_{d_m}[g]}{\ell_{d}[g]}}\right)\cdot\left( \sup_{h \neq o}{\frac{\ell_{d}[h]}{\ell_{d_m}[h]}}\right)$$
    converges to 1, see \cite[Prop.~3.5]{oregon-reyes.ms}. Otherwise, there would exist a subsequence (also denoted $d_m$), some $\lam>1$ and sequences $(g_m)_m,(h_m)_m \subset \G$ of non-torsion elements such that \small
    \begin{equation}\label{eq.ineqcurrents}   \frac{i_\calS(t_m\al_m,\eta_{[g_m]})i_\calS(\al,\eta_{[h_m]})}{i_\calS(\al,\eta_{[g_m]})i_\calS(t_m\al_m,\eta_{[h_m]})}=\frac{i_\calS(\al_m,\eta_{[g_m]})i_\calS(\al,\eta_{[h_m]})}{i_\calS(\al,\eta_{[g_m]})i_\calS(\al_m,\eta_{[h_m]})}=\frac{\ell_{d_m}[g_m]}{\ell_{d}[g_m]}\cdot \frac{\ell_{d}[h_m]}{\ell_{d_m}[h_m]}\geq \lam
    \end{equation}
    \normalsize
    for all $m$. After extracting a new subsequence and rescaling, by the compactness of $\bbP\Curr(\G)$ we can assume that $\eta_{[g_m]} \to \beta$ and $\eta_{[h_m]} \to \gam$ with $\beta$ and $\gam$ non-zero currents. But by letting $m$ tend to infinity in \eqref{eq.ineqcurrents}, and since $i_\calS(\al,\eta)=\ell_{d}(\eta)>0$ for $\eta\neq 0$ we get $1=\frac{i_\calS(\al,\beta)i_\calS(\al,\gam)}{i_\calS(\al,\beta)i_\calS(\al,\gam)}\geq \lam$, which is the desired contradiction. 
\end{proof}


\subsection{Currents dual to cubulations}

In this subsection, we relate certain cubulations of $\G$ to a particular type of co-geodesic current.

\begin{definition}\label{def.discrcurrent}
    A \emph{discrete co-geodesic current} on $\G$ is a Borel measure on $\calC_{\geq 2}(\G)$ of the form 
    \[\al=\sum_{S\in \calS}{\del_S}\]  
    for a discrete set $\calS \subset \calC(\partial \G)$ consisting of finitely many $\G$-orbits of elements in $\calC(\partial \G)$ homeomorphic to a $(n-2)$-sphere.
\end{definition}
If $\al$ is a discrete co-geodesic current, then $\calS_\al=\calS=\supp(\al)$ is a system of spheres at infinity. Also, from \cite[Thm.~2.8]{sasaki} it follows that if $S$ belongs to $\calS_\al$ for some discrete co-geodesic current $\al$ then $S=\Lam(H)$ for some quasiconvex subgroup $H<\G$. 

The goal of this section is to produce cubulations of $\G$ dual to discrete co-geodesic currents. This is guaranteed by the next proposition.

\begin{proposition}\label{prop.dualcubulation}
    Let $\al$ be a discrete co-geodesic current on $\G$ with support $\calS$. Then there exists a cocompact cubical action of $\G$ on a $\CAT(0)$ cube complex $\wtilde{\calX}$ satisfying the following.
    \begin{enumerate}
        \item There exists a natural $\G$-equivariant bijection between $\calS$ and the set of hyperplanes of $\wtilde\calX$. In particular, every hyperplane stabilizer is quasiconvex and has as limit set a sphere in $\calS$, and every sphere in $\calS$ is the limit set of a hyperplane stabilizer.
        \item Any pseudo metric on $\G$ induced by its action on $\wtilde{\calX}^{(1)}$ with the combinatorial metric is dual to $\al$. That is,
            \begin{equation}\label{eq.dualityformula}
                \ell_{\wtilde{\calX}}[g]=i_{\calS}(\al,\eta_{[g]})
            \end{equation}
        for any $g\in \G$.
        \item The action of $\G$ on $\wtilde{\calX}$ is proper if and only if $\al$ is weakly filling. 
    \end{enumerate}
\end{proposition}

The rest of this subsection is devoted to proving \cref{prop.dualcubulation}, so from now on we fix a discrete co-geodesic current $\al$ with support $\calS$, and let $\{S_1,\dots,S_k\}$ be a complete set of representatives of $\G$-orbits in $\calS$. For each $1\leq i\leq k$ we set 
\[\G_i:=\{g\in \G \colon gS_i=S_i\}.\]
By our assumptions, each $\G_i$ is a quasiconvex subgroup with limit set $S_i$. Similarly, if $S=gS_i\in \calS$ for $g\in \G$ we define $\G_S=g\G_ig^{-1}$.


Our first step is to construct a wallspace structure on $\G$. We fix a finite and symmetric generating set for $\G$ (and hence a), and for $1\leq i \leq k$ and $r>0$, we let $W_i$ be the $r$-neighborhood of $\G_i$ in $\G$. Since $W_i$ is quasiconvex and $\Lam(W_i)=S_i$ separates $\partial \G$ into two components, there exists $\lam>0$ such that we can find $r>0$ large enough for which $\G \bs W_i$ consists of exactly two $\lam$-quasiconnected components $U_i,U_i^*$. We set $\calW_i:=\{W_i \cup U_i,W_i\cup U_i^*\}$ and for $S_i\in \calS$ and $g\in \G$ we also consider $W_S=gW_i$ and $\calW_S=\{W_S\cup U_S,W_S\cup U_S^*\}$ for $\{U_S,U_S^\ast\}=\{gU_i,gU_i^\ast\}$. Note that each $\calW_S$ is only well-defined as an unordered pair since an element in $\G_S$ may exchange $U_S$ and $U_S^\ast$. 
We define the $\G$-equivariant wallspace structure $$\bbW:=\{\calW_S \colon S\in \calS\}$$ on $\G$. Since each $S\in \calS$ can be recovered as the intersection $S=\Lam(U_S)\cap \Lam(U_S^\ast)=\Lam(W_S)$, the identity $\calW_S=\calW_{S'}$ implies $S=S'$.

We let $\wtilde{\calX}$ be the $\CAT(0)$ cube complex dual to $\bbW$, according to \cref{subsec.CAT(0)cc}. The natural action of $\G$ on $\wtilde{\calX}$ is cocompact \cite[Thm.~3.1]{sageev.cod} and there is a $\G$-equivariant 1-to-1 correspondence between $\bbW$ and $\calS$. We let $\frakh_S\subset \wtilde{\calX}$ denote the hyperplane determined by $S\in \calS$.


\begin{lemma}\label{lem.closetoginfty}
    Let $g\in \G$ be non-torsion and $S\in \calS$ such that $g^{\infty}\notin S$. If $g^{\infty}$ belongs to the interior of the limit set of $U_S$ (rep. $U_S^*$), then for all $s\geq 0$ and all $m$ large enough we have $g^m\in g^iU_S$ (resp. $g^m\in g^iU_S^*$) for all $0\leq i \leq s$.   
\end{lemma}
\begin{proof}
    If $g^\infty \in \Int\Lam(U_S)=\Lam(U_S)\bs S$, then $g^\infty\in g^i\Int\Lam(U_S)$ for all $i$. In particular, for any fixed $s$, all the points in $\G$ close enough to $g^\infty$ belong to $\bigcap_{0\leq i \leq s}{g^{i}U_S}$, which is true of $g^m$ for $m$ large enough. 
\end{proof}

\begin{lemma}\label{lem.criterionloxo}
    Let $g\in \G$ be a non-torsion element. If $i_{\calS}(\al,\eta_{[g]})>0$ then $\ell_{\wtilde{\calX}}[g]>0$.
\end{lemma}
\begin{proof}
    Suppose that $i_{\calS}(\al,\eta_{[g]})>0$. This means that there exists some $S\in \calS$ separating $g^{-\infty}$ and $g^\infty$ in $\partial \G$. So, without loss of generality assume that $g^{\infty}\in \Int\Lam(U_S)$ and $g^{-\infty}\in \Int\Lam(U_S^*)$. From this, let $k_0$ be such that $g^k\in U_S$ and $g^{-k}\in U_S^*$ for $k\geq k_0$. In particular, for any fixed $k\geq k_0$ we have
    \[g^{-k_0}\in g^rU_S^*  \ \ \text{ and } \ \ g^k\in g^rU_S\]
    for all $0\leq r\leq k-k_0$.

    Now, let $v\in \wtilde{\calX}$ be any vertex in the canonical cube corresponding to the identity $o\in \G$. Then the hyperplane $g^r\frakh_S=\frakh_{g^rS}$ separates $g^{-k_0}v$ and $g^kv$ for all $k\geq k_0$ and $0\leq r \leq k-k_0$. Since $S$ is not fixed by any non-trivial power of  $g$ (otherwise $g^{\pm\infty}\in \Lam(\G_S)=S$), all the hyperplanes $g^r\frakh_S$ are different, and hence
    \begin{equation*}
        \ell_{\wtilde{\calX}}[g]=\lim_{k \to \infty}{\frac{d_{\wtilde{\calX}}(g^{-k_0}v,g^kv)}{k}}\geq \lim_{k \to \infty}{\frac{k-k_0+1}{k}}=1. \qedhere
    \end{equation*}
\end{proof}

\begin{lemma}\label{lem.Hseparates}
    Let $g\in \G$ be an element acting loxodromically on $\wtilde{\calX}$ (i.e. $\ell_{\wtilde\calX}[g]>0$), and let $\sig\subset \wtilde{\calX}^{(1)}$ be a $\genby{g}$-equivariant combinatorial geodesic. Then 
    \[\{S\in \calS \colon S \text{ separates }g^{-\infty},g^{\infty}\}=\{S \in \calS \colon \frakh_S \text{ crosses }\sig\}.\]
\end{lemma}

\begin{proof}
    Suppose that $S$ separates $g^{\pm \infty}$, and without loss of generality assume that $g^\infty\in \Int\Lam(U_S)$ and $g^{-\infty}\in \Int\Lam(U_S^*)$. Let $v\in \sig$ be any vertex and let $R=d_{\wtilde{\calX}}(v,v_o)$, where $v_o$ is any vertex in the canonical cube corresponding to the identity element $o$. By \cref{lem.closetoginfty} applied to $s\geq 2R$, we can find $m_0$ so that if $m\geq m_0$ then $g^m\in g^iU_S$ and $g^{-m}\in g^iU_S^\ast$ for $0\leq i \leq s$. Since $v$ and $v_o$, as orientations of $\bbW$ differ in exactly $R$ walls, by our choice of $s$ we can find some $0\leq i \leq s$ such that $g^mv(g^i\calW_S)=g^i(W_S\cup U_S)$ and $g^{-m}v(g^i\calW_S)=g^i(W_S\cup U_S^\ast)$ for all $m\geq m_0$. In particular, the hyperplane $g^i\frakh_S$ separates $g^mv$ and $g^{-m}v$. Then $g^i\frakh_S$ crosses $\sig$, and hence $\frakh_S$ also crosses $\sig$ by the $\genby{g}$-invariance of $\sig$. 

    Conversely, suppose that $\frakh_S$ crosses $\sigma$. After possibly replacing $g$ by a power, we can assume that $g$ preserves the orientation on $\sig$. We fix a vertex $v\in \sig$ such that $\frakh_S$ separates $v$ and $gv$, and let $N(\frakh_S)\subset \wtilde{\calX}$ be the carrier of $\frakh_S$. 
    
    We first claim that 
    \begin{equation}\label{eq.divergence0}
        d_{\wtilde{\calX}}(N(\frakh_S),g^mv)\to \infty \ \ \text{ as } \ \ m\to \pm\infty.
    \end{equation} Otherwise, since $N(\frakh_S)$ is convex and all the points $g^mv$ belong to $\sig$, after possibly replacing $g$ by $g^{-1}$ we can assume that $d_{\wtilde\calX}(N(\frakh_S),g^mv)=R$ is constant for all $m$ large enough. In particular, there exists some $m_0$ such that the hyperplanes $g^m\frakh_S$ and $g^{m'}\frakh_S$ are transverse for all $m,m'\geq m_0$. But $\G$ acts cocompactly on $\wtilde{\calX}$ and $\wtilde\calX$ is finite-dimensional, so for $m,m'$ large enough these transversalities are impossible by \cite[Lem.~13.13]{haglund-wise.special}.

    Now we use our claim to prove that \begin{equation}\label{eq.divergence1}
        d(W_S,g^m)\to \infty \ \ \text{ as } \ \ m \to \pm \infty,
    \end{equation} where $d$ is some word metric on $\G$ with respect to a finite generating set. Indeed, $\G_S$, being the stabilizer of $\frakh_S$ on $\G$, acts cocompactly on $\frakh_S$ and hence on $N(\frakh_S)$ \cite[Exercise~1.6]{sageev.cat0}. Similarly, $\G_S$ acts cocompactly on $W_S$, so let $K_1\subset N(\frakh_S)$ and $K_2\subset W_S$ be compact subsets such that $\G_S \cdot K_1=N(\frakh_S)$ and $\G_S\cdot K_2=W_S$. Also, since $d_{\wtilde\calX}$ is $\G$-invariant there exists $\lam>0$ such that 
    \[d_{\wtilde\calX}(h_1v,h_2v)\leq \lam d(h_1,h_2)+\lam\]
for all $h_1,h_2\in \G$. Combining these facts, for all $h\in \G$ we deduce
\begin{align*}
 d_{\wtilde\calX}(N(\frakh_S),hv)& = d_{\wtilde\calX}(\G_S \cdot K_1,hv)\\ & \leq  d_{\wtilde\calX}(\G_S \cdot v,hv) + d_{\wtilde\calX}(K_1,v)  \\
 & \leq  \lam d(\G_S,h) + \lam+d_{\wtilde\calX}(K_1,v) \\
 & \leq  \lam d(W_S,h)+\lam d(K_2,o) + \lam+d_{\wtilde\calX}(K_1,v),
\end{align*} and the conclusion follows from our previous claim. 

From \eqref{eq.divergence1} we have that  $g^{\pm \infty}\notin S=\Lam(W_S)$. Also, if $v_o\in \wtilde\calX$ is any vertex in the canonical cube corresponding to $o$, then \eqref{eq.divergence0} implies that $g^mv_o$ and $g^{-m}v_o$ lie in different halfspaces determined by $\frakh_S$ for $m$ large enough. Therefore, without loss of generality we can assume $g^m\in U_S$ and $g^{-m}\in U_S^\ast$ for all $m$ large enough. This means that $g^\infty\in \Lam(U_S)\bs S$ and $g^{-\infty}\in \Lam(U^\ast_S)\bs S$, and $S$ separates $g^\infty$ and $g^{-\infty}$, as desired.   
\end{proof}

\begin{proof}[Proof of \cref{prop.dualcubulation}]
We must verify that the action of $\G$ on $\wtilde{\calX}$ satisfies all the items of the statement, where Item (1) follows by the construction of $\bbW$ and $\wtilde{\calX}$.

To prove Item (2), let $g\in \G$ be a non-torsion element, which we can assume is primitive. Our claim is that
 \[i_\calS(\al,\eta_{[g]})=\ell_{\wtilde{\calX}}[g].\] 
By \cref{lem.criterionloxo}, the assertion is true if $\ell_{\wtilde{\calX}}[g]=0$, so we can assume that $g$ acts loxodromically on $\wtilde{\calX}$. We claim that there exists $N$ such that $g^N$ acts \emph{stably and without inversions} on $\wtilde{\calX}$, meaning that $g^N$ does not exchange the halfspaces determined by any hyperplane of $\wtilde\calX$ \cite[Def.~4.1]{haglund}. If not, then we can find an infinite sequence $(S_{(j)})_j$ of hyperspheres in $\calS$ and a sequence $(r_j)_j$ of positive integers such that $g^{(2r_1)\cdots (2r_{j-1})r_j}$ exchanges the halfspaces determined by $\frakh_{S_{(j)}}$ for each $j$. Note that all the hyperspheres $S_{(j)}$ are distinct. Since $g^{(2r_1)\cdots (2r_{j-1})(2r_j)}$ preserves $\frakh_{S_{(j)}}$, we have $g^{(2r_1)\cdots (2r_{j-1})(2r_j)}\in \G_{S_{(j)}}$ for each $j$, and indeed 
\begin{equation}\label{eq.height}
    g^{(2r_1)\cdots (2r_{j-1})(2r_j)}\in \G_{S_{(1)}}\cap \cdots \cap \G_{S_{(j)}}
\end{equation}
for each $j$. But $\G_{S_{(j)}}\neq \G_{S_{(j')}}$ for $j\neq j'$ (because $S_{(j)}\neq S_{(j')}$) and each $\G_{S_{(j)}}$ is the conjugate of a quasiconvex subgroup $\G_{j}$ for $1\leq i \leq k$, so for $j$ large enough the inclusion in \eqref{eq.height} contradicts the fact that $\{\G_1,\dots,\G_k\}$ has finite \emph{height} (see e.g.~\cite[Main Theorem]{GMRS}). Therefore, by \cite[Thm.~1.5]{haglund} there exists a bi-infinite combinatorial geodesic $\sig \subset \wtilde{\calX}^{(1)}$ that is $\genby{g^N}$-invariant. 

Let $\calF$ be a fundamental domain for the action of $\G$ on $\calI_\calS$ (recall \cref{def.cogeodesic}), and assume that if $(S,(hg^{\pm\infty},hg^{\mp \infty}))\in \calF$ for some $h\in \G$, then indeed $h(g^{\pm\infty},g^{\mp\infty})=(g^{\pm \infty},g^{\mp \infty})$. This implies $h\in \genby{g}$ by \cref{remark.primitive}.

Let $$\calA:=\{(S,(g^\infty,g^{-\infty}))\in \calF \colon S\in \calS, S \text{ separates }g^{-\infty},g^{\infty}\}$$ and $$\calB_r:=\{\genby{g^r}-\text{orbits of }S\in \calS \colon S \text{ separates }g^{-\infty},g^{\infty}\}$$    
for $r>0$, and note that 
    \begin{align*}
        i_\calS(\al,\eta_{[g]}) & =\frac{1}{2}\# \{(S,(hg^{\pm \infty},hg^{\mp \infty}))\in \calF \colon S\in \calS, S \text{ separates }hg^{-\infty},hg^{\infty}\} \\
        & = \# \{(S,(g^\infty,g^{-\infty}))\in \calF \colon S\in \calS, S \text{ separates }g^{-\infty},g^{\infty}\}= \# \calA.
    \end{align*}

    We claim that the map $\Phi: \calA \ra \calB_1$ that sends $(S,(g^\infty,g^{-\infty}))$ to the $\genby{g}$-orbit of $S$ is a bijection, so that $\#\calA =\# \calB_1$. Clearly $\Phi$ is well-defined, and if $S\in \calS$ and $S$ separates $g^{\pm\infty}$, then $h(S,(g^\infty,g^{-\infty}))\in \calF$ for some $h\in \G$, which indeed belongs to $\genby{g}$ by our assumptions on $\calF$. This means that the $\genby{g}$-orbit of $S$ equals $\Phi(hS,(g^\infty,g^{-\infty}))$, and hence $\Psi$ is surjective. To show it is injective, let $S,S'\in \calS$ separate $g^{\pm \infty}$ and such that $(S,(g^\infty,g^{-\infty})),(S',(g^\infty,g^{-\infty}))\in \calF$ and $\Psi(S,(g^\infty,g^{-\infty}))=\Psi(S',(g^\infty,g^{-\infty}))$. Then $S'=g^rS$ for some $r$ and $(S',(g^\infty,g^{-\infty}))=g^r(S,(g^\infty,g^{-\infty}))$. Since $\calF$ is a fundamental domain we deduce $r=0$, implying that $(S,(g^\infty,g^{-\infty}))=(S',(g^\infty,g^{-\infty}))$.
    
    
    To conclude the proof Item (2) we will show that $\ell_{\wtilde{\calX}}[g^N]=\#\calB_N=N\#\calB_1$. For the first equality, let $v$ be any vertex of $\sig$. By \cref{lem.Hseparates} we have
    \begin{align*}
        \#\calB_{N} & =\#\{\genby{g^N}-\text{orbits of }S\in \calS \colon S \text{ separates }g^{-\infty}, g^{\infty}\}\\
        & =\#\{\genby{g^N}-\text{orbits of }S \in \calS \colon \frakh_S \text{ crosses }\sig\}\\
        & =\#\{S \in \calS \colon \frakh_S \text{ separates }v\text{ and }g^Nv\}\\
        & =d_{\wtilde{\calX}}(v,g^Nv)=\ell_{\wtilde{\calX}}[g^N],
    \end{align*}
    where in the last equality we used that $v$ belongs to the $\langle g^N\rangle$-invariant geodesic $\sig$. For the second equality, note that if $S\in \calS$ separates $g^{-\infty},g^{\infty}$, then no power of $g$ fixes $S$, and hence the $\genby{g}$-orbit of $S$ is the disjoint union of the $\langle g^N \rangle$-orbits of $S,gS,\dots,g^{N-1}S$. These identities and our claim imply $$\ell_{\wtilde{\calX}}[g]=\frac{1}{N}\ell_{\wtilde{\calX}}[g^N]=\frac{1}{N}\# \calB_N=\# \calB_1=\#\calA=i_\calS(\al,\eta_{[g]}),$$ which proves Item (2).

    Finally, from \cite[Thm.1.5]{haglund} we have that the action of $\G$ on $\wtilde\calX$ is proper if and only if $\ell_{\wtilde\calX}[g]>0$ for all $g\neq o$, which happens if and only if $\al$ is weakly filling by Item (2). This proves Item (3).
\end{proof}

\begin{remark}\label{rmk.uniqueness}
As we see from Item (2) in \cref{prop.dualcubulation}, the translation length function $\ell_{\wtilde\calX}$ is completely determined by the discrete current $\al$, while the cubulation $\wtilde\calX$ is far from canonical. However, under some additional irreducibility assumptions, $\wtilde\calX$ is also determined by $\al$ up to a $\G$-equivariant isomorphism, see \cite{beyrer-fioravanti.1,beyrer-fioravanti.2}.
\end{remark}

\section{The 2-dimensional case}\label{sec.2-dim}

In this section we assume that $n=2$, so that $\G$ is the fundamental group of a closed surface of negative Euler characteristic, which we assume is orientable. As explained in \cref{ex.surfacegroups}, in this case co-geodesic currents are precisely the geodesic currents, and the intersection number introduced in \cref{def.cogeodesic} coincides with the usual intersection number $i:\Curr(\G) \times \Curr(\G) \ra \R$. 

A geodesic current $\eta\in \Curr(\G)$ is \emph{filling} if $i(\eta,\al)>0$ for all non-zero geodesic currents $\al$. This is consistent with \cref{def.filling-dual} (1). We let $\Curr_f(\G)$ (resp. $\bbP\Curr_f(\G)$) be the set of all (resp. projective classes of ) filling geodesic currents. If $\eta\in \Curr_f(\G)$, then there exists a pseudo metric $d_\mu\in \calD_\G$ dual to $\eta$ in the sense of \cref{def.filling-dual} (2), see \cite[Sec.~6.3]{cantrell-reyes.manhattan}. This gives us a metric structure $\rho_{[\eta]}\in \scrD_\G$ that depends only on the projective class $[\eta]\in \bbP\Curr_f(\G)$. Our results from \cref{sec.co-geodesic} imply that the metric structures $\rho_{[\eta]}$ can be approximated by cubulations.

\begin{proposition}\label{prop.generalsurfaces}
    Let $\rho_{[\eta]}\in \scrD_\G$ be a metric structure represented by the filling geodesic current $\eta\in \Curr_f(\G)$. Then there exists a sequence $(\rho_{\wtilde{\calX}_m})_m\subset \scrD_\G$ of metric structures represented by geometric actions of $\G$ on $\CAT(0)$ cube complexes $\wtilde{\calX}_n$ such that 
    \[\Del(\rho_{[\eta]},\rho_{\wtilde{\calX}_m}) \to 0 \text{ as }n \to \infty.
    \]
Moreover, we can choose each $\wtilde{\calX}_m$ to have a single $\G$-orbit of hyperplanes, all of them having infinite cyclic stabilizers. 
\end{proposition}

\begin{proof}
   Suppose $\eta\in \Curr_f(\G)$ is a filling current. Then by \cite[Prop.~2]{bonahon}, there is a sequence $(t_m\eta_{[g_m]})_m$ of real multiples of rational currents converging to $\eta$. 
   By \cref{prop.contcogeodesic} (1) and by continuity of the usual intersection number, the sequence $(\eta_{[g_m]})_m$ must be eventually filling. Hence for $m$ large enough, the currents $\eta_{[g_m]}$ determine geometric actions of $\G$ on a $\CAT(0)$ cube complex $\wtilde{\calX}_m$ by \cref{prop.dualcubulation} (2). The corresponding metric structures $\rho_{\wtilde{\calX}_m}$ then converge to $\rho_{[\eta]}$ by \cref{prop.contcogeodesic} (2). It follows from \cref{prop.dualcubulation} (1) that each $\wtilde{\calX}_m$ has a single orbit of hyperplanes, and that the stabilizers of these hyperplanes are commensurable to conjugates of the cyclic subgroup $\langle g_m\rangle$.
\end{proof}

From \cref{prop.generalsurfaces} we deduce \cref{thm.2-dimensions} from the introduction. 




\begin{proof}[Proof of \cref{thm.2-dimensions}] 

Let $M$ be a closed negatively curved Riemannian surface with fundamental group $\G$, and let $\al_M\in \Curr(\G)$ be the associated \emph{Liouville current} \cite{Otal}. This current is dual to the metric structure in $\scrD_\G$ induced by the action of $\G$ on $\wtilde{M}$ by Deck transformations. The current $\al_M$ is filling, and hence the result follows from \cref{prop.generalsurfaces}.
\end{proof}

We also deduce \cref{coro.fillingcurrents=cubulations}.

\begin{proof}[Proof of \cref{coro.fillingcurrents=cubulations}]  
    By \cref{prop.generalsurfaces}, it is enough to prove that $\bbP\Curr_f(\G)$ is a closed subspace of $\scrD_\G$ containing all the metric structures induced by geometric cubulations with cyclic hyperplane stabilizers. 

    To see that $\bbP\Curr_f(\G)$ is closed, let $(\eta_m)_m\in \Curr_f(\G)$ be a sequence of filling currents with $\rho_{[\eta_m]}$ converging to $\rho$ in $\scrD_\G$. Up to taking a subsequence, we can assume that $[\eta_m]\to [\eta_\infty]$ in $\bbP\Curr(\G)$, and hence \cref{prop.contcogeodesic} and the continuity of the intersection number imply that $\rho=\rho_{[\eta_\infty]}$. \cref{rmk.filling=proper} then implies that $\eta_{\infty}$ is filling and $\rho\in \bbP\Curr_f(\G)$. 

    Finally, let $\G$ act geometrically on the $\CAT(0)$ cube complex $\wtilde\calX$ with cyclic hyperplane stabilizers. If $\frakh\subset \wtilde\calX$ is a hyperplane, we let $g_{\frakh}\in \G$ be a primitive element such that $\langle g_\frakh\rangle$ and the stabilizer of $\frakh$ are commensurable. Note that we may have $\langle g_\frakh \rangle=\langle g_{\frakh'}\rangle$ for $\frakh \neq \frakh'$. We define the geodesic current 
    \[\alpha_{\wtilde\calX}:=\sum_{\frakh}{(\del_{\{g_{\frakh}^\infty,g_{\frakh}^{-\infty}\}}+\del_{\{g_{\frakh}^{-\infty},g_{\frakh}^{\infty}\}})}\] on $\calG$, 
    where the sum runs over all the hyperplanes $\frakh$ of $\wtilde\calX$. As in the proof of \cref{prop.dualcubulation} it can be verified that $i(\al_{\wtilde\calX},\eta_{[g]})=\ell_{\wtilde\calX}[g]$ for all $g\in \G$, and hence $\rho_{\wtilde\calX}=\rho_{[\al_{\wtilde{\calX}}]}\in \bbP\Curr_f(\G)$.
\end{proof}

We now survey the many classes of metric structures on a surface group that can be represented by a filling geodesic current, and hence satisfy the hypotheses for \cref{prop.generalsurfaces}.



\begin{example}[Non-positively curved metrics on surfaces]\label{ex.nonpossurfaces}
  Let $M$ be a closed surface with fundamental group $\G$. We say that a metric $\frakg$ on $M$ is a \emph{non-positively curved cone metric} if it has a finite set of singularities of cone-type, each of which has an angle strictly larger than $2\pi$, and away from these singularities, the metric is Riemannian and non-positively curved. In \cite{constantine2018marked}, Constantine describes a geodesic current $\al_{\frakg}\in \Curr(\G)$ dual to $\frakg$: the identity $i(\al_\frakg,\eta_{[g]})=\ell_\frakg[g]$ holds for all $g\in \G$, where $\ell_{\frakg}$ is the translation length function associated to the action of $\G$ on the universal cover $\wtilde{M}$ equipped with the pullback metric of $\frakg$. The current $\al_\frakg$ is filling, so that \cref{prop.generalsurfaces} applies and we obtain the following.

  \begin{corollary}\label{coro.NPC}
Let $\frakg$ be a non-positively curved cone metric on the closed surface $M$ with fundamental group $\G$, and let $\wtilde{M}_\frakg$ be the universal cover of $M$ equipped with the pullback metric. Then for any $\lam>1$, there exists a geometric action of $\G$ on a $\CAT(0)$ cube complex $\wtilde{\calX}$ and a $\G$-equivariant $\lam$-quasi-isometry from $\wtilde{\calX}$ to $\wtilde{M}_{\frakg}$.
\end{corollary}
\end{example}

\begin{example}[Positively ratioed representations]\label{ex.positiveratioed}
Let $G$ be a real, connected, non-compact, semisimple, linear Lie group, and let $[P]$ be the conjugacy class of a parabolic subgroup $P <G$. For an arbitrary hyperbolic group $\G$, there is a notion of $[P]$-\emph{Anosov representation} $\pi:\G \ra G$, extending the class of convex-cocompact representations into rank-1 semisimple Lie groups. See for instance \cite[Sec.~4]{kassel.ICM}.

The conjugacy class $[P]$ corresponds to a subset $\theta$ of the set of restricted simple roots $\Del$ of $G$. For a given $[P]$-Anosov representation $\pi:\G \ra G$ and each $\al\in \theta$, one can define a length functional \[\ell^\pi_\al:\G \ra \R_{\geq 0},\] 
that assigns for each $g\in \G$ the logarithm of the diagonal matrix of eigenvalues of $\pi(g)$ composed with $\al + \iota(\al)$, where $\iota$ denotes the opposite involution on $\Del$. 

Suppose now that $\G$ is a hyperbolic surface group. In \cite{martone2019positively}, Martone and Zhang introduced the notion of a \emph{positively ratioed} representation. These are a class of $[P]$-Anosov representations $\pi:\G \ra G$ dual to geodesic currents in the following sense: if $[P]$ corresponds to $\theta \subset \Del$ and $\al\in \theta$, then there exists a geodesic current $\mu^\pi_\al\in \Curr(\G)$ 
such that 
\[i(\mu^\pi_\al,\eta_{[g]})=\ell_{\al}^\pi[g]\]
for all $g\in \G$ \cite[Def.~2.25]{martone2019positively}. As noted in \cite[Lem.~2.13]{derosa-martinezgranado}, the geodesic currents $\mu_\al^\pi$ are non-atomic and have full support, so in particular they are filling. Hence in light of \cref{prop.generalsurfaces}, positively ratioed representations are approximable by cubulations.

Examples of positively ratioed representations are \emph{Hitchin representations} \cite[Sec.~3.1]{martone2019positively}. A representation of $\G$ into $\PSL(m,\R)$ is $m$-\emph{Fuchsian} if it is the composition of a Fuchsian representation of $\G$ into $\PSL(2,\R)$ and an irreducible representation of $\PSL(2,\R)$ into $\PSL(m,\R)$. A representation $\pi :\G \ra \PSL(m,\R)$ is Hitchin if
lies in the same connected component of the character variety as an $m$-Fuchsian representation. Hitchin representations are $[P]$-positively ratioed for $[P]$ corresponding to $\theta=\Del$ \cite[Thm.~3.4]{martone2019positively}. 

In particular, if $\ell_\pi:\G \ra \R$ assigns to $g\in \G$ the sum of the logarithms of the maximal and minimal absolute values of eigenvalues of $\pi(g)$, then there exists a filling current $\mu^\pi\in \Curr(\G)$ dual to $\ell_\pi$. Moreover, $\ell_\pi$ is the stable translation length function of the pseudo metric on $\G$ given by $d^\pi(g,h)= \log(\|\pi(g^{-1}h)\|\|\pi(h^{-1}g)\|)$, for $\|\cdot \|$ an arbitrary norm on $\R^m$. This pseudo metric belongs to $\calD_\G$ (see e.g.~\cite[Lem.~3.14]{cantrell-reyes.manhattan}), and hence induces a metric structure $\rho_\pi\in \scrD_\G$. Applying \cref{prop.generalsurfaces} to $\mu^\pi$, we obtain the following.

\begin{corollary}\label{coro.hitchin}
Let $\pi:\G \ra \PSL(m,\R)$ be a Hitchin representation, and let $\rho_\pi=[d^\pi]\in \scrD_\G$ be the metric structure induced by $\pi$ as above. Then there exists a sequence of metric structures induced by geometric actions of $\G$ on $\CAT(0)$ cube complexes that converges to $\rho_\pi$ in $\scrD_\G$. 
\end{corollary}

Another class of positively ratioed representations is that of \emph{maximal representations}. These representations are studied in the context \emph{Hermitian} Lie groups, which are groups whose corresponding symmetric space admits a complex structure. Such Lie groups include all symplectic groups $\PSp(2m,\R)$. 
Given a surface group representation $\pi:\G \ra G$ into a Hermitian Lie group, we can define the \emph{Toledo invariant}, an integer that is locally constant with respect to deformations \cite[Sec.~3.2]{martone2019positively}. Hence, each component of the character variety is assigned a number. The Toledo invariant is bounded above by the product of the absolute value of the Euler characteristic of the surface group and the rank of the symmetric space associated to $G$, and a representation $\pi:\G \ra G$ is \emph{maximal} when its Toledo invariant achieves this number in absolute value.


If $\pi:\G \ra\PSp(2m,\R)$ is a maximal representation, then it is $[P]$-Anosov and positively ratioed for the appropriate conjugacy class $[P]$ \cite[Cor.~3.11]{martone2019positively}. This implies that the length function $\ell^\pi:\G \ra \R$ given by 
\[\ell^\pi[g]= \log \lam_1(\pi(g))+\cdots+ \log\lam_m(\pi(g))\]
is dual to a geodesic current $\mu^\pi$, where $\lam_1(A),\dots,\lam_m(A)$ denote the $m$ largest absolute values of eigenvalues of $A\in \PSp(2m,\R)$. As in the case of Hitchin representations, from \cite[Lem.~3.14]{cantrell-reyes.manhattan} we deduce that $\ell^\pi$ is the translation length function of a pseudo metric $d^\pi$ in $\calD_\G$, and again we obtain a metric structure $\rho_\pi\in \scrD_\G$. \cref{prop.generalsurfaces} then implies the following.

\begin{corollary}\label{coro.maximal}
    Let $\pi:\G \ra \PSp(2m,\R)$ be a maximal representation, and let $\rho_\pi=[d^\pi]\in \scrD_\G$ be the metric structure induced by $\pi$. Then there exists a sequence of metric structures induced by geometric actions of $\G$ on $\CAT(0)$ cube complexes that converges to $\rho_\pi$ in $\scrD_\G$. 
\end{corollary}
\end{example}

\section{Currents on quasiconformal spheres}\label{sec.currentsQC}

In this section, we assume that $n\geq 3$ and $\G< \SO^+(n,1)$ is a torsion-free uniform lattice. Then $\Hy^n=\SO^+(n,1)/\SO(n)$ and $M=\G \bs \Hy^n$ is a closed hyperbolic $n$-manifold with fundamental group $\G$. Note that $\partial \G=\partial \Hy^n=\bbS^{n-1}$ is the $(n-1)$-dimensional sphere with the standard conformal structure. The goal of this section is to construct systems of spheres at infinity for $\G$ given by quasiconformal hyperspheres in $\bbS^{n-1}$. The main results of the section are \cref{prop.QCKcogeodesic}, which certifies that these systems are co-geodesic, and \cref{prop.criterioncontinuous}, which is a criterion for continuity of the intersection number. We also describe Haar co-geodesic currents, which are dual to the standard lattice action of $\G$ on $\Hy^n$. We will work with $K$-\emph{quasiconformal maps} on $\bbS^{n-1}$ for which we refer the reader to \cite[Ch.~22]{drutu-kapovich}.

\subsection{Co-geodesic currents from quasiconformal hyperspheres}

\begin{definition}
By a \emph{round (hyper)sphere} in $\bbS^{n-1}$, we mean the limit set of a totally geodesic copy of $\Hy^{n-1}$ embedded in $\Hy^n$. Let $\calQC_1=\calQC_1^n$ be the set of all round spheres in $\bbS^{n-1}$. More generally, given $K\geq 1$, a $K$-\emph{quasiconformal (hyper)sphere} in $\bbS^{n-1}$ is the image of a round sphere under a $K$-quasiconformal homeomorphism of $\bbS^{n-1}$. We let $\calQC_K=\calQC_K^n$ be the space of all the $K$-quasiconformal spheres in $\bbS^{n-1}$. 
All these spaces are endowed with the subspace (Hausdorff) topology from $\calC_{\geq 2}(\partial \G)$, so they naturally carry topological $\SO^+(n,1)$-actions. 
\end{definition}

By Gehring’s Theorem \cite[Thm.~22.31]{drutu-kapovich}, the notation $\calQC_1$ is consistent, since 1-quasiconformal maps on $\bbS^{n-1}$ are conformal. Our next lemma asserts that the spaces $\calQC_K$ are systems of spheres at infinity.

\begin{lemma}\label{lem.QCKsystemofspheres}
For every $K\geq 1$, the set $\calQC_K$ is a closed subset of $\calC_{\geq 2}(\partial \G)$. In particular, $\calQC_K$ is a system of spheres at infinity for $\G$.
\end{lemma}

\begin{proof}
Each $S\in \calQC_K$ is homeomorphic to $S^{n-2}$, and clearly $\calQC_K$ is $\G$-invariant, so it is enough to show that $\calQC_K$ is a closed subset of $\calC_{\geq 2}(\partial \G)$. To this end, let $(S_m)_m\subset \calQC_K$ be a sequence of $K$-quasiconformal spheres that converge in the Hausdorff topology to some $S_\infty\in \calC_{\geq 2}(\partial \G)$. Since each $S_m$ is connected, $S_\infty$ is also connected, so that $S_\infty$ is infinite and there exist three different points $a,b,c\in S_\infty$. Let $S_0\subset \bbS^{n-1}$ be a round hypersphere, fix three distinct points $a_0,b_0,c_0$ in $S_0$, and for each $m$ let $\phi_m: \bbS^{n-1} \ra \bbS^{n-1}$ be a $K$-quasiconformal homeomorphism satisfying $S_m=\phi_m(S_0)$. After precomposing with a conformal map preserving $S_0$, we can also assume that $\{\phi_m(a_0),\phi_m(b_0),\phi_m(c_0)\} \to \{a,b,c\}$. Therefore, up to a subsequence, we can assume that $\phi_m$ converges uniformly to a $K$-quasiconformal homeomorphism $\phi_\infty$ (see for instance \cite[Thm.~22.34]{drutu-kapovich}). Equicontinuity of the sequence $(\phi_m)_m$ then implies that $S_\infty$ is the $K$-quasiconformal sphere $\phi_\infty(S_0)$.
\end{proof}

From \cref{lem.Scurrentscompact}, we infer that each $\calQC_K$ is locally compact, separable, and metrizable, and that the space of $\calQC_K\Curr(\G)$ of $\calQC_K$-currents is locally compact and metrizable. 

Recall that $\calI_K=\calI_{\calQC_K} \subset \calQC_K \times \calG$ is the set of all pairs $(S,(p,q))$ such that  $S$ separates $p$ and $q$. We also define $\calJ_K\subset \calQC_K \times \calG$ as the set of all pairs $(S,(p,q))$ such that  either $(S,(p,q))\in \calI_K$ or $p \notin S$ and $q\in S$. Note that $\calJ_K$ is also invariant under the diagonal action of $\G$, and that $\calI_K$ is an open subset of $\calQC_K \times \calG$.

\begin{lemma}\label{lem.JKpropdisc}
The diagonal action of $\Gam$ on $\calJ_K$ is properly discontinuous. In particular, the action of $\G$ on $\calI_K$ is properly discontinuous.
\end{lemma}
\begin{proof}
The action of $\G$ on $\ov{\Hy}^n:=\Hy^n\cup \partial \Hy^n$ is a \emph{convergence action}, meaning that the diagonal action of $\G$ on the space $\Om^3(\ov{\Hy}^n)$ of ordered triples of distinct points in $\ov{\Hy}^n$ is properly discontinuous \cite[Prop.~2.8]{bowditch}. Therefore, it is enough to construct a $\G$-equivariant continuous map $\pi:\calJ_K \to \Om^3(\ov{\Hy}^n)$.

For $(S,(p,q))\in \calJ_K$ we set \[\pi(S,(p,q)):=(\proj_{\Hull(S)}(p),p,q),\] where $\proj_{\Hull(S)}(p))\in \Hy^n$ denotes the nearest point projection of $p$ onto $\Hull(S)$, the convex hull of $S$ in $\Hy^n$. Since $p\neq q$ and $p\notin S$, the map $\pi$ is well-defined and $\G$-equivariant. The continuity of $\pi$ then follows from the continuity of $(S,p)\to \proj_{\Hull(S)}(p)$, which is proven in the next lemma. 
\end{proof}

\begin{lemma}\label{lem.projctns}
For each $K\geq 1$, the map from $\{(S,p) \in \calQC_K \times \partial \Hy^n \colon p\notin S\}$ into $\Hy^n$ given by
$$(S,p) \mapsto \proj_{\Hull(S)}(p)$$
is continuous.
\end{lemma}

\begin{proof}
The projection $\proj_{\Hull(S)}:\partial \Hy^n \bs S \ra \Hull(S)$ can be characterized as follows. Fixed $x_0\in \Hy^n$, for each $p\in \partial \Hy^n \bs S$, the projection $\proj_{\Hull(S)}$ is the unique point $y\in \Hull(S)$ minimizing the Busemann function $x \mapsto \beta(x,x_0;p)$, where
\[\beta(x,x_0;p):=\lim_{x_m\to p}{(d_{\Hy^n}(x,x_m)-d_{\Hy^n}(x_0,x_m))}. \]

Now, consider the sequences $(S_m)_m\subset \calQC_K$ and $(p_m)_m\subset \partial \Hy^n$ converging to $S_\infty\in \calQC_K$ and $p_\infty$ respectively, so that $p_\infty\notin S_\infty$. This implies that $p_m\notin S_m$ for $m$ large enough. This convergence also implies that $(y_m:=\proj_{\Hull(S_m)}(p_m))_m$ is a bounded sequence, so that up to extracting a subsequence and reindexing we can assume that $y_m$ converges to $y_\infty$ as $m \to \infty$. Since $\Hull(S_m)$ converges to $\Hull(S_\infty)$ in the Hausdorff topology restricted to any compact subset of $\Hy^n$ (see e.g. \cite[Thm.~1.4]{bowditch.convex}), we have $y_\infty \in \Hull(S_\infty)$. 

Also, for any $z\in \Hull(S_\infty)$ we can find (after extracting a subsequence), a sequence $(z_m)_m$ with $z_m\in \Hull(S_m)$ for each $m$ and such that $z_m \to z$. The continuity of the Busemann function then gives us
\begin{align*}
    \beta(z,x_0;p_\infty)=\lim_m{\beta(z_m,x_0,p_m)}\geq \lim_m{\beta(y_m,x_0;p_m)}=\beta(y_\infty,x_0;p_\infty),
\end{align*}
and since $z$ was arbitrary we obtain $y_\infty=\proj_{\Hull(S_\infty)}(p_\infty)$.
Te same reasoning works for any subsequence of $(S_m,p_m)_m$, so we deduce that $\proj_{\Hull(S_m)}(p_m)$ converges to $\proj_{\Hull(S_\infty)}(p_\infty)$. 
\end{proof}

\begin{lemma}\label{lem.JKcocompact}
The action of $\G$ on $\calJ_K$ is cocompact. In consequence, the action of $\G$ on $\calJ_K \bs \calI_K$ is cocompact.
\end{lemma}

\begin{proof}
The second assertion follows from the first one since $\calI_K$ is open in $\calJ_K$. For the first assertion, let $D\subset \Hy^n$ be a compact subset such that $\G \cdot D=\Hy^n$, and define 
$$A=\{(S,(p,q))\in \calJ_K \colon \proj_{\Hull(S)}(p) \in D\},$$ where
$\proj_{\Hull(S)}(p)$ is as in \cref{lem.JKpropdisc}. Clearly $\G \cdot A=\calJ_K$, so it is enough to prove that $A$ is compact. 

First, let $B=\{S\in 
\calQC_K \colon \Hull(S)\cap D\neq \emptyset\}.$ A compactness argument about $K$-quasiconformal maps as in the proof of \cref{lem.QCKsystemofspheres} implies that $B$ is compact. Let $(S_m,(p_m,q_m))_m\subset A$ be an arbitrary sequence. Since each $S_m$ belongs to $B$, after taking a subsequence and reindexing we can assume that $p_m\to p, q_m\to q$ and $S_m\to S$ for some $p,q\in \partial \Hy^n$ and $S\in B$. We also have $\proj_{\Hull(S_m)}(p_m)\in D$ for all $m$, and the compactness of $D$ implies that $p\notin S$, and that either $q\in S$ or $p,q$ lie in different components of $\partial \Hy^n \bs  S$. Thus $(S,(p,q))\in \calJ_K$. In addition, from the continuity of the map $(S',p') \mapsto \proj_{\Hull(S')}(p')$ given by \cref{lem.projctns} and the compactness of $D$ we obtain $(S,(p,q))\in A$, which concludes the proof of the first assertion. 
\end{proof}

As $\Gam$ is torsion-free, the action of $\G$ on $\calJ_K$ is free, and hence the quotient $\ov{\calJ}_K:=\G\bs \calJ_K$ is a compact, metrizable space. Therefore, every $\G$-invariant Radon measure $\mu$ on $\calJ_K$ pushes down to a measure $\ov\mu$ on $\ov\calJ_K$. We also define $\ov{\calI}_K=\Gam \bs \calI_K$, which is an open subset of $\ov{\calJ}_K$. Note that $\ov{\calJ}_K \bs \ov{\calI}_K$ is also compact by \cref{lem.JKcocompact}.

If $A \subset \calJ_K$ is as compact set such that $\G \cdot A=\calJ_K$, then for any $(\al,\eta)\in \calQC_K\Curr(\G) \times \Curr(\G)$, the product measure $\al\times \eta$ is Radon and hence $(\al\times \eta)(A)$ is finite. Since $i_K(\al,\eta)\leq \frac{1}{2}\ov{\al\times \eta}(\ov\calJ_K) \leq \frac{1}{2}(\al \times \eta)(A)$, we deduce the following. 

\begin{proposition}\label{prop.QCKcogeodesic}
For any $K\geq 1$ the action of $\G$ on $\calI_K$ is properly discontinuous, and for every $(\al,\eta)\in \calQC_K\Curr(\G)  \times\Curr(\G)$ the intersection number $i_{\calQC_K}(\al,\eta)$ is finite. Therefore, $\calQC_K$ is a co-geodesic system of spheres at infinity for $\G$.
\end{proposition}

\subsection{A criterion for continuity}

In order to apply \cref{prop.contcogeodesic} we must verify that the intersection number is continuous at certain pairs of currents. For our purposes, the following criterion will suffice.

\begin{proposition}\label{prop.criterioncontinuous}
For every $K\geq 1$ and $(\al,\eta)\in \calQC_K\Curr(\G) \times\Curr(\G)$ satisfying 
\begin{equation}\label{eqnmeasu}
(\al\times \eta)(\calJ_K \bs \calI_K)=0,
\end{equation} 
the intersection number $i_{\calQC_K}:\calQC_K\Curr(\G) \times \Curr(\G) \ra \R$ is continuous at $(\al,\eta)$.
\end{proposition}

Throughout this subsection, we consider a pair $(\al,\eta)$ satisfying \eqref{eqnmeasu}, and take sequences $(\al_m)_m \subset \calQC_K\Curr(\G)$ and $(\eta_m)_m\subset \Curr(\G)$ converging to $\al$ and $\eta$, respectively. We also define the measures $\ov{\mu}_m=\ov{\al_m\times \eta_m}$ and $\ov{\mu}=\ov{\al\times \eta}$ on $\ov{\calJ}_K$, and denote the intersection number on $\calQC_K\Curr(\G)\times \Curr(\G)$ by $i_K$. 

\begin{lemma}\label{lem.limitzero}
We have $\lim_{m\to \infty}{\ov\mu_m(\ov{\calJ}_K \bs \ov{\calI}_K)}=0$.
\end{lemma}
\begin{proof}
By \cref{lem.JKcocompact}, choose a compact set $A\subset \calJ_K \bs \calI_K$ such that $\G A=\calJ_K \bs \calI_K$, and note that $(\al\times \eta)(A)=0$ by \eqref{eqnmeasu}. Since $\al_m\times \eta_m \to \al \times \eta$ (see for instance \cite[Thm.~2.8]{billinsley}) we obtain
\begin{equation*}
    \limsup_{m \to \infty}{\ov\mu_m(\ov{\calJ}_K \bs \ov{\calI}_K)}\leq \limsup_{m\to \infty}{(\al_m\times \eta_m)(A)}\leq (\al \times \eta)(A)=0. \qedhere
\end{equation*}
\end{proof}
Let $\pi: \calJ_K \to \ov{\calJ}_K$ be the quotient map. We say that an open set $U \subset \ov{\calJ}_K$ is \emph{nice} if $U=\pi(\wtilde{U})$ for a precompact open set $\wtilde{U}\subset \calJ_K$ such that $g\wtilde{U}\cap \wtilde{U}= \emptyset$ for any non-trivial element $g\in \Gam$. In particular, $(\al_m\times \eta_m)(V)=\ov\mu_m(\pi(V))$ for all $m$ and all $V\subset \wtilde{U}$.

\begin{lemma}\label{lem.convergenceI}
 $\ov\mu_m \to \ov\mu$ weak-$\ast$ on $\ov{\calI}_K$.
\end{lemma}
\begin{proof}
Let $f\in C_c(\ov{\calI}_K)$ be a continuous function with support contained in the compact set $D\subset \ov{\calI}_K$. As $\G$ is torsion-free and acts properly on the locally compact separable metrizable space $\calI_K$, $\ov{\calI}_K$ has a locally finite covering $(U_i)_{i\in I}$ consisting of nice open sets. If $(\rho_i)_{i\in I}$ is a partition of unity subordinate to $(U_i)_{i\in I}$ and $I_0$ is the finite set of all the indices $i$ such that $U_i\cap D \neq \emptyset$, then $f=\sum_{i \in I_0}{\rho_if}$. Let $\wtilde{f_i}\in C_c(\calI_K)\subset C_c(\calQC_K \times \calG)$ be a lift of $\rho_if$ supported on the precompact open set $\wtilde{U}_i$ such that $\pi(\wtilde{U}_i)=U_i$ and $g \wtilde{U}_i\cap \wtilde{U}_i=\emptyset$ for $g\in \G \bs \{o\}$. The convergence $\al_m \times \eta_m \to \al \times \eta$ then implies
\begin{align*}
    \lim_{m \to \infty}{\int{f}{d\ov\mu_m}} & =\lim_{m \to \infty}{\left(\sum_{i\in I_0}{\int{\rho_if}{d\ov\mu_m}}\right)} \\
    & =\sum_{i\in I_0}{\left(\lim_{m \to \infty}{\int{\wtilde{f}_i}{d(\al_m\times\eta_m)}}\right)}=\sum_{i \in I_0}{\int{\wtilde{f}_i}{d(\al\times \eta)}}=\int{f}{d\ov\mu}. \qedhere
\end{align*}
\end{proof}

\begin{corollary}\label{coro.liminfinter}
$\liminf_{m\to \infty}{i_K(\al_m,\eta_m)}\geq i_K(\al,\eta).$
\end{corollary}

\begin{proof}
The space $\ov{\calI}_K$ is locally compact, separable and metrizable, hence $\sigma$-compact and there exists an increasing sequence $U_1\subset U_2\subset U_3 \subset  \dots \subset \ov{\calI}_K$ of precompact open sets such that $\bigcup_i{U_i}=\ov{\calI}_K$. The convergence $\ov\mu_m \to \ov\mu$ on $\ov\calI_K$ given by \cref{lem.convergenceI} implies that $\liminf_m{\ov\mu_m(\ov{\calI}_K)}\geq \liminf_m{\ov\mu_m(U_i)} \geq \ov\mu(U_i)$ for each $i$, and hence
\begin{equation*}
    \liminf_{m \to \infty}{i_K(\al_m,\eta_m)}=\frac{1}{2}\liminf_{m \to \infty}{\ov\mu_m(\ov{\calI}_K)} \geq \frac{1}{2}\sup_i{\ov\mu(U_i)}=\frac{1}{2}\ov\mu(\ov{\calI}_K)=i_K(\al,\eta). \qedhere
\end{equation*}
\end{proof}

\begin{lemma}\label{lem.limsupint}
We have $\limsup_{m\to \infty}{\ov\mu_m(\ov{\calJ}_K)}\leq \ov\mu(\ov{\calJ}_K)$.
\end{lemma}

\begin{proof}
We fix a metric on $\ov{\calJ}_K$ inducing its topology and use the notation $B(x,r)$ for the open ball of radius $r$ around $x\in \ov\calJ_K$, and $S(x,r)$ for the sphere of radius $r$ around $x$. Define $\ov\Del_K=\ov{\calJ}_K \bs \ov{\calI}_K$, and for each $x\in \ov\Del_K$ consider a nice open ball $B(x,t_x)$, so that $\ov\Del_K \subset \bigcup_{x\in \ov\Del_K}{B(x,t_x)}$. By compactness of $\ov\Del_K$ we can find $x_1,\dots,x_q\in \ov\Del_K$ and $t_i=t_{x_i}>0$ such that $\ov\Del_K\subset \bigcup_{1\leq i \leq q}{B(x_i,t_i)}$. Compactness also allows us to find $0<s_i<t_i$ such that $\bigcup_{1\leq i \leq q}{B(x_i,s_i)}$ still contains $\ov\Del_K$. 

For each $1\leq i \leq q$, find a number $r_i$ such that $s_i<r_i<t_i$ and $\ov{\mu}(S(x_i,r_i))=0$, and define $Q_i:=\bigcup_{1\leq j \leq i}{\ov{B(x_j,r_j)}}$, where $\ov{B(x_j,r_j)}$ denotes the closure of $B(x_j,r_j)$ in $\ov\calJ_K$. 

We claim that $\limsup_{m \to \infty}{\ov\mu_m(Q_i)}\leq \ov\mu(Q_i)$ for all $1\leq i \leq q$. The claim holds for $i=1$, since $Q_1$ is a closed subset of the nice open set $B(x_1,t_1)$ and $\al_m\times \eta_m \to \al\times \eta$. Now, assume that the claim holds for some $1\leq i \leq q-1$, and note that \begin{equation}\label{eq.Qunion}
    Q_{i+1}=Q_i \cup (\ov{B(x_{i+1},r_{i+1})} \bs \Int(Q_i)),
\end{equation} where $\Int$ denotes interior. We also have
\[Q_i \cap (\ov{B(x_{i+1},r_{i+1})} \bs \Int(Q_i)) \subset Q_i \bs \Int(Q_i) \subset \bigcup_{1\leq j \leq i}{S(x_j,r_j)},\]
and hence $\ov\mu(Q_i \cap (\ov{B(x_{i+1},r_{i+1})} \bs \Int(Q_i)))=0$ by our choice of the $r_j$'s. By our inductive assumption we have 
\begin{equation}\label{eq.1meas}
    \limsup_m{\ov\mu_m(Q_i)}\leq \ov\mu(Q_i),
\end{equation} and since $Q_i \cap (\ov{B(x_{i+1},r_{i+1})} \bs \Int(Q_i))$ and $\ov{B(x_{i+1},r_{i+1})} \bs \Int(Q_i)$ are closed subsets of the nice open subset $B(x_{i+1},t_{i+1})$, we have \begin{equation}\label{eq.2meas}
    \lim_m{\ov\mu_m(Q_i \cap ( \ov{B(x_{i+1},r_{i+1})} \bs \Int(Q_i)))}\leq \ov\mu(Q_i \cap ( \ov{B(x_{i+1},r_{i+1})} \bs \Int(Q_i)))=0
\end{equation} 
and 
\begin{equation}\label{eq.3meas}
\limsup_m{\ov\mu_m( \ov{B(x_{i+1},r_{i+1})} \bs \Int(Q_i))} \leq \ov\mu(\ov{B(x_{i+1},r_{i+1})} \bs \Int(Q_i)).
\end{equation}
The inequalities \eqref{eq.1meas}, \eqref{eq.2meas} and \eqref{eq.3meas} together with \eqref{eq.Qunion} imply that $$\limsup_m{\ov\mu_m(Q_{i+1})}\leq \ov\mu(Q_{i+1}),$$
which concludes the proof of our claim.

Finally, $\ov\calJ_K=Q_q \cup (\ov\calJ_K\bs \Int(Q_q))$, where $\ov{\calJ}_K\bs \Int(Q_q)$ is a compact subset of $\ov{\calI}_K$, and $Q_q\cap (\ov\calJ_K\bs \Int(Q_q))=(Q_q \bs \Int(Q_q))$ is a closed set of $\ov\mu$-measure 0. From these observations, our claim, and \cref{lem.convergenceI} we conclude 
\begin{align*}
    \limsup_{m \to \infty}{\ov\mu_m(\ov\calJ_K)} & \leq \limsup_{m\to \infty}{\ov\mu_m(Q_q)}+\limsup_{m\to \infty}{\ov\mu_m(\ov\calJ_K\bs \Int(Q_q))} \\ & \leq \ov\mu(Q_q)+\ov\mu(\ov\calJ_K\bs \Int(Q_q))=\ov\mu(\ov\calJ_K).
\end{align*}
This finishes the proof of the lemma.
\end{proof}

\begin{proof}[Proof of \cref{prop.criterioncontinuous}]
Let $(\al,\eta)\in \calQC_K\Curr(\G)\times \Curr(\G)$ be such that $(\al \times \eta)(\calJ_K \bs \calI_K)=0$ and consider a sequence $(\al_m,\eta_m)_m\subset \calQC_K\Curr(\G)\times \Curr(\G)$ converging to $(\al,\eta)$. By \cref{coro.liminfinter} we have $\liminf_{m\to \infty}{i_K(\al_m,\eta_m)}\geq i_K(\al,\eta),$ and by \cref{lem.limsupint} we have 
\begin{align*}
\limsup_{m \to \infty}{i_K(\al_m,\eta_m)} & = \frac{1}{2}\limsup_{m \to \infty}{\ov{\mu}_m(\ov\calI_K)}\\
& \leq  \frac{1}{2}\limsup_{m \to \infty}{\ov{\mu}_m(\ov\calJ_K)}\\
& \leq  \frac{1}{2}\ov{\mu}(\ov\calJ_K)\\
& =  \frac{1}{2}\ov{\al \times \eta}(\ov\calI_K)=i_K(\al,\eta).
\end{align*}
This implies that the limit $\lim_{m\to \infty}{i_K(\al_m,\eta_m)}$ exists and is equal to $i_K(\al,\eta)$.
\end{proof}

\subsection{The Haar co-geodesic current}\label{subsec.haarcurrent}
In this subsection, we describe a canonical (projective) co-geodesic current on $\G$ supported on round hyperspheres that is dual to the lattice action of $\G$ on $\Hy^n$. The topological action of $\SO^+(n,1)$ on $\calQC_1$ is transitive, and since the stabilizer of any round hypersphere is conjugate to $\SO^+(n-1,1)$, we obtain a $\SO^+(n,1)$-equivariant homeomorphism between $\calQC_1$ and the homogeneous space $\SO^+(n,1)/\SO^+(n-1,1)$. Any Haar measure on $\SO^+(n,1)$ then induces a $\SO^+(n,1)$-invariant Radon measure on $\SO^+(n,1)/\SO^+(n-1,1)$, and hence on $\calQC_1$ \cite[Sec.~6.2]{drutu-kapovich}.

\begin{definition}[Haar current]
    By a \emph{Haar co-geodesic current} on $\calQC_1\Curr(\G)$ we mean any co-geodesic current $\al_{Haar}$ induced by a Haar measure on the homogeneous space $\SO^+(n,1)/\SO^+(n-1,1)$ as describe above. The \emph{projective Haar co-geodesic current} is the equivalence class $[\al_{Haar}]\in \bbP\calQC_1\Curr(\G)$, which is independent of the choice of $\al_{Haar}$.
\end{definition}

Crofton's formula \cite[Prop.~2.1]{robertson} implies that the Haar current is dual to the lattice action on $\Hy^n$.

\begin{lemma}\label{lem.duallattice}
    The Haar projective co-geodesic current is dual to the metric structure $\rho_{\Hy^n}\in \scrD_\G$ induced by the action of $\G$ on $\Hy^n$. That is, there is some $t>0$ such that
    \[\ell_{\Hy^n}[g]=t\cdot i_{\calQC_1}(\al_{Haar},\eta_{[g]})\]
    for all $g\in \G$.
\end{lemma}

The Haar current will be our limit co-geodesic current when we apply \cref{prop.contcogeodesic}, so we must verify that the intersection number is continuous at this current when paired with any geodesic current. This is the content of the next lemma.

\begin{lemma}\label{lem.Haarintersectionctns}
For every $\eta\in \Curr(\G)$ and $K\geq 1$, the intersection number $i_K:\calQC_K\Curr(\G) \times \Curr(\G) \ra \R$ is continuous at $(\al_{Haar},\eta)$.
\end{lemma}
\begin{proof}
We want to apply \cref{prop.criterioncontinuous}, so it is enough to show that $(\al_{Haar}\times \eta) (\calJ_K \bs \calI_K)=0$ for any $K\geq 1$ and and $\eta\in \Curr(\G)$. Indeed, since $\al_{Haar}$ is supported on $\calQC_1$ we can assume that $K=1$. Therefore, by Fubini's Theorem we get
\begin{equation*}
 (\al_{Haar}\times \eta) (\calJ_K \bs \calI_K) =  \int_{\calG}{\al_{Haar}(\{S\in \calQC_1 \colon p\notin S, q\in S\})}{d\eta(p,q)}.  
\end{equation*}
Also, for any $q\in \bbS^{n-1}$ the set $\{S\in \calQC_1 \colon p\in S\}$ is a smooth submanifold of $\calQC_1$ of positive codimension, and since $\al_{Haar}$ is a Riemannian volume, we have 
\[\al_{Haar}(\{S\in \calQC_1 \colon p\notin S, q\in S\})\leq \al_{Haar}(\{S\in \calQC_1 \colon q\in S\})=0\]
for all $p,q$. This implies that $ (\al_{Haar}\times \eta) (\calJ_K \bs \calI_K)=0$ and concludes the proof. 
\end{proof}


\section{The arithmetic case}\label{sec.arithmetic}

In this section, we prove \cref{thm.main.arithmetic}, which states that a torsion-free uniform arithmetic lattice of simplest type is approximable by cubulations. Throughout this section, we assume $n\geq 3$ and let $M$ be a closed arithmetic $n$-manifold of simplest type. Then $M=\G\bs \Hy^n$ for $\G$ a torsion-free cocompact lattice in $G:=\SO^+(n,1)$. Recall that $\calQC_1$ is the space of round hyperspheres in $\bbS^{n-1}$, which is naturally identified with $G/L$ for $L:=\SO^+(n-1,1)$. Under this identification, points in $\calQC_1$ also correspond to subsets of $G$. We let $\pi:G \ra \G \bs G$ be the natural projection, and note that $\G \bs G$ is compact. We also let $G$ act on $\G \bs G$ by right translations. We will need some notions of homogeneous dynamics and arithmetic lattices, for which we refer the reader to \cite{zimmer}.

Since $M$ is arithmetic of simplest type, it contains infinitely many immersed totally geodesic codimension-1 submanifolds, see e.g.~\cite[Sec.~2]{BHW}. By a dense commensurator argument, we can prove the following.

\begin{lemma}\label{lem.seqarith}
    There exists a sequence $(S_m)_m\subset \calQC_1$ converging to $S_\infty\in \calQC_1$ and such that:
    \begin{itemize}
        \item The stabilizer of $S_m$ in $G$ acts cocompactly on $S_m$;
        \item $\pi(S_m)$ is closed in $\G \bs G$ for each $m$; and,
        \item $\pi(S_\infty)$ is dense in $\G \bs G$.
    \end{itemize}
\end{lemma}

\begin{proof}
Let $\calL\subset \calQC_1$ be the subset of all the round hyperspheres that are limit sets of a (totally geodesic) quasiconvex subgroup of $\G$, which is non-empty since $M$ is compact and arithmetic of simplest type. We claim that $\calL$ is dense in $\calQC_1$. Indeed, let $S$ be any hypersphere in $\calL$, and assume it is the limit set of the quasiconvex subgroup $\G_S<\G$. By arithmeticity, the commensurator $\Comm_G(\G)$ is dense in $G$ \cite[Prop.~6.2.4]{zimmer}, and since $G$ acts transitively on $\calQC_1$, for any arbitrary $S_\infty \in \calQC_1$ we can find a sequence $(g_m)_m$ in $\Comm_G(\G)$ with $S_m:=g_m S$ converging to $S_\infty$. Now we note that $S_m$ is the limit set of the (totally geodesic) quasiconvex subgroup $\G_m:=\G \cap g_m\G_Sg_m^{-1}<\G$, so that each $S_m$ belongs to $\calL$. 

Since $\calL$ is countable, there is a sequence $(S_m)_m$ in $\calL$ converging to $S_\infty$ in $\calQC_1 \bs \calL$. We claim that this sequence satisfies the conclusion of the lemma. The first two assertions follow by the definition of $\calL$, because each $S_m$ is the limit set of a totally geodesic hypersurface projecting to an immersed submanifold that is closed in $M=\G \bs G /\SO(n)$. To prove the third assertion, we note that the stabilizer in $\G$ of $S_\infty$ is not a lattice in the stabilizer in $G$ of $S_\infty$ (because $S_\infty \notin \calL$). Then the projection in $M$ of the totally geodesic hypersurface with boundary $S_\infty$ is not closed, and similarly $\pi(S_\infty)$ is not closed in $G \bs \G$. Then $\pi(S_\infty)$ is dense in $\G \bs G$ by \cite[Thm.~B]{shah1991closures}. 
\end{proof}



A Borel probability measure $\mu$ on $\G \bs G$ is \emph{homogeneous} if there is a closed, connected subgroup $V<G$ and a point $\ov x=\pi(x)\in \G \bs G$ such that $\ov x \cdot V$ is closed in $\G \bs G$ and $\mu$ is (right) $V$-invariant, with support $\ov x \cdot V$. The action of $V$ on $\supp(\mu)$ is transitive, so it is a homogeneous $V$-space and hence $\mu$ satisfies
\begin{equation}\label{eq.defhomogeneous}
    \int_{G}{F}{d\wtilde{\mu}}=\int_V{F(x y)}{dH_V(y)}
\end{equation}
for all compactly supported continuous functions $F\in C_c(G)$, where $\wtilde{\mu}$ is the $\G$-equivariant lift of $\mu$ to $G$ (with respect to the left action of $\G$) and $dH_V$ is a Haar measure on $V$. The measure $H_V$ is uniquely determined by \eqref{eq.defhomogeneous} since $\mu$ is a probability measure.

Note that if $\mu$ is homogeneous, then any subset of full $\mu$-measure is dense in $\supp(\mu)$. We also note that if $\mu$ is a $\G$-invariant probability measure on $\G\bs G$, then it is the homogeneous measure induced by a Haar measure on $G$.

If $S_m$ is as in \cref{lem.seqarith}, then $\pi(S_m)$ is closed and equals $\ov{x}_m\cdot L$ for some $\ov{x}_m$ in $\G \bs G$. Then the homogeneous probability measure $\mu=\mu_m$ on $\G \bs G$ with support $\pi(S_m)$ satisfies \eqref{eq.defhomogeneous} for $V=L=\SO^+(n-1,1)$.

\begin{proposition}\label{prop.convHaar}
    Let $(S_m)_m$ and $S_\infty$ be as in \cref{lem.seqarith}, and for each $m$ let $\mu_m$ be the homogeneous probability measure on $\G \bs G$ supported on $\pi(S_m)$. Then $\mu_m$ weak-$\ast$ converges to the homogeneous probability measure supported on $\G \bs G$. 
\end{proposition}

\begin{proof}
    It is enough to show that any convergent subsequence of $(\mu_m)_m$ converges to the homogeneous probability measure supported on $\G \bs G$, so we take a subsequence converging to the probability measure $\mu_\infty$, which we still denote by $(\mu_m)_m$. 

    By \cite[Thm.~2.3]{shah1991unif}, $L$ acts ergodically on each $\pi(S_m)$ with respect to $\mu_m$ (any measurable $L$-invariant set is null or conull). Since $L=\SO(n-1,1)$ is finitely generated by 1-parameter unipotent subgroups \cite[Prop.~1.5.4]{margulis}, we choose such a generating collection $U_1,\dots,U_k$ of $L$. Since each $U_i$ is a closed and non-compact subgroup, Moore's Ergodicity Theorem \cite[Thm.~2.2.6]{zimmer} implies that $\mu_m$ is $U_i$-ergodic, for each $m$ and $i$.

    Let $\ov x\in \supp(\mu_\infty)$, so that $\ov x=\lim_{m}{\ov x_m}$ for $\ov x_m\in \supp(\mu_m)$. We can assume that each $\ov x_m$ is $U_i$-generic \cite[Cor.~1.5]{mozes-shah} for all $1\leq i\leq k$. Indeed, for any $m$ the set of $U_i$-generic points has full $\mu_m$-measure, and hence the set of points that are $U_i$-generic for all $i$ is dense in $\supp(\mu_m)$. Let $g_m\in G$ be such that $\ov x=\ov x_m g_m$, and take this sequence so that $g_m$ converges to the identity in $G$. Then \cite[Thm.~1.1]{mozes-shah} gives us some $m_0$ such that 
    \[\pi(S_m)=\supp(\mu_m)\subset \supp(\mu_\infty)g_m\]
    for all $m\geq m_0$. But $S_m$ converges to $S_\infty$ in $\calQC_1$ and $g_m$ converges to the identity, so we actually deduce 
    \[\pi(S_\infty)\subset \supp(\mu_\infty).\]
    
From \cref{lem.seqarith} we have that $\pi(S_\infty)$ is dense in $\G \bs G$, so that $\mu_\infty$ has full support in $\G \bs G$. \cite[Thm.~1.1]{mozes-shah} also implies that $\G \bs G=\supp(\mu_\infty)=\ov x \cdot \Lam(\mu_\infty)$, where $\Lam(\mu_\infty)$ is the subgroup of $G$ preserving $\mu_\infty$. We deduce that $G=\Lam(\mu_\infty)$, and hence $\mu_\infty$ is $G$-invariant and is the homogeneous probability measure supported on $\G \bs G$. 
\end{proof}

\begin{proof}[Proof of \cref{thm.main.arithmetic}]
    Let $(S_m)_m\subset \calQC_1$ be the sequence given by \cref{lem.seqarith}. Since the stabilizer of each $S_m$ acts cocompactly on it, the sets $\calS_{m}:=\{gS_m : g\in \G\}$ are discrete and $\G$-invariant subsets of $\calQC_1$. For each $m$ we define
    \[\al_m:=\sum_{S\in \calS_m}{\del_S},\]
    which is a discrete co-geodesic current. We claim that the projective $\calQC_1$-currents $[\al_m]$ converge to the projective Haar current $[\al_{Haar}]$ in $\bbP\calQC_1\Curr(\G)$. Note that this proves the theorem, after applying \cref{prop.dualcubulation} to each $\al_m$, as well as \cref{prop.contcogeodesic} and Lemmas \ref{lem.duallattice} and  \ref{lem.Haarintersectionctns}.

    To prove our claim, let $(\mu_m)_m$ be the sequence of homogeneous probability measures on $\G \bs G$ with support $\pi(S_m)$ respectively, which converge to the homogeneous measure $\mu_\infty$ on $\G \bs G$ by \cref{prop.convHaar}. Let $\wtilde\mu_m$ be the lift of $\mu_m$ to $G$ that is $\G$-invariant for the left action. Since $\mu_m$ is $L$-homogeneous, there exists $t_m>0$ such that for all $F\in C_c(G)$ we have
    \begin{equation}\label{eq.mu_m}
        \int{F}{d\wtilde \mu_m}=t_m \cdot \sum_{xL\in \calS_m}{\left( \int_{L}{F(xy)}{dH_L(y)}\right)}, 
    \end{equation}
    where $H_L$ is a fixed Haar measure on $L$. Also, if $\hat F\in C_c(\calQC_1)=C_c(G/L)$, then by \cite[Prop.~2.48]{folland} there exists $F\in C_c(G)$ such that
    \begin{equation}\label{eq.hatf}
        \hat{F}(xL)=\int_{L}{F(xy)}{dH_L(y)}
    \end{equation}
    for all $x\in G$. 
    From \eqref{eq.mu_m} and \eqref{eq.hatf} we deduce that 
    \begin{equation}\label{eq.al_m}
        t_m\int{\hat F}{d\al_m}=t_m \cdot\sum_{xL\in \calS_m}{\hat F(xL)}= t_m  \cdot \sum_{xL\in \calS_m}{\left( \int_{L}{F(xy)}{dH_L(y)}\right)}=\int{F}{d\wtilde \mu_m}
    \end{equation}
    for all $\hat F\in C_c(\calQC_1)$ and all $m$.
    
    Since $\mu_m$ converges to the homogeneous measure supported on $\G \bs G$ by \cref{prop.convHaar}, the measures $\wtilde \mu_m$ weak-$\ast$ converge to a Haar measure $H_G$ on $G$. But then there exists a Haar current $\al_{Haar}\in \calQC_1\Curr(\G)$ such that for all $\hat F\in C_c(\calQC_1)$ we have
    \begin{equation}\label{eq.alHaar}
        \int{F}{dH_G}=\int_{\calQC_1}{\left(\int_{L}{F(xy)}{dH_L(y)}\right)}{d\al_{Haar}(xL)}=\int_{\calQC_1}{\hat F}d\al_{Haar}  
    \end{equation}  
    for $F$ satisfying \eqref{eq.hatf}.
    Combining \eqref{eq.al_m} and \eqref{eq.alHaar} we deduce that $t_m\al_m$ weak-$\ast$ converges to $\al_{Haar}$, which proves the claim and concludes the proof of the theorem. 
\end{proof}

\begin{remark}\label{rmk.singleorbitarith}
    The supports of the currents $\al_m$ in the proof of \cref{thm.main.arithmetic} all consist of a single $\G$-orbit in $\calQC_1$. This implies that all the approximating cubulations have a single $\G$-orbit of hyperplanes. 
\end{remark}


\section{The 3-dimensional case}\label{sec.3-dim}

For this section we assume $n=3$, so that $M=\G \bs \Hy^3$ is a closed hyperbolic 3-manifold. Our goal is to prove \cref{thm.main.3manifold}. The main difficulty is to approximate the Haar current introduced in \cref{subsec.haarcurrent} by discrete currents. To do this, we use the analogous result at the level of measures on the Grassmannian of 2-planes in $M$, and promote this to the convergence of co-geodesic currents by integrating over minimal disks with quasicircles as limit sets. For this, we mainly follow the work of Seppi \cite{seppi}, inspired by \cite[Sec.~2.6]{alassal}.

In this section we replace the term ($K$-)quasiconformal hypersphere by ($K$-)\emph{quasicircle}, and we write $\calQC$ for the union of all the spaces $\calQC_K$ for $K\geq 1$. We let $\Gr_2(M)$ and $\Gr_2(\Hy^3)$ denote the spaces of (unoriented) 2-dimensional tangent planes to $M$ and $\Hy^3$ respectively, and let $d=d_{\Hy^3}$ denote the distance on $\Hy^3$.

\subsection{Currents from quasiFuchsian surfaces}

Let $\Sig\loopra M$ be an immersed smooth surface (which we always assume to be closed and connected) and let $\mu_\Sig$ denote the area measure on $\Sig$ induced by the metric on $M$. Also, if $j_\Sig:\Sig \ra \Gr_2(M)$ maps a point $p\in \Sig$ to its tangent plane $T_p\Sig$, we let $\nu_\Sig$ be the pushforward of $\mu_\Sig$ via $j_\Sig$, normalized to have total mass 1. We also let $\nu_{Haar}$ denote the Haar probability measure on $\Gr_2(M)=\G \bs \hspace{-0.7mm}\PSL(2,\C)/\SO(2)$.

We say that a smooth immersed surface $\Sig \loopra M$ is ($K$-)\emph{quasiFuchsian} if the limit set of any lift of $\Sig$ to $\Hy^3$ is a ($K$-)quasicircle. All quasiFuchsian surfaces are incompressible, with fundamental group quasiconvex in $\G=\pi_1(M)$, and for such surfaces $\Sig$ we define 
$$\calS_\Sig=\{\Lam(\wtilde{\Sig})\in \calQC \colon \wtilde{\Sig} \ra \Hy^3 \text{ is a lift of } \Sig\}.$$ 
The set $\calS_\Sig$ is discrete and $\G$-invariant, so we obtain $\calS_\Sig$-currents by defining
\begin{equation}\label{eq.defcccurrents}
    \al_\Sig:=\sum_{C\in \calS_\Sig}{\del_{C}} \ \text{ and } \ \hat{\al}_\Sig:=\frac{1}{\Area(\Sig)}\al_\Sig.
\end{equation}
Note that $\al_\Sig$ is a discrete co-geodesic current according to \cref{def.discrcurrent}, and that $\calS_\Sig$ is a co-geodesic system of spheres at infinity by \cref{prop.QCKcogeodesic}. Let $\calT$ be the collection of all immersed totally geodesic surfaces in $M$. The main result of the section is the following.

\begin{theorem}\label{thm.3dimcurrtocurr}
    Let $(\Sig_m \loopra M)_m$ be a sequence of $K_m$-quasiFuchsian, closed and connected minimal surfaces with $K_m$ tending to $1$, and suppose that
    \begin{equation}\label{eq.conv.measu.grass}
    \nu_{\Sig_m} \xrightharpoonup{\ast} \nu_\infty=\lam_{Haar}\nu_{Haar}+\sum_{T\in \calT}{\lam_T\nu_T}\end{equation}
    for $\lam_{Haar},\lam_T\geq 0$ such that $\lam_{Haar}+\sum_{T\in \calT}{\lam_T}=1$. Then for any $K>1$ and $m_0$ so that $K_m\leq K$ for all $m\geq m_0$, the sequence of $\calQC_K$-currents $(\hat{\al}_{\Sig_m})_{m\geq m_0}$ weak-$\ast$ converges in $\calQC_K\Curr(\G)$ to
    \[\hat{\al}_\infty:=\lam_{Haar}\hat{\al}_{Haar}+\sum_{T\in \calT}{\lam_T\hat{\al}_T},\]
    where $\hat\al_{Haar}$ is the Haar co-geodesic current on $\G$ satisfying \eqref{eq.Haarcurrent}.
\end{theorem}


By a result of Al Assal \cite[Thm.~1.1]{alassal}, for any measure $\nu_\infty$ as in \eqref{eq.conv.measu.grass} we can find a sequence $(\Sigma_m \loopra M)_m$ of immersed minimal surfaces satisfying all the assumptions of the theorem above and such that $\nu_{\Sig_m} \xrightharpoonup{*} \nu_\infty$. For $\nu_\infty=\nu_{Haar}$, this was first proven by Labourie \cite[Thm.~5.7]{labourie} using possibly disconnected surfaces, and by Lowe and Neves \cite[Prop.~6.1]{lowe-neves} using connected surfaces.

This result is our last step in the proof of the 3-dimensional case of Futer-Wise's \cref{conj:futer.wise}.

\begin{proof}[Proof of \cref{thm.main.3manifold}]
By either \cite[Prop.~6.1]{lowe-neves} or \cite[Thm.~1.1]{alassal}, there exists a sequence $(\Sigma_m \loopra M)_m$ of immersed $K_m$-quasiFuchsian closed and connected minimal surfaces with $K_m$ tending to 1, and so that $\nu_{\Sig_m}$ weak-$\ast$ converges to $\nu_{Haar}$. Fix $K>1$, and after removing the first elements of this sequence and reindexing, by \cref{thm.3dimcurrtocurr} we have that the sequence $(\hat{\al}_{\Sig_m})_m$ of $\calQC_K$-currents converges to $\hat{\al}_{Haar}$. 

For each $m$ the projective current $[\hat \al_{\Sig_m}]$ is represented by the discrete co-geodesic current $\al_{\Sig_m}$, and so by \cref{prop.dualcubulation} we can construct a cocompact cubical action of $\G$ on a $\CAT(0)$ cube complex $\wtilde{\calX}_m$ satisfying all the items (1)-(3) of that proposition. In particular, the metric structure $\rho_{\wtilde{\calX}_m}$ is dual to the projective current $[\hat{\al}_{\Sig_m}]$. The theorem then follows by \cref{lem.duallattice} and \cref{prop.contcogeodesic}, from which we deduce that the action of $\G$ on $\wtilde{\calX}_m$ is proper for all $m$ large enough, and that the metric structures $\rho_{\wtilde{\calX}_m}$ converge to $\rho_{\Hy^3}$ in $(\scrD_\G,\Del)$. This gives us $\G$-equivariant $\lam_m$-quasi-isometries from $\wtilde{\calX}_m$ to $\Hy^3$ with $\lam_m$ converging to 1.
\end{proof}

\begin{remark}\label{rmk.singleorbit3}
By construction, the support of the co-geodesic current $\hat \al_\Sig$ defined in \eqref{eq.defcccurrents} consists of a single $\G$-orbit of quasicircles. Therefore, the approximating cube complexes in the proof of \cref{thm.main.3manifold} have a single $\G$-orbit of hyperplanes.  
\end{remark}

\begin{remark}\label{rmk.approxothercurrents}
    In a forthcoming work of the second-named author and Didac Martinez-Granado \cite{martinezgranado-reyes}, the currents $\hat{\al}_\infty$ as in the conclusion of \cref{thm.3dimcurrtocurr} are shown to be dual to metric structures in $\ov\scrD_\G$ induced by cocompact isometric actions of $\G$.
\end{remark}


\subsection{Asymptotically Fuchsian minimal disks}

For the proof of \cref{thm.3dimcurrtocurr} we require some terminology, as well as some results about minimal disks in $\Hy^3$, following Seppi \cite{seppi}. Let $\wtilde{\Sig}\subset \Hy^3$ be an embedded smooth oriented surface, and for a point $x\in \wtilde{\Sig}$, let $\frakn^{\wtilde{\Sig}}_p:\R \ra \Hy^3$ be the arc-length parametrization of the geodesic normal to $\wtilde{\Sig}$ at $p$, so that $\frakn^{\wtilde{\Sig}}_p(0)=p$ and $\frakn_p^{\wtilde{\Sig}}$ is oriented consistently with the orientation of $\wtilde{\Sig}$. If $\wtilde{\Sigma}=D$ is a totally geodesic disk and $\rho>0$, we let $P^{\pm}$ be the totally geodesic disks such that $\frakn_p^{D}(\pm \rho)\in P^\pm$ respectively and that are orthogonal to $\frakn_p^D$ at these points. We define $U_{D,x,\rho}$ as the closed region of $\Hy^3$ bounded by the disks $P^{\pm}$, which is a neighborhood of $D$. Note that $\Lam(U_{D,x,\rho})$ is an annulus in $\bbS^2=\partial \Hy^3$. 

Given $K\geq 1$, we define $\calU_{D,x,\rho,K}$ as the set of $K$-quasicircles $C$ contained in $\Lam(U_{D,x,\rho})$ and so that the inclusion $C\subset \Lam(U_{D,x,\rho})$ is $\pi_1$-injective. We denote $\calU_{D,x,\rho}=\bigcup_{K}{\calU_{D,x,\rho,K}}$, and note that the sets $(\calU_{D,x,\rho})_{\rho>0}$ form a system of neighborhoods of $\Lam(D)$ in $\calQC$. Also, note the invariance identity $U_{gD,gx,\rho}=g(U_{D,x,\rho})$ valid for any isometry $g$ of $\Hy^3$. For the rest of the section we will work with embedded minimal disks with prescribed quasicircles as limit sets, whose existence was proved by Anderson \cite{anderson}. Indeed, if the quasicircle is close enough to a round circle, such a minimal disk is unique, as was proven by Seppi \cite{seppi}. More precisely, we have the following rephrasing of \cite[Thm.~A \& Prop.~4.1]{seppi}.

\begin{proposition}\label{prop.seppi}
    There exists $K_0>1$ and an increasing homeomorphism $a:[1,K_0]\ra [0,1/2]$ so that the following holds. If $K\leq K_0$ and $C\in \calQC_{K}$ then there exists a unique embedded minimal disk $\wtilde{\Sig}_C \ra \Hy^3$ whose limit set equals $C$. Moreover, $\wtilde{\Sig}_{C}\subset \Hull(C)$ for any such $C$ and for any $p\in C$ we have:
    \begin{itemize}
        \item[i)] the principal curvatures of $\wtilde{\Sig}_C$ at $p$ lie within the interval $(-a(K),a(K))$; and, 
        \item[ii)] if $D$ is the totally geodesic disk tangent to $\wtilde{\Sig}_C$ at $p$, then $C\in \calU_{D,p,a(K)}$ (and hence $\Hull(C)\subset U_{D,p,a(K)}$).
    \end{itemize}    
\end{proposition}

\begin{remark}\label{rmk.foliation}
Item ii) of \cref{prop.seppi} is not explicit in the statements of \cite[Thm.~A \& Prop.~4.1]{seppi} and deserves an explanation. Given $C\in \calQC_K$ with $K$ close enough to 1, using \cite[Prop.~4.1]{seppi} we have that the level sets $\wtilde{\Sig}_{C}^{(t)}=\{ \frakn_{p}^{\wtilde{\Sig}_C}(t)\colon p\in \wtilde{\Sig}_{C}\}$ form a foliation of $\Hy^3$ by smooth surfaces equidistant to $\wtilde{\Sig}_{C}$, see for instance the proof of Theorem 3.3 in \cite{uhlenbeck}. In this case, from the proof of \cite[Prop.~4.1]{seppi} we deduce the existence of a number $r>0$ depending only on $K$ and tending to $0$ as $K$ tends to 1, such that the region $\bigcup_{|s|\leq t}\wtilde{\Sig}_{C}^{(s)}$ is convex in $\Hy^3$ for $|t|\geq r$. In particular, for any $p\in \wtilde\Sig_C$ and $t\geq r$, the disk $\wtilde\Sig_C$ is contained in the closed (and convex) set bounded by the totally geodesic planes tangent to $\wtilde{\Sig}_{C}^{(\pm t)}$ at $\frakn_p^{\wtilde\Sig_C}(\pm t)$ respectively. But this region is precisely $U_{D,p,t}$ for $D$ being the totally geodesic disk tangent to $\wtilde\Sig_C$ at $p$, and hence Item ii) follows since we can always take $t=a(K):=2r$. 
\end{remark}

In the sequel, we fix $K_0$ given by \cref{prop.seppi} and we assume that if a quasicircle is $K$-quasiconformal then $K\leq K_0$. For $C\in \calQC_{K_0}$ we denote by $\wtilde{\Sig}_C$ the unique minimal disk in $\Hy^3$ with limit set $C$. As in the case of immersed surfaces in $M$, for such a quasicircle we construct the Borel (and Radon) measure $\wtilde{\nu}_C$ on $\Gr_2(\Hy^3)$ as follows. If $\wtilde{\mu}_C$ is the area measure on $\wtilde{\Sig}_C$ induced by its inclusion on $\Hy^3$ and $\wtilde{j}_C:\wtilde{\Sig}_C \ra \Gr_2(\Hy^3)$ maps any point $p\in \wtilde{\Sig}_C$ to its tangent plane $T_p\wtilde{\Sig}_C$, we let $\wtilde{\nu}_C$ be the pushforward of $\wtilde{\mu}_C$ via $\wtilde{j}_C$. 

The key step in the proof of \cref{thm.3dimcurrtocurr} is the following proposition.

\begin{proposition}\label{prop.convergenceintegral}
    Let $F\in C_c(\Gr_2(\Hy^3))$ be a compactly supported continuous function, and consider the function $\hat{F}:\calQC_{K_0}\ra \R$ given by
    \[C \mapsto \hat{F}(C)=\int{F}{d\wtilde{\nu}_C}.\]
    Then for any $C\in \calQC_1$ and any sequence $(C_m)_m$ of $K_m$-quasicircles converging to $C$ with $K_m$ tending to $1$, we have that 
    \[\hat{F}(C_m) \ \text{ converges to } \ \hat{F}(C).\]
\end{proposition}

We will combine this proposition with the following lemma.

\begin{lemma}\label{lem.weakconvergencenice}
    Let $X$ be a locally compact metrizable space and let $$X \supset X_1 \supset X_2 \supset X_3 \supset \dots$$ be a decreasing sequence of closed subsets with intersection $X_\infty = \bigcap_{m}{X_m}$. Let $(\al_m)_m$ be a sequence of Radon measures on $X$ so that each $\al_m$ has support contained in $X_m$, and suppose that $\al_m$ weak-$\ast$ converges to the Radon measure $\al_\infty$. Also, let $\hat{F}:X \ra \R$ be a bounded function with precompact support, which in addition satisfies the following: if $(x_m)_m$ is a sequence in $X$ converging to $x$ and each $x_m$ belongs to $X_m$, then $\hat{F}(x_m)$ converges to $\hat{F}(x)$. If $\hat{F}|_{\supp{\al_m}}$ is measurable for all $m$, then
    \[\int{\hat{F}}{d\al_m} \ \text{ converges to } \ \int{\hat{F}}{d\al_\infty}.\]
\end{lemma}

\begin{proof}
    Note that $\hat{F}|_{X_\infty}$ is continuous, so by Tietze extension theorem let $\ov{F}:X \ra \R$ be a continuous function with compact support that agrees with $\hat{F}$ on $X_\infty$, and let $K\subset X$ be a compact set containing the support of both $\hat{F}$ and $\ov{F}$. If we set $K_m:=K \cap X_m$, a compactness argument and our assumptions on $\hat F$ imply that for any $\ep>0$ we can find $m_0\geq 0$ such that $\|\hat{F}-\ov{F}\|_{K_m}:=\sup_{x\in K_m}{|\hat{F}(x)-\ov{F}(x)|}\leq \ep$ for all $m\geq m_0$. We also have
    \begin{align*}
        \left|\int{\hat{F}}{d\al_m}-\int{\ov{F}}{d\al_m}\right|\leq \int_{K_m}{|\hat{F}-\ov{F}|}{d\al_m}\leq \al_m(K_m)\|\hat{F}-\ov{F}\|_{K_m},
    \end{align*}
    and since $\limsup_{m}{\al_m(K_m)}\leq \al_\infty(K)<\infty$ and $\int{\ov{F}}{d\al_m}$ converges to $\int{\ov{F}}{d\al_\infty}$, we deduce that  $\int{\hat{F}}{d\al_m}$ converges to $\int{\ov{F}}{d\al_\infty}$. But this integral equals $\int{\hat{F}}{d\al_\infty}$ since the support of $\al_\infty$ is contained in $X_\infty$.
\end{proof}

Now we see how \cref{prop.convergenceintegral} implies \cref{thm.3dimcurrtocurr}.

\begin{proof}[Proof of \cref{thm.3dimcurrtocurr} assuming \cref{prop.convergenceintegral}]
Define $\hat\al_{Haar}$ as the unique Haar co-geodesic current on $\G$ such that
\begin{equation}\label{eq.Haarcurrent}
\int{F}{d\wtilde\nu_{Haar}}=\int_{\calQC_1}{\left(\int{F}{d\wtilde\nu_C}\right)}{d\hat{\al}_{Haar}(C)}
\end{equation}
for any $F\in C_c(\Hy^3)$, where $\wtilde\nu_{Haar}$ is the $\G$-equivariant lift of $\nu_{Haar}$ to $\Gr_2(\Hy^3)$. 

Let $\wtilde{\Sig}_m$, $K_m$, $\nu_\infty$ and $\hat\al_\infty$ be as in the statement of the theorem, and for fixed $K>1$ we assume that $K_m\leq \min\{K,K_0\}$ for all $m$. It is enough to prove that any subsequence of $(\hat{\al}_{\Sig_m})_m$ has a convergent subsequence converging to $\hat\al_\infty$. We reindex any such sequence to be $(\hat{\al}_{\Sig_m})_m$, and by \cref{lem.Scurrentscompact} and after extracting a further subsequence and reindexing we find a sequence $(t_m)_m$ of positive numbers so that $\al_m:=t_m\hat{\al}_{\Sig_m}$ converges to the non-zero $\calQC_K$-current $\al_\infty$. We further assume that the sequence $(K_m)_m$ is non-increasing.

For each $m$ let $\wtilde{\nu}_m$ be the $\G$-equivariant lift of $\nu_{\Sig_m}$ to $\Gr_2(\Hy^3)$, so that for any $F\in C_c(\Gr_2(\Hy^3))$ we have
\begin{equation}\label{eq.intnu_m}
    \int{F}{d\wtilde\nu_m}=\int_{\calQC_{K}}{\left( \int{F}{d\wtilde{\nu}_C}\right)}{d\hat{\al}_{\Sig_m}(C)}=\int_{\calQC_{K_0}}{\hat{F}(C)}{d\hat{\al}_{\Sig_m}(C)},
\end{equation}
where $\hat{F}:\calQC_{K_0}\ra \R$ is defined as in \cref{prop.convergenceintegral}.

We first show that the sequence $(t_m)_m$ converges, so let $G:\Gr_2(\Hy^3)\ra \R$ be any continuous, nonnegative and compactly supported function satisfying $\sum_{g\in \G}{(G\circ g)}\equiv 1$. If $G$ induces $\hat{G}:\calQC_{K_0}\ra \R$, then for each $m$ we have
\[1=\int{G}{d\wtilde\nu_m}=\int{\hat{G}}{d\hat{\al}_{\Sig_m}}.\] Therefore, by applying \cref{prop.convergenceintegral} to $F=G$ and \cref{lem.weakconvergencenice} to $X=\calQC_{K_0}$ and $X_{m}=\calQC_{K_m}$ (recall that $(K_m)_m$ is non-increasing) we deduce that  
\[t_m=\int{\hat{G}}{d\al_m}  \ \text{ converges to } \ \int{\hat{G}}{d\al_\infty}=:t_\infty.\]

Now, let $\ov{F}\in C_c(\calQC_K)$ be an arbitrary compactly supported continuous function. Since $K_m$ converges to 1, $\al_\infty$ is supported on $\calQC_1$, and $\ov{F}$ restricts to a continuous compactly supported function on $\calQC_1$, 
we can find a continuous function $F\in C_c(\Gr_2(\Hy^3))$ such that
\[\ov{F}(C)=\hat{F}(C)=\int{F}{d\wtilde{\nu}_C} \text{ for all }C\in \calQC_1,\]
see for instance \cite[Prop.~2.48]{folland}.

Suppose that $\nu_{\Sig_m}$ converges to $\nu_\infty=\lam_{Haar}\nu_{Haar}+\sum_{T\in \calT}{\lam_T\nu_T}$ as in the statement of the theorem. Then $\wtilde{\nu}_m$ weak-$\ast$ converges to $\wtilde{\nu}_\infty$ in $\Gr_2(\Hy^3)$, the $\G$-equivariant lift of $\nu_\infty$, implying
\begin{align*}
    \int_{\calQC_K}{\ov{F}}{d\al_\infty}=\int_{\calQC_1}{\ov{F}}{d\al_\infty}=\int_{\calQC_1}{\left(\int{F}{d\wtilde{\nu}_C}\right)}{d\al_\infty(C)}=\int_{\calQC_K}{\left(\int{F}{d\wtilde{\nu}_C}\right)}{d\al_\infty(C)}.
\end{align*}
By applying \cref{prop.convergenceintegral}, \cref{lem.weakconvergencenice} and \eqref{eq.intnu_m} we obtain
\begin{align*}
\int_{\calQC_K}{\left(\int{F}{d\wtilde{\nu}_C}\right)}{d\al_\infty(C)} 
& =\lim_{m\to \infty}{ \int_{\calQC_K}{\left(\int{F}{d\wtilde{\nu}_C} \right)}{ t_md\hat{\al}_{\Sig_m}(C)}} \\
& =\lim_{m\to \infty}{t_m \int_{\calQC_K}{F}{ d\wtilde{\nu}_m}}\\ & =t_\infty \int{F}{d\wtilde{\nu}_\infty}\\ 
& =t_\infty \int{\left(\int{F}{d\wtilde{\nu}_C} \right)}{d\hat{\al}_\infty(C)} =t_\infty \int{\ov{F}}{d\hat{\al}_\infty},
\end{align*}
where in the last equality we used that $\hat{\al}_\infty$ is supported on $\calQC_1$. Then 
\[\int{\ov{F}}{d\al_\infty}=t_\infty \int{\ov{F}}{d\hat{\al}_\infty}\]
for all $\ov{F}\in C_c(\calQC_K)$, which gives us $\al_\infty=t_\infty \hat{\al}_\infty$. But since $\al_\infty$ is non-zero, we conclude that $t_\infty>0$ and $\hat{\al}_{\Sig_m} \xrightharpoonup{\ast} \hat{\al}_\infty$.
\end{proof}


The rest of the section is devoted to the proof of \cref{prop.convergenceintegral}. We will need some lemmas related to the geometry of the sets $U_{D,x,\rho}$, $\calU_{D,x,\rho,K}$ and $\calU_{D,x,\rho}$. These lemmas will be used to control the convergence of minimal disks with limit sets converging to round circles.

\begin{lemma}\label{lem.disjoint}
    For any $\rho,R>0$ there exists $r>0$ satisfying the following. If $D_0\subset \Hy^3$ is a totally geodesic plane with $x_0\in D_0$, $D$ is a totally geodesic plane not contained in $U_{D_0,x_0,\rho}$ and $p\in D$ satisfies $d(p,x_0)\leq R$, then 
    \[\calU_{D,p,r}\cap \calU_{D_0,x_0,\rho/2}=\emptyset.\]
\end{lemma}
\begin{proof}
    The conclusion follows because for $r>0$ small enough (depending only on $\rho$ and $R$) the limit set $\Lam(U_{D,p,r})$ separates at least one of the components of $\partial \Hy^3 \bs \Lam(U_{D_0,x_0,\rho/2})$ into two components, see \cref{fig:enter-label}. Details are left to the reader.
\end{proof}

\begin{figure}[hbt]\label{fig.proof}
\centering

\tikzset{every picture/.style={line width=0.75pt}} 

\begin{tikzpicture}[x=0.75pt,y=0.75pt,yscale=-1,xscale=1]

\draw  [line width=1.5]  (207.58,167.95) .. controls (207.58,101.67) and (261.31,47.95) .. (327.58,47.95) .. controls (393.86,47.95) and (447.58,101.67) .. (447.58,167.95) .. controls (447.58,234.22) and (393.86,287.95) .. (327.58,287.95) .. controls (261.31,287.95) and (207.58,234.22) .. (207.58,167.95) -- cycle ;
\draw  [draw opacity=0][dash pattern={on 5.63pt off 4.5pt}][line width=1.5]  (208.1,167.49) .. controls (209.3,154.17) and (262.33,143.45) .. (327.58,143.45) .. controls (393.58,143.45) and (447.08,154.42) .. (447.08,167.95) .. controls (447.08,168.09) and (447.08,168.24) .. (447.07,168.39) -- (327.58,167.95) -- cycle ; \draw  [color={rgb, 255:red, 208; green, 2; blue, 27 }  ,draw opacity=1 ][dash pattern={on 5.63pt off 4.5pt}][line width=1.5]  (208.1,167.49) .. controls (209.3,154.17) and (262.33,143.45) .. (327.58,143.45) .. controls (393.58,143.45) and (447.08,154.42) .. (447.08,167.95) .. controls (447.08,168.09) and (447.08,168.24) .. (447.07,168.39) ;  
\draw  [draw opacity=0][dash pattern={on 5.63pt off 4.5pt}][line width=1.5]  (243.04,251.88) .. controls (234.41,241.65) and (264.23,196.65) .. (310.26,150.62) .. controls (356.79,104.09) and (402.27,74.13) .. (411.83,83.7) .. controls (412.09,83.95) and (412.32,84.23) .. (412.52,84.54) -- (327.58,167.95) -- cycle ; \draw  [color={rgb, 255:red, 58; green, 110; blue, 172 }  ,draw opacity=1 ][dash pattern={on 5.63pt off 4.5pt}][line width=1.5]  (243.04,251.88) .. controls (234.41,241.65) and (264.23,196.65) .. (310.26,150.62) .. controls (356.79,104.09) and (402.27,74.13) .. (411.83,83.7) .. controls (412.09,83.95) and (412.32,84.23) .. (412.52,84.54) ;  
\draw  [draw opacity=0][dash pattern={on 4.5pt off 4.5pt}] (210.66,144.43) .. controls (210.63,144.29) and (210.62,144.14) .. (210.62,144) .. controls (210.62,133.78) and (263.09,125.5) .. (327.81,125.5) .. controls (391.36,125.5) and (443.1,133.49) .. (444.95,143.45) -- (327.81,144) -- cycle ; \draw  [color={rgb, 255:red, 245; green, 166; blue, 35 }  ,draw opacity=1 ][dash pattern={on 4.5pt off 4.5pt}] (210.66,144.43) .. controls (210.63,144.29) and (210.62,144.14) .. (210.62,144) .. controls (210.62,133.78) and (263.09,125.5) .. (327.81,125.5) .. controls (391.36,125.5) and (443.1,133.49) .. (444.95,143.45) ;  
\draw  [draw opacity=0][dash pattern={on 4.5pt off 4.5pt}] (210.66,181.43) .. controls (210.63,181.29) and (210.62,181.14) .. (210.62,181) .. controls (210.62,170.78) and (263.09,162.5) .. (327.81,162.5) .. controls (391.36,162.5) and (443.1,170.49) .. (444.95,180.45) -- (327.81,181) -- cycle ; \draw  [color={rgb, 255:red, 245; green, 166; blue, 35 }  ,draw opacity=1 ][dash pattern={on 4.5pt off 4.5pt}] (210.66,181.43) .. controls (210.63,181.29) and (210.62,181.14) .. (210.62,181) .. controls (210.62,170.78) and (263.09,162.5) .. (327.81,162.5) .. controls (391.36,162.5) and (443.1,170.49) .. (444.95,180.45) ;  
\draw  [draw opacity=0][dash pattern={on 4.5pt off 4.5pt}] (218.54,122) .. controls (220.1,112.02) and (268.62,104) .. (328.25,104) .. controls (387.88,104) and (436.4,112.02) .. (437.96,122) -- (328.25,122.5) -- cycle ; \draw  [color={rgb, 255:red, 248; green, 231; blue, 28 }  ,draw opacity=1 ][dash pattern={on 4.5pt off 4.5pt}] (218.54,122) .. controls (220.1,112.02) and (268.62,104) .. (328.25,104) .. controls (387.88,104) and (436.4,112.02) .. (437.96,122) ;  
\draw  [draw opacity=0][dash pattern={on 4.5pt off 4.5pt}] (219.54,209.5) .. controls (221.1,199.51) and (269.62,191.5) .. (329.25,191.5) .. controls (388.88,191.5) and (437.4,199.51) .. (438.96,209.5) -- (329.25,210) -- cycle ; \draw  [color={rgb, 255:red, 248; green, 231; blue, 28 }  ,draw opacity=1 ][dash pattern={on 4.5pt off 4.5pt}] (219.54,209.5) .. controls (221.1,199.51) and (269.62,191.5) .. (329.25,191.5) .. controls (388.88,191.5) and (437.4,199.51) .. (438.96,209.5) ;  
\draw  [draw opacity=0][dash pattern={on 4.5pt off 4.5pt}] (235.82,243.66) .. controls (227.7,232.11) and (256.74,187.36) .. (301.99,142.1) .. controls (347.76,96.33) and (393.02,67.15) .. (403.93,76.22) -- (320.41,160.51) -- cycle ; \draw  [color={rgb, 255:red, 57; green, 135; blue, 224 }  ,draw opacity=1 ][dash pattern={on 4.5pt off 4.5pt}] (235.82,243.66) .. controls (227.7,232.11) and (256.74,187.36) .. (301.99,142.1) .. controls (347.76,96.33) and (393.02,67.15) .. (403.93,76.22) ;  
\draw  [draw opacity=0][dash pattern={on 4.5pt off 4.5pt}] (249.82,257.68) .. controls (241.7,246.13) and (270.73,201.38) .. (315.99,156.12) .. controls (361.76,110.35) and (407.01,81.17) .. (417.93,90.24) -- (334.4,174.53) -- cycle ; \draw  [color={rgb, 255:red, 57; green, 135; blue, 224 }  ,draw opacity=1 ][dash pattern={on 4.5pt off 4.5pt}] (249.82,257.68) .. controls (241.7,246.13) and (270.73,201.38) .. (315.99,156.12) .. controls (361.76,110.35) and (407.01,81.17) .. (417.93,90.24) ;  
\draw  [draw opacity=0][fill={rgb, 255:red, 80; green, 227; blue, 194 }  ,fill opacity=0.7 ] (338,192.06) .. controls (293.96,234.67) and (252.43,261.29) .. (243.33,252.2) .. controls (233.77,242.63) and (263.73,197.15) .. (310.26,150.62) .. controls (312.86,148.02) and (315.46,145.47) .. (318.05,142.97) -- (327.58,167.95) -- cycle ; \draw  [draw opacity=0] (338,192.06) .. controls (293.96,234.67) and (252.43,261.29) .. (243.33,252.2) .. controls (233.77,242.63) and (263.73,197.15) .. (310.26,150.62) .. controls (312.86,148.02) and (315.46,145.47) .. (318.05,142.97) ;  
\draw  [color={rgb, 255:red, 0; green, 0; blue, 0 }  ,draw opacity=0 ][fill={rgb, 255:red, 208; green, 2; blue, 27 }  ,fill opacity=0.71 ] (208.1,167.49) .. controls (208.1,153.96) and (261.6,142.99) .. (327.59,142.99) .. controls (393.57,142.99) and (447.07,153.96) .. (447.07,167.49) .. controls (447.07,181.02) and (393.57,191.99) .. (327.59,191.99) .. controls (261.6,191.99) and (208.1,181.02) .. (208.1,167.49) -- cycle ;
\draw  [draw opacity=0][fill={rgb, 255:red, 74; green, 194; blue, 226 }  ,fill opacity=0.7 ] (317.6,143.41) .. controls (361.48,101.04) and (402.77,74.63) .. (411.83,83.7) .. controls (421.4,93.26) and (391.44,138.74) .. (344.91,185.27) .. controls (342.23,187.95) and (339.55,190.57) .. (336.89,193.14) -- (327.58,167.95) -- cycle ; \draw  [draw opacity=0] (317.6,143.41) .. controls (361.48,101.04) and (402.77,74.63) .. (411.83,83.7) .. controls (421.4,93.26) and (391.44,138.74) .. (344.91,185.27) .. controls (342.23,187.95) and (339.55,190.57) .. (336.89,193.14) ;  
\draw  [draw opacity=0][line width=1.5]  (447.58,167.95) .. controls (446.23,181.24) and (393.26,191.93) .. (328.11,191.93) .. controls (262.11,191.93) and (208.61,180.96) .. (208.61,167.43) .. controls (208.61,165.94) and (209.26,164.48) .. (210.49,163.07) -- (328.11,167.43) -- cycle ; \draw  [color={rgb, 255:red, 208; green, 2; blue, 27 }  ,draw opacity=1 ][line width=1.5]  (447.58,167.95) .. controls (446.23,181.24) and (393.26,191.93) .. (328.11,191.93) .. controls (262.11,191.93) and (208.61,180.96) .. (208.61,167.43) .. controls (208.61,165.94) and (209.26,164.48) .. (210.49,163.07) ;  
\draw  [draw opacity=0] (437.96,123) .. controls (436.4,132.98) and (387.88,141) .. (328.25,141) .. controls (268.62,141) and (220.1,132.98) .. (218.54,123) -- (328.25,122.5) -- cycle ; \draw  [color={rgb, 255:red, 248; green, 231; blue, 28 }  ,draw opacity=1 ] (437.96,123) .. controls (436.4,132.98) and (387.88,141) .. (328.25,141) .. controls (268.62,141) and (220.1,132.98) .. (218.54,123) ;  
\draw  [draw opacity=0] (444.39,141.48) .. controls (444.64,141.98) and (444.77,142.48) .. (444.77,142.99) .. controls (444.77,153.21) and (392.31,161.49) .. (327.59,161.49) .. controls (262.86,161.49) and (210.4,153.21) .. (210.4,142.99) .. controls (210.4,142.16) and (210.74,141.34) .. (211.42,140.53) -- (327.59,142.99) -- cycle ; \draw  [color={rgb, 255:red, 245; green, 166; blue, 35 }  ,draw opacity=1 ] (444.39,141.48) .. controls (444.64,141.98) and (444.77,142.48) .. (444.77,142.99) .. controls (444.77,153.21) and (392.31,161.49) .. (327.59,161.49) .. controls (262.86,161.49) and (210.4,153.21) .. (210.4,142.99) .. controls (210.4,142.16) and (210.74,141.34) .. (211.42,140.53) ;  
\draw  [draw opacity=0] (444.39,190.48) .. controls (444.64,190.98) and (444.77,191.48) .. (444.77,191.99) .. controls (444.77,202.21) and (392.31,210.49) .. (327.59,210.49) .. controls (262.86,210.49) and (210.4,202.21) .. (210.4,191.99) .. controls (210.4,191.16) and (210.74,190.34) .. (211.42,189.53) -- (327.59,191.99) -- cycle ; \draw  [color={rgb, 255:red, 245; green, 166; blue, 35 }  ,draw opacity=1 ] (444.39,190.48) .. controls (444.64,190.98) and (444.77,191.48) .. (444.77,191.99) .. controls (444.77,202.21) and (392.31,210.49) .. (327.59,210.49) .. controls (262.86,210.49) and (210.4,202.21) .. (210.4,191.99) .. controls (210.4,191.16) and (210.74,190.34) .. (211.42,189.53) ;  
\draw  [draw opacity=0] (438.96,210.49) .. controls (437.4,220.48) and (388.88,228.5) .. (329.25,228.5) .. controls (269.62,228.5) and (221.1,220.48) .. (219.54,210.49) -- (329.25,210) -- cycle ; \draw  [color={rgb, 255:red, 248; green, 231; blue, 28 }  ,draw opacity=1 ] (438.96,210.49) .. controls (437.4,220.48) and (388.88,228.5) .. (329.25,228.5) .. controls (269.62,228.5) and (221.1,220.48) .. (219.54,210.49) ;  
\draw  [draw opacity=0][line width=1.5]  (412.1,83.99) .. controls (420.82,94.15) and (390.98,139.19) .. (344.91,185.27) .. controls (298.38,231.8) and (252.9,261.76) .. (243.33,252.2) .. controls (243.11,251.98) and (242.91,251.74) .. (242.73,251.48) -- (327.58,167.95) -- cycle ; \draw  [color={rgb, 255:red, 58; green, 110; blue, 172 }  ,draw opacity=1 ][line width=1.5]  (412.1,83.99) .. controls (420.82,94.15) and (390.98,139.19) .. (344.91,185.27) .. controls (298.38,231.8) and (252.9,261.76) .. (243.33,252.2) .. controls (243.11,251.98) and (242.91,251.74) .. (242.73,251.48) ;  
\draw  [draw opacity=0] (403.4,75.83) .. controls (403.74,76.05) and (404.05,76.3) .. (404.33,76.59) .. controls (414.5,86.76) and (385.17,132.58) .. (338.82,178.93) .. controls (292.47,225.28) and (246.65,254.61) .. (236.48,244.44) .. controls (235.85,243.82) and (235.38,243.06) .. (235.05,242.17) -- (320.41,160.51) -- cycle ; \draw  [color={rgb, 255:red, 57; green, 135; blue, 224 }  ,draw opacity=1 ] (403.4,75.83) .. controls (403.74,76.05) and (404.05,76.3) .. (404.33,76.59) .. controls (414.5,86.76) and (385.17,132.58) .. (338.82,178.93) .. controls (292.47,225.28) and (246.65,254.61) .. (236.48,244.44) .. controls (235.85,243.82) and (235.38,243.06) .. (235.05,242.17) ;  
\draw  [draw opacity=0] (417.39,89.84) .. controls (417.73,90.07) and (418.04,90.32) .. (418.33,90.61) .. controls (428.5,100.78) and (399.17,146.59) .. (352.82,192.95) .. controls (306.46,239.3) and (260.64,268.63) .. (250.47,258.46) .. controls (249.85,257.84) and (249.37,257.08) .. (249.04,256.19) -- (334.4,174.53) -- cycle ; \draw  [color={rgb, 255:red, 57; green, 135; blue, 224 }  ,draw opacity=1 ] (417.39,89.84) .. controls (417.73,90.07) and (418.04,90.32) .. (418.33,90.61) .. controls (428.5,100.78) and (399.17,146.59) .. (352.82,192.95) .. controls (306.46,239.3) and (260.64,268.63) .. (250.47,258.46) .. controls (249.85,257.84) and (249.37,257.08) .. (249.04,256.19) ;  
\draw (193.5,169) node  {\includegraphics[width=14.25pt,height=43.5pt]{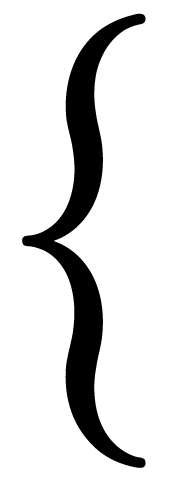}};
\draw (459,167.5) node [xscale=-1] {\includegraphics[width=15pt,height=72.75pt]{leftbracket.png}};

\draw (402.97,72.33) node [anchor=north west][inner sep=0.75pt]  [font=\Huge,rotate=-315]  {$\}$};
\draw (406,161.5) node [anchor=north west][inner sep=0.75pt] [font=\large]    {$D_{0}$};
\draw (371,107.4) node [anchor=north west][inner sep=0.75pt] [font=\large]   {$D$};
\draw (433,61.4) node [anchor=north west][inner sep=0.75pt]    {$\Lambda ( U_{D,p,r})$};
\draw (100,159.5) node [anchor=north west][inner sep=0.75pt]    {$\Lambda ( U_{D_{0} ,x_{0} ,\rho /2})$};
\draw (470,159.5) node [anchor=north west][inner sep=0.75pt]    {$\Lambda ( U_{D_{0} ,x_{0} ,\rho })$};
\draw (321,162.5) node [anchor=north west][inner sep=0.75pt]  [font=\large]  {$\bullet $};
\draw (309,163) node [anchor=north west][inner sep=0.75pt]  [font=\scriptsize]  {$x_{0}$};
\draw (335,148.4) node [anchor=north west][inner sep=0.75pt]  [font=\large]  {$\bullet $};
\draw (345,149.4) node [anchor=north west][inner sep=0.75pt]  [font=\scriptsize]  {$p$};

\end{tikzpicture}

    \caption{Proof of \cref{lem.disjoint}}
    \label{fig:enter-label}
\end{figure}

\begin{lemma}\label{lem.tangentminimalsfc}
    For any $\rho,R>0$ there exist $K'\in (1,K_0]$ and $\eta'>0$ such that the following holds. Let $D_0\subset \Hy^3$ be a totally geodesic plane with $x_0\in D_0$, and let $C\in \calU_{D_0,x_0,\eta',K'}$. If $p\in \wtilde{\Sig}_C$ satisfies $d(p,x_0)\leq R$, then the totally geodesic plane tangent to $\wtilde{\Sig}_C$ at $p$ is contained in $U_{D_0,x_0,\rho}$ and its limit set belongs to $\calU_{D_0,x_0,\rho}$. 
\end{lemma}

\begin{proof}
    Suppose that the conclusion does not hold, and after translating by isometries of $\Hy^3$, assume the existence of a totally geodesic plane $D_0$, a point $x_0\in D_0$, and sequences $K_m\to 1$ and $\eta_m\to 0$ satisfying the following. For each $m$ there exists a quasicircle $C_m\in \calU_{D_0,x_0,\eta_m,K_m}$ and a point $p_m$ in $\wtilde{\Sig}_{C_m}\subset \Hy^3$ such that $d(p_m,x_0)\leq R$ and the totally geodesic disk $D_m$ containing $p_m$ and tangent to $\wtilde{\Sig}_m$ at $p_m$ is not contained in $U_{D_0,x_0,\rho}$. By \cref{lem.disjoint} we know that \begin{equation}\label{eq.disjoint}
        \calU_{D_m,p_m,r}\cap \calU_{D_0,x_0,\rho/2}=\emptyset
    \end{equation} for some $r>0$ independent of $m$. Also, by \cref{prop.seppi} ii) we have that $\Hull(C_m)\subset U_{D_m,p_m,a(K_m)}$, and hence $C_m\in \calU_{D_m,p_m,r_m,K_m}$ for $r_m=a(K_m) \to 0$. Since $\eta_m\to 0$, the sequence $C_m$ converges to $C=\Lam(D_0)$. On the other hand, \eqref{eq.disjoint} implies that $\Lam(D_0)$ does not belong to the interior of $\calU_{D_0,x_0,\rho/2}$, which is our desired contradiction.  
\end{proof}

The next lemma relies on elementary hyperbolic geometry and its proof is left to the reader. 

\begin{lemma}\label{lem.closehyperplanes}
    For any $\del,R>0$ there exists $\rho>0$ so that the following holds. Let $D_0\subset \Hy^3$ be a totally geodesic disk, $x_0$ a point in $D_0$, and $D$ another totally geodesic disk with $\Lam(D)\in \calU_{D_0,x_0,\rho}$. If $p\in D$ satisfies $d(p,x_0)\leq R$ and $\ell\subset \Hy^3$ is the geodesic containing $p$ and orthogonal to $D$, then
    \begin{enumerate}
        \item $\ell$ intersects $D_0$ at a unique point $q$;
        \item The distance $d(p,q)$ is at most $\del$; and, 
        \item If $\ell'\subset D_0$ is a geodesic ray based at $q$, then the angle between $\ell$ and $\ell'$ at $q$ is at least $\pi/2-\del$.
    \end{enumerate}
\end{lemma}


From now on, we fix a round circle $C_\infty\in \calQC_1$ and let $D_\infty\subset \Hy^3$ be the totally geodesic plane with limit set $C_\infty$. Given $C\in \calQC_{K_0}$, we let $A_C\subset \wtilde{\Sig}_{C}$ be the set of all points $p$ such that the normal geodesic $\frakn^C_p=\frakn_p^{\wtilde{\Sig}_C}$ intersects $D_\infty$ at a unique point, which we denote by $\pi_C(p)$. Note that $A_C$ is open in $\wtilde{\Sig}_C$ and the function $\pi_C:A_C \ra D_\infty$ is smooth. We also fix an arbitrary sequence of quasicircles $C_m\in \calQC_{K_m}$ with $K_m \to 1$ and $C_m \to C_\infty$. After removing the first elements of this sequence, we can assume the orientation on the surfaces $\wtilde{\Sig}_{C_m}$ is consistent with a fixed orientation on $D_\infty$. From this we can define $\tau_m:A_{C_m}\ra \R$ according to the equation \[\frakn^{C_m}_p(\tau_m(p))=\pi_{C_m}(p).\]
We also write $\pi_n$ for $\pi_{C_n}$.
For the proof of \cref{prop.convergenceintegral} we require some lemmas about the sequence $(\wtilde{\Sig}_{C_m})_m$ of minimal disks converging to $D_\infty=\wtilde{\Sig}_{C_\infty}$.

\begin{lemma}\label{lem.01}
    For any compact $Z\subset \Hy^3$ and $\del>0$ and for all $m$ large enough we have:
    \begin{enumerate}
        \item $\wtilde{\Sig}_{C_m}\cap Z \subset A_{C_m}$;
        \item if $p\in \wtilde{\Sig}_{C_m}\cap Z$ then $d(p,\pi_m(p))\leq \del$; and,
        \item if $p \in \wtilde{\Sig}_{C_m}\cap Z$ and $\ell\subset D_\infty$ is a geodesic containing $\pi_m(p)$, then the angle at $\pi_m(p)$ between $\ell$ and $\frakn_p^{C_m}$ is at least $\pi/2-\del$.
    \end{enumerate}
    In particular, we can find a compact set $B$ so that if $m$ is large enough then
    $\pi_m(\wtilde{\Sig}_{C_m}\cap Z) \subset D_\infty \cap B$.
\end{lemma} 

\begin{proof}
    Fix a point $x\in D_\infty$ and assume $Z$ is contained in the ball of radius $R$ around $x$. For $\del>0$, let $\rho=\rho(\al,R)>0$ be given by \cref{lem.closehyperplanes}. For this $\rho$ there exists $m_1$ such that $m\geq m_1$ implies $C_m\in \calU_{D_\infty,x,\eta',K'}$, where $\eta'=\eta'(\rho,R)$ and $K'=K'(\rho,R)$ are given by \cref{lem.tangentminimalsfc}. That lemma then implies that for $p\in Z\cap \wtilde{\Sig}_{C_m}$, if $D$ is the totally geodesic plane tangent to $\wtilde{\Sig}_{C_m}$ at $p$, then $D\subset U_{D_\infty,x,\rho}$ and $\Lam(D)\in \calU_{D_\infty,x,\rho}$. The conclusion then follows by \cref{lem.closehyperplanes}.
\end{proof}

\begin{lemma}\label{lem.02}
    For any compact $Z\subset \Hy^3$ there exists a compact set $B$ such that for all $m$ large enough we have
    \[D_\infty\cap Z \subset \pi_m(A_{C_m}\cap B).\]
\end{lemma}

\begin{proof}
Assume that $Z$ is the ball of radius $R$ around $x\in D_\infty$. For the sake of contradiction, suppose that, after extracting a subsequence and reindexing, there exists a sequence $(q_m)_m\subset D_\infty \cap Z$ such that $q_m\notin \pi_m(A_{C_m}\cap Z_m)$, with $Z_m$ being the ball of radius $m$ around $x$.  Also, since $K_m \to 1$, for $m$ large enough the surfaces $(\wtilde{\Sig}_{C_m}^{(t)})_t=(\{ \frakn_{p}^{C_m}(t)\colon p\in \wtilde{\Sig}_{C_m}\})_t$ form a foliation of $\Hy^3$ by smooth surfaces, see \cref{rmk.foliation}. Therefore, for large enough $m$ we have $q_m\in \frakn_{p_m}^{C_m}$ for a unique $p_m \in \wtilde{\Sig}_{C_m}$, and by our assumptions we have $d(p_m,q_m) \to \infty$. Also, by \cref{prop.seppi} we have $\wtilde{\Sig}_{C_m}\subset \Hull(C_m)\subset U_{D_\infty,x,\rho_m}$ for a sequence $(\rho_m)_m$ converging to $0$, and so for all $m$ large enough there exists some $r_m\in \wtilde{\Sig}_{C_m}\cap Z$.

Now, for $m$ large enough consider the point $z_m=\frakn_{p_m}^{C_m}(t_m)\in \frakn_{p_m}^{C_m}$ between $p_m$ and $q_m$ and at distance $2R+1$ from $q_m$, and let $D_m$ be the totally geodesic plane orthogonal to $\frakn_{p_m}^{C_m}$ at this point. Note that $|t_m| \to \infty$ and that $D_m$ separates $p_m$ from $Z$, so in particular $p_m$ and $r_m$ belong to different components of $\Hy^3 \bs D_m$. This is our desired contradiction, since for $m$ large enough and $|t|\geq 1$ the region $\bigcup_{|s|\leq |t|}{\wtilde{\Sig}_{C_m}^{(s)}}$ is convex in $\Hy^3$ (see \cref{rmk.foliation}), and hence $D_m$, being the totally geodesic plane tangent to $\wtilde{\Sig}_{C_m}^{(t_m)}$ at $z_m$, does not intersect $\Hull(C_m)$.
\end{proof}

\begin{lemma}\label{lem.jac}
    For any compact set $Z\subset \Hy^3$ we have
    \[\lim_{m \to \infty}{\sup_{p\in Z\cap \wtilde{\Sig}_{C_m}}{||\Jac \pi_m(p)|-1|}}=0.\]
    In particular, for any compact set $Z\subset \Hy^3$ we have
    \[\lim_{m \to \infty}{\sup_{q\in Z\cap D_\infty}{||\Jac \pi_m^{-1}(q)|-1|}}=0.\]
\end{lemma}

\begin{proof}
    It is enough to prove the first equality since the second one follows from it and \cref{lem.02}. By \cref{lem.01} we can always assume that $Z\cap D_\infty \subset A_{C_m}$. We first claim that
    \[\sup_{p\in B \cap \wtilde{\Sig}_{C_m}}{|d\tau_m(p)|} \to 0\]
    as $m$ tends to $\infty$.
    To show this, for $m$ large enough consider the foliation $(\wtilde{\Sig}_{C_m}^{(t)})_t$ of $\Hy^3$ by smooth equidistant surfaces as in \cref{rmk.foliation}. This gives us coordinates $(u,v,t)$ on $\Hy^3$, where $(u,v)$ are coordinates on $\wtilde{\Sig}_{C_m}$ such that $\partial_u,\partial_v$ form an orthonormal frame and $(u,v,t)=\frakn_{(u,v)}^{C_m}(t)$. 
    
    Now let $\gam:(-\ep,\ep) \ra \wtilde{\Sig}_{C_m}$ be a smooth curve with $\gam(0)=p\in B$ and $\dot{\gam}(0)=v$ a unit tangent vector, and define $\om(t):=\pi_m(\gam(t))=(\gam(t),s(t))$, for $s(t)=\tau_m(\gam(t))$. Also, let $\thet(t)$ be the angle at $\pi_m(\gam(t))$ between $\dot{\om}(t)$ and $\partial _t$, and set $\thet_{m,p}(v)=\thet(0)$. In the coordinates $(u,v,t)$, the Riemannian metric on $\Hy^3$ can be described as
    \[G_t=g_t+dt^2,\]
   where $g_t$ is the Riemannian metric on $\wtilde{\Sig}_{C_m}$ described in the coordinates $(u,v)$ by the matrix
   \begin{equation}\label{eq.riemmetric}
       g_t=(\cosh{(t)}\Id+\sinh{(t)}\msf{A}_m)^2,
   \end{equation}
   with $\msf{A}_n$ representing the second fundamental form of $\wtilde{\Sig}_{C_m}$ (see Section 5 in \cite{uhlenbeck} for details). In particular, we can compute
   \begin{align*}
       \dot{s}(t)^2 & =G_{s(t)}(\dot{s}(t)\partial_t,\partial_t)^2 \\  
       & = G_{s(t)}(\dot{\om}(t),\partial_t)^2 \\
       & =\cos^2(\thet(t)) \cdot G_{s(t)}(\dot{\om}(t),\dot{\om}(t))\\
       & =\cos^2(\thet(t))\cdot (\cosh^2{(s(t))}\inn{\dot{\gam}(t)}{\dot{\gam}(t)}+\sinh{(2s(t))}\inn{\dot{\gam}(t)}{\msf{A}_m\dot{\gam}(t)} \\ 
        & \qquad +\sinh^2{(s(t))}\inn{\dot{\gam}(t)}{\msf{A}_m^2\dot{\gam}(t)}), 
   \end{align*}
   and hence 
   \begin{align*}
       |d\tau_m(p)(v)| & =|\dot{s}(0)| \\
       & \leq |\cos (\thet_{m,p}(v))|\cdot (\cosh^2{(\tau_m(p))}+\sinh{(2|\tau_m(p)|)}|\lam_m(p)|\\ & \qquad+\sinh^2{(\tau_m(p))}|\lam_m(p)|^2)^{1/2},
   \end{align*}
   where $\inn{\cdot}{\cdot}$ denotes the metric on $\wtilde\Sig_{C_m}$ and $\pm\lam_m(p)$ are the principal curvatures of $\wtilde{\Sig}_{C_m}$ at $p$. 
   The claim then follows from the convergences \[\|\lam_m\|_{\wtilde{\Sig}_{C_m}}\to 0, \hspace{2mm} |\tau_m|_{B\cap \wtilde{\Sig}_{C_m}}\to 0 \ \text{ and } \ \inf_{p\in B\cap \wtilde{\Sig}_{C_m}}{\left(\inf_{v\in T^1_{p}\wtilde{\Sig}_{C_m}}{\thet_{m,p}(v)}\right)} \to \pi/2, \] which follow from \cref{prop.seppi} ii) and Items (2) and (3) of \cref{lem.01} respectively.  

   Now, for $m$ large enough and $p=(u,v)\in B\cap \wtilde{\Sig}_{C_m}$ we note the identity
   \begin{align*}
       \Jac \pi_m(p)& =\area_{\pi_m(p)}(d\pi_m(p)(\partial_u),d\pi_m(p)(\partial_v))\\ & =\area_{\pi_m(p)}(\partial_u+d\tau_m(\partial_u)\partial_t,\partial_v+d\tau_m(\partial_v)\partial_t)\\
       & =\area_{\pi_m(p)}(\partial_u,\partial_v)+\area_{\pi_m(p)}(d\tau_m(\partial_u)\partial_t,\partial_v)+\area_{\pi_m}(\partial_u,d\tau_m(\partial_v)\partial_t),
   \end{align*}
   where ``$\area$'' denotes the signed area, chosen so that $\area_{(u,v,0)}(\partial_u,\partial_v)=1$. 
   
   From \eqref{eq.riemmetric} we also have
   \begin{align*}
       |G_{\tau_m(p)}(\partial_u,\partial_v)| & =|(\cosh^2{(\tau_m(p))}\inn{\partial_u}{\partial_v}+\sinh{(2\tau_m(p))}\inn{\partial_u}{\msf{A}_m\partial_v} \\ & \qquad+\sinh^2{(\tau_m(p))}\inn{\partial_u}{\msf{A}_m^2\partial_v})|\\
    & \leq |\sinh{(2\tau_m(p))}||\lam_m(p)|+\sinh^2{(\tau_m(p))}|\lam_m(p)|^2,
   \end{align*}
   and
    \begin{align*}
       |G_{\tau_m(p)}(\partial_u,\partial_u)-1| & =|(\cosh^2{(\tau_m(p))}\inn{\partial_u}{\partial_u}+\sinh{(2\tau_m(p))}\inn{\partial_u}{\msf{A}_m\partial_u}\\ & \qquad  +\sinh^2{(\tau_m(p))}\inn{\partial_u}{\msf{A}_m^2\partial_u})-1|\\
    & \leq |\cosh^2(\tau_m(p))-1| +|\sinh{(2\tau_m(p))}||\lam_m(p)| \\ & \qquad +\sinh^2{(\tau_m(p))}|\lam_m(p)|^2.
   \end{align*}
    As a consequence of \cref{prop.seppi} (i) and \cref{lem.01} (2) we deduce that the term $G_{\tau_m(p)}(\partial_u,\partial_v)$ converges to zero uniformly for $p\in B$, and similarly that the norms 
    $G_{\tau_m(p)}(\partial_u,\partial_u)^{1/2}$ and $G_{\tau_m(p)}(\partial_v,\partial_v)^{1/2}$ converge to $1$ uniformly. This gives us the uniform convergence $|\area_{\pi_m(p)}(\partial_u,\partial_v)| \to 1$ for $p\in B$.
    
    Similarly, since $G_{\tau_m(p)}(\partial_t,\partial_u)=G_{\tau_m(p)}(\partial_t,\partial_v)=0$ and $G_{\tau_m(p)}(\partial_t,\partial_t)=1$, by a similar argument and our claim we can deduce that the term $\area(d\tau_m(\partial_u)\partial_t,\partial_v)+\area_{\pi_m}(\partial_u,d\tau_m(\partial_v)\partial_t)$ converges to 0 uniformly for $p\in B$. This implies the uniform convergence $|\Jac \pi_m(p)|\to 1$ and concludes the proof of the lemma.
\end{proof}

\begin{lemma}\label{lem.cont}
    For any compact $Z\subset \Hy^3$ and any $F\in C_c(\Gr_2(\Hy^3))$ we have
    \[\lim_{m \to \infty}{\sup_{q\in Z\cap \wtilde{\Sig}_{C_\infty}}{|F(T_{\pi^{-1}_m(q)}\wtilde{\Sig}_{C_m})-F(T_qD_\infty})|}=0.\]
\end{lemma}

\begin{proof}
    Let $d'$ be any metric on $\Gr_2(\Hy^3)$ inducing its topology. By \cref{lem.01}, for any $\del>0$ and $m$ large enough and any fixed compact $B\subset \Hy^3$ we have \[d'(T_{p}\wtilde{\Sig}_{C_m},T_{\pi_m(p)}D_\infty)\leq \del\] for all $p\in \wtilde{\Sig}_{C_m}\cap B$. The conclusion then follows from the uniform continuity of $F$ and \cref{lem.02}.
\end{proof}

\begin{proof}[Proof of \cref{prop.convergenceintegral}]
    Let $F\in C_c(\Gr_2(\Hy^3))$ be arbitrary, inducing $\hat{F}:\calQC_{K_0} \ra \R$, and let $Z\subset \Hy^3$ be the compact set of all points $p$ with $(p,P)\in \supp(F)$ for some plane $P$ tangent at $p$. By Lemmas \ref{lem.01} and \ref{lem.02} we can find a compact set $B\subset \Hy^3$ satisfying:
    \begin{itemize}
        \item $Z\subset B$;
        \item $\wtilde{\Sig}_{C_m}\cap Z \subset A_{C_m}$ for all $m$ large enough;
        \item $\pi_m(\wtilde{\Sig}_{C_m}\cap Z)\subset D_\infty \cap B$ for all $m$ large enough; and,
        \item $D_\infty\cap B \subset \pi_m(\wtilde{\Sig}_{C_m})$ for all $m$ large enough.
    \end{itemize}    
From this, for all $m$ large enough we have
\begin{align*}
    \hat{F}(C_m)=\int{F}{d\wtilde{\nu}_{C_m}} & =\int_{\wtilde{\Sig}_{C_m}\cap Z}{F(f_{C_m}(p))}{d\wtilde{\mu}_{C_m}(p)}\\ 
    &=\int_{\pi_m(\wtilde{\Sig}_{C_m}\cap Z)}{F(f_{C_m}(\pi_m^{-1}(q)))}{d{\pi_m}_\ast(\wtilde{\mu}_{C_m})(q)}\\
    &=\int_{D_\infty\cap B}{F(f_{C_m}(\pi_m^{-1}(q)))}{|\Jac \pi_m^{-1}(q)|d\wtilde{\mu}_{C_\infty}(q)}, 
\end{align*}
and hence
\small\begin{align*}
    |\hat{F}(C_m)-\hat{F}(C_\infty)|& =\left|\int_{D_\infty\cap B}{(F\circ f_{C_m}\circ \pi_m^{-1})}{|\Jac \pi_m^{-1}|{d\wtilde{\mu}}_{C_\infty}}-\int_{D_\infty\cap B}{(F\circ f_{C_\infty})}{d\wtilde{\mu}_{C_\infty}}\right| \\
    & \leq \int_{D_\infty\cap B}{|(F\circ f_{C_m}\circ \pi_m^{-1})-(F\circ f_{C_\infty})|}{d\wtilde{\mu}_{C_\infty}}\\ 
    & \qquad+\int_{D_\infty\cap B}{|(F\circ f_{C_m}\circ \pi_m^{-1})|||\Jac \pi_m^{-1}|-1|}{d\wtilde{\mu}_{C_\infty}}\\ 
    & \leq \wtilde{\mu}_{C_\infty}(B)\cdot \|(F\circ f_{C_m}\circ \pi_m^{-1})-(F \circ f_{C_\infty})\|_{D_\infty\cap B}\\
    & \qquad+\|F\|_\infty\cdot \||\Jac \pi_m^{-1}|-1\|_{D_\infty\cap B}.
\end{align*}
\normalsize
The last terms tend to 0 as $m$ tends to infinity by Lemmas \ref{lem.jac} and \ref{lem.cont}, which gives us the desired convergence $\hat{F}(C_m)\to \hat{F}(C_\infty)$.
\end{proof}


\section{Approximating convex-cocompact representations}\label{sec.quasifuchsian}

In this section we study the behavior of general convex-cocompact groups of isometries of $\Hy^n$ ($n\geq 2$), and their induced metric structures. When $n=3$, we combine \cref{prop.homeoccpt} below with \cref{thm.main.3manifold} and a theorem of Brooks \cite{brooks} to approximate torsion-free convex-cocompact subgroups of $\PSL(2,\C)$ by cubulations, deducing Propositions \ref{prop.main.quasifuchsian} and \ref{prop.QFintCurr=Teich}.

Let $\G$ be a non-elementary group without torsion, and let $\Isom^+(\Hy^n)$ be the group of orientation-preserving isometries of $\Hy^n$. Recall that a representation $\pi:\G \ra \Isom^+(\Hy^n)$ is \emph{convex-cocompact} if the inclusion $\G \ra \Hy^n$ given by $g \mapsto \pi(g)x$ is a quasi-isometric embedding for some (any $x\in \Hy^n$). The ($n$-dimensional) convex-cocompact space of $\G$ is the quotient $\mathscr{CC}^n_\G$ of the space all the convex-cocompact representations $\pi:\G \ra \Isom^+(\Hy^n)$, where two such representations $\pi,\pi'$ are equivalent if they are conjugate in $\Isom^+(\Hy^n)$. When $n=3$ and $\G$ is a surface group, a convex-cocompact representation $\pi:\G \ra \Isom^+(\Hy^3)=\PSL(2,\C)$ is called \emph{quasiFuchsian}, and $\scrQ\scrF_\G=\mathscr{CC}_\G^3$ is the \emph{quasiFuchsian space} of $\G$. 

If $\scrC\scrC^n_\G$ is non-empty, then $\G$ is hyperbolic and each convex-cocompact representation $\pi$ determines a point $\rho_{\pi}\in \scrD_\G$ as follows: if $x$ is any point of $\Hy^n$, then the metric $d_\pi(g,h):=d_{\Hy^n}(\pi(g)x,\pi(h)x)$ belongs to $\calD_\G$ and the class $\rho_\pi=[d_\pi]$ is independent of the point $x$. Indeed, this gives a well-defined map $\scrC\scrC^n_\G \ra \scrD_\G$ that is injective by \cite[Thm.~1]{burger}. When $\scrC\scrC^n_\G$ is equipped with the quotient topology from the compact-open topology on the space of convex-cocompact representations, the next proposition relates this topology and the topology induced by the inclusion into $\scrD_\G$.

\begin{proposition}\label{prop.homeoccpt}
    For each $n$, the map $\scrC\scrC^n_\G \ra \scrD_\G$ is continuous. Moreover, if $\G$ is not a free group, then this inclusion is a homeomorphism onto a closed subset of $\scrD_\G$.
\end{proposition}

\begin{remark}
    The assumption of $\G$ not being free in \cref{prop.homeoccpt} cannot be dropped. For instance, let $\G=\genby{a,b}$ be a rank-2 free group and for $t>0$ consider the isometries 
    \[f_t(z)=\frac{z+t}{tz+1}, \ \ g_t(z)=\frac{z+it}{-itz+1}\]
    on $\Hy^2$, when considered with the Poincar\'e disk model. Then the representation $\pi_t:\G \ra \Isom^+(\Hy^2)$ that maps $a$ to $f_t$ and $b$ to $g_t$ is convex-cocompact for $t$ large enough. Moreover, as $t$ tends to infinity, the metric structures $\rho_{\pi_t}\in \scrD_\G$ converge to the metric structure $\rho_S$ induced by the word metric on $\G$ with respect to the generating set $S=\{a^{\pm},b^{\pm}\}$. The image of the translation length function $\ell_S$ is contained in a discrete subgroup of $\R$, which is not the case for the translation lengths of Fuchsian representations \cite{dalbo}. This implies that the metric structure $\rho_S$ does not belong to the image of $\scrC\scrC_\G^n$ for any $n$. Analogous examples can be constructed for higher-rank free groups in higher-dimensional hyperbolic spaces. 
\end{remark}

For the proof of \cref{prop.homeoccpt} we require the following lemma.

\begin{lemma}\label{lem.unifcobounded}
Let $(\pi_m)_m$ be a sequence of convex-cocompact representations of $\G$ into $\Isom^+(\Hy^n)$ and assume that they converge in the compact-open topology to the convex-cocompact representation $\pi_\infty:\G\ra \Isom^+(\Hy^n)$. For each $m$, let $\calC_m\subset \Hy^n$ be the convex hull of the limit set $\Lam_m$ of $\pi_m(\G)$. Then there exist closed subsets $B_m\subset \calC_m$ such that $\pi_m(\G)\cdot B_m=\calC_m$ and satisfying  
$$\sup_{m}{\diam(B_m)}<\infty.$$
\end{lemma}

\begin{proof}
    Let $\Lam_\infty\subset \partial \Hy^n$ be the limit set of $\pi_\infty(\G)$ and $\calC_\infty$ its convex hull. Fix a point $x\in \calC_\infty$ and let $B_\infty$ be the closed Dirichlet domain for the action of $\pi_\infty(\G)$ on $\calC_\infty$ centered at $x$. That is,
    \[B_\infty=\{z\in \calC_\infty : d_{\Hy^n}(x,z)\leq d_{\Hy^n}(x,\pi_\infty(g)z) \text{ for all }g\in \G\}.\]
   Note that $B_\infty$ is compact and convex since $\pi_\infty$ is convex-cocompact. 

   By \cite[Thm.~7.1]{mcmullen}, the sequence of limit sets $(\Lam_m)_m$ converges to $\Lam_\infty$ in the Hausdorff topology of compact subsets of $\partial \Hy^n$, and hence $\calC_m$ converges to $\calC_\infty$ in the Hausdorff topology when restricted to any compact subset of $\Hy^n$ by \cite[Thm.~1.4]{bowditch.convex}. 
   
   We consider a sequence $(x_m)_m$ converging to $x$ and such that $x_m\in \calC_m$ for each $m$, and let $B_m\subset \calC_m$ be the closed Dirichlet domain for the action of $\pi_m(\G)$ on $\calC_m$ centered at $x_m$. The sets $B_m$ are also convex and compact, and satisfy $\pi_m(\G)\cdot B_m=\calC_m$.

   We claim that for any $\ep>0$ and for any $m$ large enough, the set $B_m$ is contained in the $\ep$-neighborhood $N_\ep(B_\infty)$ of $B_\infty$ in $\Hy^n$. This assertion easily implies the conclusion of the lemma. Suppose for the sake of contradiction that the claim does not hold and let $\ep>0$ be such that $B_m \bs N_\ep(B_\infty)$ is non-empty for infinitely many $m$. Convexity and compactness of $B_\infty$ and $B_m$ implies that, up to taking a subsequence and reindexing, there exists a sequence $(y_m)_m$ with each $y_m$ in $B_m \bs N_\ep(B_\infty)$ and converging to a point $y$. The convergence $C_m \to C_\infty$ implies $y\in \calC_\infty\bs B_\infty$, and hence there exists $g\in \G$ such that $d_{\Hy^n}(x,\pi_\infty(g)y)<d_{\Hy^n}(x,y)$.
   The convergences $\pi_m \to \pi_\infty$, $x_m \to x$ and $y_m\to y$ then imply
   \[d_{\Hy^n}(x_m,\pi_m(g)y_m)<d_{\Hy^n}(x_m,y_m)\]
   for all $m$ large enough, contradicting that $y_m\in B_m$.
\end{proof}

We will also need a lemma about precompactness of sequences of metric structures induced by pseudo metrics with uniformly bounded geometry. If $\G$ is an arbitrary non-elementary hyperbolic group, a pseudo metric $d\in \calD_\G$ is $\al$-\emph{roughly geodesic} ($\al\geq 0$) if for all $g,h\in \G$ we can find a sequence $g=g_0,\dots,g_k=h$ in $\G$ satisfying 
\[|d (g_i,g_j)-|i-j||\leq \al\]
for all $0\leq i,j\leq k$. For $\del,\al\geq 0$, we let $\scrD_\G^{\del,\al}\subset \scrD_\G$ be the set of all the metric structures $[d]$, where $d$ is $\del$-hyperbolic, $\al$-roughly geodesic and has exponential growth rate 1. The subspace $\scrD_\G^{\del,\al}$ proper \cite[Thm.~1.5]{oregon-reyes.ms}, and is invariant under the natural isometric action of $\Out(\G)$ on $\scrD_\G$. Moreover, if $\G$ is torsion-free then the action if $\Out(\G)$ on $\scrD_\G^{\del,\al}$ is proper and cocompact \cite[Thm.~1.6 \& Thm.~1.7]{oregon-reyes.ms}.

For two pseudo metrics $d,d'\in \calD_\G$ their \emph{dilation} is given by the quantity
\[\Dil(d,d'):=\sup_{g}{\frac{\ell_d[g]}{\ell_{d'}[g]}},\]
where the supremum is taken over all the non-torsion elements $g\in \G$. Note that $\Del([d],[d'])$ equals $\log(\Dil(d,d')\cdot \Dil(d',d))$ for all $d,d'\in \calD_\G$ \cite[Prop.~3.5]{oregon-reyes.ms}. For two metric structures $\rho,\rho'\in \scrD_\G$ we also define $\Dil(\rho,\rho')=\Dil(d,d')$, where $\rho=[d],\rho'=[d']$ and $d,d'$ have exponential growth rate 1.

\begin{lemma}\label{lem.precompactness}
    Let $\G$ be a torsion-free, non-elementary hyperbolic group and let $(d_m)_m$ be a sequence of pseudo metrics in $\calD_\G$. Suppose that all these pseudo metrics have exponential growth rate 1, and that there exist $\del,\al$ such that each $d_m$ is $\del$-hyperbolic and $\al$-roughly geodesic. 
    If the sequence $(\Dil(d_m,d_0))_m$ is bounded for some $d_0\in \calD_\G$, then the sequence $([d_m])_m$ of metric structures is precompact in $\scrD_\G$. 
\end{lemma}

\begin{proof}
    For each $m$ we let $\rho_m=[d_m]$, so that $(\rho_m)_m$ is a sequence in  $\scrD_\G^{\del,\al}$. We also define $\rho_0=[d_0]$, and without loss of generality assume that $d_0$ has exponential growth rate 1. Since $\scrD_\G^{\del,\al}$ is proper \cite[Thm.~1.6]{oregon-reyes.ms}, it is enough to show that the sequence $(\rho_m)_m$ is bounded in $\scrD_\G$. Our assumption of $\G$ being torsion-free and \cite[Thm.~1.7]{oregon-reyes.ms} give us a compact subset $K\subset \scrD_\G^{\del,\al}$ such that $\Out(\G) \cdot K=\scrD_\G^{\del,\al}$. We set $R:=\sup_{\rho\in K}{\Dil(\rho_0,\rho)}$, which is bounded since $\rho \mapsto \Dil(\rho_0,\rho)$ is continuous. 
    
    We can find a sequence $(\phi_m)_m$ in $\Out(\G)$ such that $\phi_m(\rho_m)\in K$ for all $m$. Then we have 
    \begin{align*}
        \Dil(\phi_m^{-1}(\rho_0),\rho_0)& = \Dil(\rho_0,\phi_m(\rho_0)) \\
        & \leq \Dil(\rho_0,\phi_m(\rho_m))\Dil(\phi_m(\rho_m),\phi_m(\rho_0))\\
        & \leq R \Dil(\rho_m,\rho_0)\\
        & =R\Dil(d_m,d_0)
    \end{align*}
    for all $m$. Our assumptions on $(d_m)_m$ then imply that the sequence $(\Dil(\phi_m^{-1}(\rho_0),\rho_0))_m$ is bounded. But the map $\Out(\G) \ra \R$ that sends $\phi$ to $\Dil(\phi(\rho_0),\rho_0)$ is proper, as can be proven by the same argument as in the proof of \cite[Thm.~1.7]{oregon-reyes.ms}. We conclude that the set $\calF=\{\phi^{-1}_m\}_m$ is finite, and hence the sequence $(\rho_m)_m$ is contained in the bounded set $\calF \cdot K$. 
\end{proof}

\begin{proof}[Proof of \cref{prop.homeoccpt}]
We fix a base point $x_0\in \Hy^n$ and a finite symmetric generating set $S\subset \G$. We also let $\ell_S$ denote the stable translation length associated to the word metric $d_S$. 

We first prove that the inclusion $\scrC\scrC_\G^n \ra \scrD_\G$ is continuous, so suppose $([\pi_m])_m\subset \scrC\scrC^n_\G$ is a sequence converging to $[\pi_\infty]$. Up to conjugation, we can assume that the representations $\pi_m$ converge to $\pi_\infty$ in the compact-open topology of $\Hy^n$. In particular, the distances
$d_m(g,h):=d_{\Hy^n}(\pi_m(g)x_0,\pi_m(h)x_0)$ on $\G$ pointwise converge to the metric $d_\infty(g,h):=d_{\Hy^n}(\pi_\infty(g)x_0,\pi_\infty(h)x_0)$.  

If $\ell_m:\G \ra \R$ denotes the stable translation length function of $d_m$, we claim that $\ell_m$ pointwise converges to $\ell_\infty$, the stable translation length of $d_\infty$. Indeed, this can be done as in the proof of Claim 2 in \cite[Prop.~5.6]{oregon-reyes.ms}, using the facts that the metrics $d_m,d_\infty$ are $\log{2}$-hyperbolic. Moreover, since $\pi_m(\G)$ converges to $\pi_\infty(\G)$ algebraically, by \cite[Thm.~7.1]{mcmullen} and \cite[Thm.~1.1]{bishop-jones} we have that the exponential growth rates $v_m$ of the metrics $d_m$ converge to the exponential growth rate $v_\infty$ of $d_\infty$. We let $\hat d_m=v_m d_m$ and $\hat d_\infty=v_\infty d_\infty$, with translation length functions $\hat \ell_m$ and $\hat \ell_\infty$ respectively.  

By \cref{lem.unifcobounded} and \cite[Lem.~4.6]{BCGS} it can be proven that the metrics $\hat d_m$ are $\al$-roughly geodesic for some $\al$ independent of $m$. We also have that the metrics $\hat d_m$ have exponential growth rate 1 and are $\hat\del$-hyperbolic for some $\hat\del$ independent of $m$. Moreover, for all $m$ and all $g\in \G$ we have
\begin{equation*}
    d_m(g,o)\leq \left(\max_{s\in S}\{d_m(s,o)\} \right)d_S(g,o),
\end{equation*}
and hence 
\begin{equation*}
    \hat\ell_m[g]\leq  \left(v_m\max_{s\in S}\{d_m(s,o)\} \right)\ell_S[g].
\end{equation*}
But the sequence $(v_m\max_{s\in S}\{d_m(s,o)\})_m$ is bounded because $d_m$ pointwise converges to $d_\infty$ and  
$\sup_m v_m \leq n-1$ for all $m$ (see for example \cite[Thm.~2.1]{paulin}). Therefore, the sequence $(\Dil(\hat{d}_m,d_S))_m$ is bounded, and \cref{lem.precompactness} implies that the sequence of metric structures $(\rho_{\pi_m})_m=([\hat d_m])_m$ is precompact in $(\scrD_\G,\Del)$. 

We consider an arbitrary subsequence $([\hat d_{m_k}])_k$ with $m_k$ tending to infinity, and assume this subsequence converges to $\rho=[d]\in \scrD_\G$. If $d$ has exponential growth rate 1, then the translation lengths $\hat \ell_{m_k}$ pointwise converge to $\ell_d$ by \cite[Prop.~3.5 \& Lem.~3.6]{oregon-reyes.ms}. This implies that $\ell_d$ equals $\hat\ell_\infty$ and hence $\rho=\rho_{\pi_\infty}:=[\hat d_\infty]$. Since this holds for any convergent subsequence of $(\rho_{\pi_m})_m$, we deduce the convergence $\rho_{\pi_m} \to \rho_{\pi_\infty}$ in $(\scrD_\G,\Del)$ and the continuity of the inclusion $\scrC\scrC_\G^n \ra \scrD_\G$.

Suppose now that $\G$ is not a free group. To prove that the inclusion $\scrC\scrC_\G^n \ra \scrD_\G$ is a homeomorphism into its image and that the image is closed in $\scrD_\G$, let $[\pi_m]_m$ be a sequence in $\scrC\scrC^n_\G$ such that $\rho_m:=\rho_{\pi_m}$ converges to $\rho_\infty\in \scrD_\G$. We claim that $\rho_\infty=\rho_{\pi_\infty}$ for some $[\pi_\infty]\in \scrC\scrC^n_\G$ and that $[\pi_m]$ converges to $[\pi_\infty]$ in $\scrC\scrC^n_\G$. To do this, for each $m$ let $\pi_m:\G \ra \Isom^+(\Hy^n)$ be a convex-cocompact representation inducing $\rho_m$, with exponential growth rate $v_m$ and translation length function $\ell_m$. By assumption the sequence $(\rho_m)_m$ is bounded, so by \cite[Lem.~3.6]{oregon-reyes.ms} there exists $L>0$ such that 
\begin{equation}\label{eq.boundedmetricstructure}
    L^{-1}\ell_S[g]\leq v_m\ell_m[g]\leq L\ell_S[g]
\end{equation}
for all $g\in \G$. In addition, by \cite[Thm.~1.4]{breuillard-fujiwara} there exists a constant $C>0$ such that for each $m$ we can find a point $x_m\in \Hy^n$ satisfying
    \begin{equation}\label{eq.ineqreps}
        \max_{s\in S}{\{d_{\Hy^n}(\pi_m(s)x_m,x_m)\}}\leq \frac{1}{2}\max_{s_1,s_2\in S}{\{\ell_{m}[s_1s_2]\}}+C.
    \end{equation}
We let $(f_m)_m\subset \Isom^+(\Hy^n)$ be a sequence of isometries of $\Hy^n$ such that $f_mx_0=x_m$ for all $m$, and define $\ov{\pi}_m(g)=f_m\circ \pi_m(g)\circ f_m^{-1}$ for $g\in \G$. We also define $$d_m(g,h):=d_{\Hy^m}(\pi_m(g)x_m,\pi_m(g)x_m) \ \text{ and } \ \hat d_m(g,h) :=v_md_m(g,h)$$ for $g,h\in \G$.

We claim that $\inf_m{v_m}$ is positive. Otherwise, after taking a subsequence and reindexing we can assume $v_m \to 0$. Then from \eqref{eq.boundedmetricstructure} and \eqref{eq.ineqreps} we have
\begin{align*}
    \hat d_m(g,h) & \leq (v_m \max_{s\in S}\{d_{\Hy^n}(\pi_m(s)x_m,x_m)\} )\cdot d_S(g,h) \\
    & \leq \left(\frac{1}{2}\max_{s_1,s_2\in S}\{v_m\ell_{m}[s_1s_2]\}+Cv_m \right)\cdot d_S(g,h) \\
    & \leq \left(\frac{L}{2}\max_{s_1,s_2\in S}\{\ell_{S}[s_1s_2]\}+\sup_m{v_m}C\right)\cdot d_S(g,h)\\
    & \leq (L+C\sup_m v_m )\cdot d_S(g,h)
\end{align*}
for all $g,h\in \G$. In particular, the metrics $\hat d_m$ are pointwise bounded, so up to taking a new subsequence and reindexing we can assume that $\hat d_m$  converges to the pseudo metric $\hat d_\infty$ on $\G$. Since each $\hat d_m$ is $v_m \log{2}$-hyperbolic and $v_m \to 0$, we have that $\hat d_\infty$ is $0$-hyperbolic. Moreover, as in the proof of Claim 2 in \cite[Prop.~5.6]{oregon-reyes.ms}, we can prove that $v_m\ell_m$ pointwise converges to $\hat \ell_\infty$, the translation length function of $\hat d_\infty$. This implies that $[\hat d_\infty]=\rho_\infty$, and it can be proven that $\rho_\infty$ is a metric structure represented by a geometric action of $\G$ on a metric (simplicial) tree. This is impossible since we assume that $\G$ is torsion-free but not free. 



Proven our claim, we set $c:=\inf_m{v_m}>0$. From \eqref{eq.boundedmetricstructure} and \eqref{eq.ineqreps} we also deduce
    \begin{equation*}
        \max_{s\in S}\{d_{\Hy^n}(\ov{\pi}_m(s)x_0,x_0)\}\leq \frac{1}{2}\max_{s_1,s_2\in S}\{\ell_{m}[s_1s_2]\}+C \leq c^{-1}L+C,
    \end{equation*}
    and the sequence $(\ov{\pi}_m)_m$ is bounded.  Therefore, it has a subsequence converging algebraically to the representation $\ov{\pi}_\infty:\G \ra \Isom^+(\Hy^n)$. 
    
    In addition, the metrics 
    $d_m$ pointwise converge to the metric $$d_\infty(g,h):=d_{\Hy^n}(\ov{\pi}_\infty(g)x_0,\ov\pi_\infty(h)x_0),$$ and as before we can prove that the translation length functions $\ell_{m}$ pointwise converge to the translation length function $\ell_{\infty}$ induced by the action of $\G$ on $\Hy^n$ via $\ov{\pi}_\infty$.

    On the other hand, the functions $v_m\ell_{\ov\pi_m}=v_m\ell_m$ pointwise converge to the translation length $\hat\ell_\infty$ of a metric $\hat d_\infty\in \calD_\G$ representing $\rho_\infty$ and with exponential growth rate 1. We also know that $\sup_m v_m \leq n-1$ (see for example \cite[Thm.~2.1]{paulin}), and hence $\ell_\infty$ and $\hat \ell_\infty$ are homothetic, implying that $\ov\pi_\infty$ is convex-cocompact and that $\rho_\infty$ is induced by the class $[\ov\pi_\infty]\in \scrC\scrC^n_\G$. The exact same argument also implies that any subsequence of $(\pi_m)_m$ has a subsequence converging (up to taking conjugates) to a representation with stable translation length function homothetic to that of $\hat\ell_\infty$. Since $\scrC\scrC_\G^n \ra \scrD_\G$ is injective, we conclude that the sequence $[\pi_m]_m$ converges to $[\ov\pi_\infty]$, as desired. 
\end{proof}

From the proof of \cref{prop.homeoccpt} we also deduce the following when $\G$ is a free group.

\begin{corollary}\label{coro.contfreegroup}
    If $\G$ is a free group and $([\pi_m])_m$ is a sequence in $\scrC\scrC_\G^n$ such that $\rho_m:=\rho_{\pi_m}$ converges to $\rho_\infty$ in $\scrD_\G$, then either:
    \begin{itemize}
        \item $[\pi_m]$ converges to $[\pi_\infty]$ in $\scrC\scrC_\G^n$ and $\rho_\infty=\rho_{\pi_\infty}$; or,
        \item $\rho_\infty$ is a point in the Outer space $\scrC\scrV_\G$.
    \end{itemize}
\end{corollary}


\begin{proof}[Proof of \cref{prop.main.quasifuchsian}]
Let $\pi:\G\ra \PSL(2,\C)$ be a convex-cocompact representation with $\G$ torsion-free.  
By \cite[Thm.~1]{brooks} there exists a sequence $(\pi_m)_m$ of convex-cocompact representations algebraically converging to $\pi$ and satisfying the following. For each $m$ there exists an embedding $\G \hookrightarrow \hat \G_m$ into the fundamental group $\hat \G_m$ of a closed hyperbolic 3-manifold and a representative of the lattice action $\hat \pi_m:\hat \G_m \ra \PSL(2,\C)$ that coincides with $\pi_m$ when restricted to $\G$. 

By \cref{thm.main.3manifold}, for each $m$ let $\rho_{\wtilde \calX_m}\in \scrD_{\hat \G_m}$ be a metric structure induced by a proper and cocompact action of $\hat\G_m$ on a $\CAT(0)$ cube complex $\wtilde\calX_m$ and such that $\Del_m(\rho_{\hat\pi_m},\rho_{\wtilde\calX_m})\leq 1/m$ for $\Del_m$ being the distance on $\scrD_{\hat\G_m}$.

Since each $\pi_m$ is convex-cocompact, $\G$ embeds as a quasiconvex subgroup of $\hat\G_m$, and hence by \cite[Prop.~13.7]{haglund-wise.special} it has a \emph{convex-core} for its action on $\wtilde\calX_m$. That is, there exists a $\G$-invariant convex subcomplex $\wtilde\calY_m\subset \wtilde\calX_m$ such that the action of $\G$ on $\wtilde\calY_m$ is cocompact. We let $\rho_m:=\rho_{\wtilde\calY_m}\in \scrD_\G$ be the metric structure induced by the action on this subcomplex, for which it is clear that $\Del(\rho_{\pi_m},\rho_{\hat\calY_m})\leq 1/m$ for $\Del$ being the distance on $\scrD_{\G}$. \cref{prop.homeoccpt} implies that $\rho_{\pi_m}$ converges to $\rho_\pi$ in $\scrD_\G$, and hence $\rho_{\wtilde\calY_m}$ also converges to $\rho_\pi$, as desired.
\end{proof}

We end this section with the proof of \cref{prop.QFintCurr=Teich}, which is just an adaptation of the proof of \cite[Thm.~A]{fricker-furman}.

\begin{proof}[Proof of \cref{prop.QFintCurr=Teich}] First we prove Item (1). Let $\rho\in \scrQ\scrF_\G\subset \scrD_\G$ be induced by the quasiFuchsian representation $\pi:\G \ra \PSL(2,\C)$, and let $\ell_{\pi}$ be the stable translation length of $\G$ given by the corresponding action on $\Hy^3$. If $\pi$ is not Fuchsian, then the proof of \cite[Thm.~A]{fricker-furman} gives us a pair of non-trivial group elements $a,b\in \G$ and a constant $c>0$ such that
   \begin{equation}\label{eq.1pfqf}
       \ell_{\pi}[a]+\ell_{\pi}[b]-\ell_{\pi}[ab]\geq c.
   \end{equation}
Moreover, the pairs of fixed points $(a^{-\infty},a^{\infty}), (b^{-\infty},b^{\infty})$ are \emph{unlinked and aligned} in $\partial \G$, in the nomenclature of \cite[Sec.~2]{fricker-furman}. Now, for a fixed discrete and faithful action of $\G$ on $\Hy^2$, the proof of \cite[Thm.~2.1]{fricker-furman} gives us a point $x$ lying on the axis $A_{(ab)}\subset \Hy^2$ for $ab$, and so that $bx$ and $abx$ also lie on this geodesic with $bx$ between $x$ and $abx$. If $\rho_{[\eta]}\in \scrD_\G$ is induced by the filling geodesic current $\eta$, we consider the $\G$-invariant pseudo metric $d_\eta$ on $\Hy^2$ introduced in \cite[Sec.~4]{BIPP}. One feature of this pseudo metric is that 
\[d_{\eta}(x,z)=d_{\eta}(x,y)+d_{\eta}(y,z)\]
whenever $x,y,z$ lie on a geodesic in $\Hy^2$ with $y$ between $x$ and $z$ \cite[Prop.~4.1]{BIPP}. Another feature of $d_\eta$ is that it is dual to $\eta$ in the sense that 
\[i(\eta,\beta)=\ell_\eta(\beta) \]
for all $\beta\in \Curr(\G)$, where $i$ is the Bonahon's intersection number and $\ell_\eta$ is the stable translation length for the action of $\G$ on $(\Hy^2,d_\eta)$. 

In particular we have $\ell_{\eta}[ab]=d_\eta(abx,x)$ and we deduce
\begin{equation}\label{eq.2pfqf}
    \ell_{\eta}[ab] = d_{\eta}(abx, x) = d_{\eta}(a(bx), bx) + d_{\eta}(bx, x) \geq \ell_{\eta}[a] + \ell_{\eta}[b].
\end{equation}
Since $\eta$ is filling, we deduce that $ab$ is also non-trivial, and combining \eqref{eq.1pfqf} and \eqref{eq.2pfqf} we conclude
\begin{equation*}
    \Del(\rho,\rho_{[\eta]})\geq \log\left(\frac{(\ell_\pi[a]+\ell_\pi[b])}{(\ell_\mu[a]+\ell_\mu[b])}\cdot \frac{\ell_\mu[ab]}{\ell_{\pi}[ab]} \right)\geq \log\left(1+\frac{c}{\ell_{\pi}[ab]}\right)=:\lam_{\rho}>0.
\end{equation*}

Item (2) then follows from Item(1) and \cref{prop.homeoccpt}.
\end{proof}


\section{Proof of applications}\label{sec.applications}

In this section we prove Corollaries \ref{coro.genericquotient} and \ref{cor.nogrowthgap}. We first prove \cref{coro.genericquotient}, for which we rely on the results of Futer and Wise in \cite{futer-wise}.
\begin{proof}[Proof of \cref{coro.genericquotient}]
In the three cases considered, there exists a sequence $(\rho_m)_m=(\rho_{\wtilde\calX_m})_m$ of metric structures in $\scrD_\G$ converging to $\rho=\rho_X$ that are induced by geometric actions of $\G$ on the $\CAT(0)$ cube complexes $\wtilde\calX_m$, and such that there exist $\G$-equivariant $\lam_m$-quasi-isometries from $X$ to $\wtilde\calX_m$ with $\lam_m$ converging to 1.

Suppose first that $X$ is a negatively curved Riemannian surface. By \cref{coro.fillingcurrents=cubulations} we can choose the cubulations $\wtilde\calX_m$ so that all their hyperplane stabilizers are cyclic, so they have exponential growth rate zero for the action of $\G$ on $X$. Since $c<v_X/41$, we have
\[c<\min\left\{\frac{v_X}{20\lam_m},\frac{v_X}{40\lam_m+1}\right\}\]
for all $m$ large enough. Then the assumptions of \cite[Thm.~1.3]{futer-wise} are satisfied and with overwhelming probability as $\ell \to \infty$, for any set  
$[g_1],\dots, [g_k]$ of conjugacy classes of $\G$ with translation lengths $\ell_X[g_i]\leq \ell$ the group $\ov\G = \G/\genby{\genby{g_1,\dots, g_k}}$ is hyperbolic and cubulated.

If $X=\Hy^3$, then from the proof of \cref{thm.main.3manifold} we see that we can choose each $\wtilde\calX_m$ to have $K_m$-quasiFuchsian surface groups as hyperplane stabilizers with $K_m$ tending to 1. This implies that if $a_m$ is the maximal exponential growth rate with respect to $X$ of a hyperplane stabilizer of $\wtilde\calX_m$, then $a_m$ converges to 1 as $m$ tends to infinity. Therefore, if $c<2/41=v_X/41$ then
\[c<\min\left\{\frac{v_X-a_m}{20\lam_m},\frac{v_X}{40\lam_m+1}\right\}\]
for all $m$ large enough. In this case the assumptions of \cite[Thm.~1.3]{futer-wise} are also satisfied and we conclude as in the first case. 

The case $X=\Hy^n$ and $\G$ arithmetic of simplest type follows by the exact same argument. We only need to note from the proof of \cref{thm.main.arithmetic} that we can choose each $\wtilde X_m$ to have hyperplane stabilizers acting properly and cocompactly on codimension-1 totally geodesic hypersurfaces of $X$, so they have exponential growth rate $n-1=v_X-1$ with respect to the action on $X$. 
\end{proof}

We continue with the proof of \cref{cor.nogrowthgap}, whose proof uses the main result of Li and Wise in \cite{li-wise}.

\begin{proof}[Proof of \cref{cor.nogrowthgap}]
Let $\G$ act on $X$ as in the statement of the corollary and let $\rho=\rho_X\in \scrD_\G$ be the induced metric structure. Then there exists a sequence $(\rho_m)_m\subset \scrD_\G$ of metric structures converging to $\rho$ such that each $\rho_m$ is induced by the geometric action of $\G$ on the $\CAT(0)$ cube complex $\wtilde\calX_m$. 

If $H<\G$ is any subgroup, from \eqref{eq.EGR} we deduce that 
\[ e^{-\Del(\rho_m,\rho)}\frac{v_X(H)}{v_X}\leq \frac{v_{\wtilde \calX_m}(H)}{v_{\wtilde\calX_m}}\leq e^{\Del(\rho_m,\rho)}\frac{v_X(H)}{v_X}\]
for all $m$. Also, by Agol's theorem \cite[Thm.~1.1]{agol.haken} and \cite[Main~Theorem~1.5]{li-wise}, for each $m$ we can find a quasiconvex subgroup $H_m<\G$ of infinite index such that $v_{\wtilde \calX_m}(H_m) \geq  e^{-\Del(\rho_m,\rho)}v_{\wtilde\calX_m}$. It follows that $v_X(H_m)$ converges to $v_X$, so that the action of $\G$ on $X$ has no growth-gap.
\end{proof}


\section{Questions and future directions}\label{sec.questions}

In this section we pose some questions related to our work. 

Recall that $\rho\in \scrD_\G$ is approximable by cubulations if it lies in the closure of $\scrD_\G^{cub}$. One might ask whether every closed hyperbolic manifold is approximable by cubulations. However, even the following is open:

\begin{question}
    Is every closed hyperbolic 4-manifold cubulable?
\end{question}

\begin{question}
    Let $\G$ be a uniform lattice in $\Hy^n$ and suppose that every round hypersphere in $\bbS^{n-1}$ is the limit of a sequence of limit sets of quasiconvex codimension-1 subgroups of $\G$. Is $\rho_{\Hy^n}\in \scrD_\G$ approximable by cubulations?
\end{question}

\begin{question}
    Suppose $\frakg$ is a negatively curved Riemannian metric on the closed 3-manifold $M$, not necessarily of constant curvature. Is the point $\rho_{\frakg}\in \scrD_{\pi_1(M)}$ induced by the action on $(\wtilde M,\wtilde\frakg)$ approximable by cubulations? 

    There are some chances for this to be true if $\frakg$ is close enough to the constant curvature metric. See for instance \cite{lowe}.
\end{question}

\begin{question}
    If $M$ is a closed hyperbolic manifold obtained via Gromov--Piatetski-Shapiro surgery \cite{GPS}, is $M$ approximable by cubulations? Note that $\pi_1(M)$ embeds as a (non-totally geodesic) quasiconvex subgroup in an arithmetic lattice of simplest type \cite[Prop.~9.1]{BHW}.
\end{question}

By the work of Parry \cite{parry} we know that the stable translation length functions for isometric actions on $\R$-trees can be characterized axiomatically. It would be desirable to have similar result for actions on $\CAT(0)$ cube complexes.

\begin{question}
    Is there a set of axioms that characterizes length functions $\ell:\G \ra \R$ of the form $\ell=\ell_{\wtilde\calX}$ for a cubical action of $\G$ on the $\CAT(0)$ cube complex $\wtilde{\calX}$?
\end{question}

The next questions ask about the flexibility of approximating by cubulations among points in $\scrD_\G$.

\begin{question}
    Let $\G$ be a cubulable hyperbolic group with a geometric action on a space with measured walls. Is this action approximable by cubulations?
\end{question}

\begin{question}
Let $d_1,d_2\in \calD_\G$ be such that both $[d_1]$ and $[d_2]$ are approximable by cubulations. It is not hard to show that $[td_1+d_2]$ is also approximable by cubulations for any $t>0$. If $t>0$ is large enough, then $td_1-d_2$ also determines a metric structure $[td_1-d_2]\in \scrD_\G$ \cite[Prop.~4.1]{cantrell-reyes.manhattan}. Is $[td_1-d_2]$ approximable by cubulations?
\end{question}

\begin{question}
    Let $\G$ be cubulable hyperbolic and let $\rho\in \scrD_\G$ be approximable by cubulations. Is $\rho$ approximable by cubulations with a single $\G$-orbit of hyperplanes? 
\end{question}

The next question is a rephrasing of \cite[Question.~7.9]{futer-wise} in the language of metric structures.
\begin{question}
    Let $\G$ be a hyperbolic cubulable group. Is every point in $\scrD_\G$ approximable by cubulations?
\end{question}

Generalizing the hyperbolic case, we say that an isometric action of a (non-necessarily hyperbolic) group $\G$ on the space $X$ is approximable by cubulations if for any $\lam>1$ there exists a $\G$-equivariant $\lam$-quasi-isometric embedding from $X$ into a $\CAT(0)$ cube complex on which $\G$ acts.

\begin{question}
    Is the action of a finite volume hyperbolic 3-manifold group $\G$ on $\Hy^3$ approximable by cubulations? Since actions on $\CAT(0)$ cube complexes do not admit parabolic elements \cite{haglund}, probably we need to consider the action of $\G$ on $\Hy^3\bs \{\text{horoballs}\}$. It would be instructive to consider first the case of cusped hyperbolic surfaces.
\end{question}

\begin{question}
    Suppose $\Gamma$ is not virtually free and let $\rho\in \scrD_\G$ be approximable by cubulations. If $\rho$ is represented by a geometric action on a $\CAT(-1)$ space (or more generally by a \emph{strongly hyperbolic} metric in $\calD_\G$), must the dimension of the cubulations of any approximating sequence go to infinity? From \cite{GMM} we know that the marked length spectrum of such action is not contained in a discrete subgroup of $\R$. 
\end{question}

The next question is related to the inclusion $\scrC\scrC_\G^n \ra \scrD_\G$ for $\G$ a free group. In this case, from \cref{coro.contfreegroup} we know that the closure of $\scrC\scrC^n_\G$ equals $\scrC\scrC_\G^n\cup \scrP_n$ for $\scrP_n$ a subset of the Outer space $\scrC\scrV_\G$ of $\G$. 

\begin{question}
    Do we have $\scrP_n=\scrC\scrV_\G$ for some $n$? If that is the case, can we compute such a minimal $n$ in terms of the rank of $\G$? Can we always take $n=2$?
\end{question}

\medskip



\printbibliography














\noindent\small{Nic Brody}\\
\noindent\small{UC Santa Cruz, USA, 95060}\\
\small{\textit{Email address}: \texttt{nic@ucsc.edu}}\\

\noindent\small{Eduardo Reyes}\\
\noindent\small{Max Planck Institute for Mathematics, Bonn, Germany, 53111}\\
\small{\textit{Email address}: \texttt{eoregon@mpim-bonn.mpg.de}}\\

\end{document}